\documentclass[11pt]{article}
\usepackage[margin=1in]{geometry}
\usepackage[T1]{fontenc}
\usepackage{lmodern}
\usepackage{microtype}
\usepackage{amsmath,amssymb,amsthm,mathtools}
\usepackage{enumitem}
\usepackage{graphicx}
\usepackage{booktabs}
\usepackage{float}
\usepackage{natbib}
\usepackage{xcolor}
\usepackage[hidelinks]{hyperref}
\newtheorem{theorem}{Theorem}[section]
\newtheorem{proposition}[theorem]{Proposition}
\newtheorem{lemma}[theorem]{Lemma}
\newtheorem{corollary}[theorem]{Corollary}
\newtheorem{conjecture}{Conjecture}
\theoremstyle{definition}
\newtheorem{definition}[theorem]{Definition}
\theoremstyle{remark}
\newtheorem{remark}[theorem]{Remark}
\newcommand{\R}{\mathbb R}
\newcommand{\N}{\mathbb N}
\newcommand{\1}{\mathbf 1}
\newcommand{\Phib}{\overline\Phi}
\newcommand{\FDR}{\operatorname{FDR}}
\newcommand{\FDP}{\operatorname{FDP}}
\newcommand{\BH}{\operatorname{BH}}
\newcommand{\diag}{\operatorname{diag}}
\newcommand{\sech}{\operatorname{sech}}
\newcommand{\dd}{\,\mathrm d}

\begin{document}

\title{How Much Can Gaussian Dependence Inflate the
Benjamini-Hochberg Procedure's FDR?}
\author{Lihua Lei\\
Graduate School of Business, Stanford University\\
\texttt{lihualei@stanford.edu}}
\date{\today}
\maketitle

\begin{abstract}
We study the worst-case false discovery rate (FDR) of the Benjamini-Hochberg
procedure for both one- and two-sided Gaussian tests with general correlation matrices.
For each setting and given target FDR level $q$,  we construct
a family of finite Gaussian models whose FDR divided by $q$
diverges as $q\downarrow0$. This disproves any universal multiplicative FDR
bound.  Moreover, we establish new FDR lower bounds that are independent of the number of hypotheses.
For two-sided tests, the worst-case FDR over the number of hypotheses, mean vector, and
correlation matrix is at least an explicit ${\ell_{=}}(q)>q$:
\[
 \ell_{=}(q)=\frac{q\sqrt{\log(1/q)}}{2\sqrt{\pi}}+c_\ell q+o(q),
 \qquad c_\ell=0.6492828\ldots.
\]
For one-sided tests, we show a lower bound $\ell_{\le}(q)>q$: 
\[
 \ell_{\le}(q)=\frac{q\sqrt{\log(1/q)}}{\sqrt\pi}
 +\frac q2+o(q).
\]
Finally, we prove the leading constants are sharp for a
class of correlation matrices for both one- and two-sided tests.
\end{abstract}

\section{Introduction}

We observe a Gaussian vector $T=(T_1,\ldots,T_{N})$ with mean
$\theta$ and covariance matrix $\Sigma$, where
$\Sigma_{ii}=1$ for $i=1,\ldots,N$.  We consider two testing problems.  For
two-sided tests, $H_i:\theta_i=0$ and
\[
 P_i=p(|T_i|),\qquad p(x)=2\Phib(x)=2\{1-\Phi(x)\}.
\]
For one-sided tests, $H_i:\theta_i\leq0$ and
\[
 P_i=p_+(T_i),\qquad p_+(x)=\Phib(x)=1-\Phi(x).
\]
Here $\Phi$ and $\phi=\Phi'$ are the standard normal CDF and density.
The normal hazard is
\begin{equation}
\lambda(x)=\frac{\phi(x)}{\Phib(x)}.
\end{equation}
We use the conventions $p^{-1}(0)=p_+^{-1}(0)=\infty$.
The Benjamini-Hochberg procedure \citep{BenjaminiHochberg1995} at level $q$
orders the $p$-values as
\[
 P_{(1)}\leq\cdots\leq P_{(N)}
\]
and rejects the hypotheses corresponding to
\[
 P_{(1)},\ldots,P_{(R)},\quad
 R=\max\{{1\leq k\leq N}:P_{(k)}\leq qk/N\},
\]
with $R=0$ if the set is empty.  If $V$ is the number of rejected true
nulls, then the false discovery proportion (FDP) and false discovery rate
(FDR) are
\[
 \FDP=\frac{V}{R\vee1},
 \quad
 \FDR=\mathbb E(\FDP).
\]

A long-standing conjecture for correlated two-sided Gaussian
$z$-tests is that BH
controls the FDR at its nominal level $q$, regardless of the correlation
matrix
\citep{ReinerBenaim2007,benjamini2010discovering,roux2018inference,Sarkar2023,SarkarZhang2025}.
\citet[p.~407]{benjamini2010discovering} summarized the evidence
as follows:
\begin{quote}
``The modification to general dependence is often not needed: convincing
simutheoretical evidence indicates that the same holds for two-sided
$z$-tests with any correlation structure.''
\end{quote}
Benjamini then noted that a complete theoretical proof was still missing.
These works stated or supported this conjecture.  However, recently,
\citet{Dobriban2026} found a counterexample where
$\FDR>0.0104$ when $q=0.01$.

In this paper, we ask a more refined question in each testing problem: how
large can the FDR be as a function of $q$?  For two-sided tests, we seek a
lower bound ${\ell_{=}}(q)$ such that
\begin{equation}
 \sup_{\substack{
   N\in\N,\ \theta\in\mathbb R^N,\
   \Sigma\in\mathbb S_+^N\\
   \Sigma_{ii}=1,\ 1\leq i\leq N}}
 \FDR_{\theta,\Sigma}(\BH_q)\geq \ell_{=}(q).
 \label{eq:worst-case-fdr}
\end{equation}
Here $\mathbb S_+^N$ denotes the cone of $N\times N$
symmetric positive semidefinite matrices.
A relaxed folklore conjecture is that the worst-case FDR in
\eqref{eq:worst-case-fdr} is bounded above by $Cq$ for some universal
constant $C>0$.  However, Theorem~\ref{thm:main} and
Proposition~\ref{prop:ell-expansion} together show that the worst-case FDR is
at least a quantity satisfying, as $q\downarrow0$,
\[
 \ell_{=}(q)
 =\frac{q\sqrt{\log(1/q)}}{2\sqrt{\pi}}+{c_\ell}q+o(q),
 \qquad c_\ell=0.6492828\ldots.
\]
In particular, the ratio of the supremum in
\eqref{eq:worst-case-fdr} to $q$ diverges as $q\downarrow0$.

Moreover, we derive an exact formula for ${\ell_{=}}(q)$.  We prove that
${\ell_{=}}(q)>q$ for every $q\in(0,1)$.  Instead of seeking a counterexample
for a given nominal level $q$, Theorem~\ref{thm:main} and
Proposition~\ref{prop:ell-strict} together disprove the conjecture that BH
controls the FDR for correlated two-sided Gaussian tests.
Figure~\ref{fig:ellq} illustrates the resulting two-sided FDR
inflation.
\begin{figure}[H]
\centering
\includegraphics[width=\linewidth]{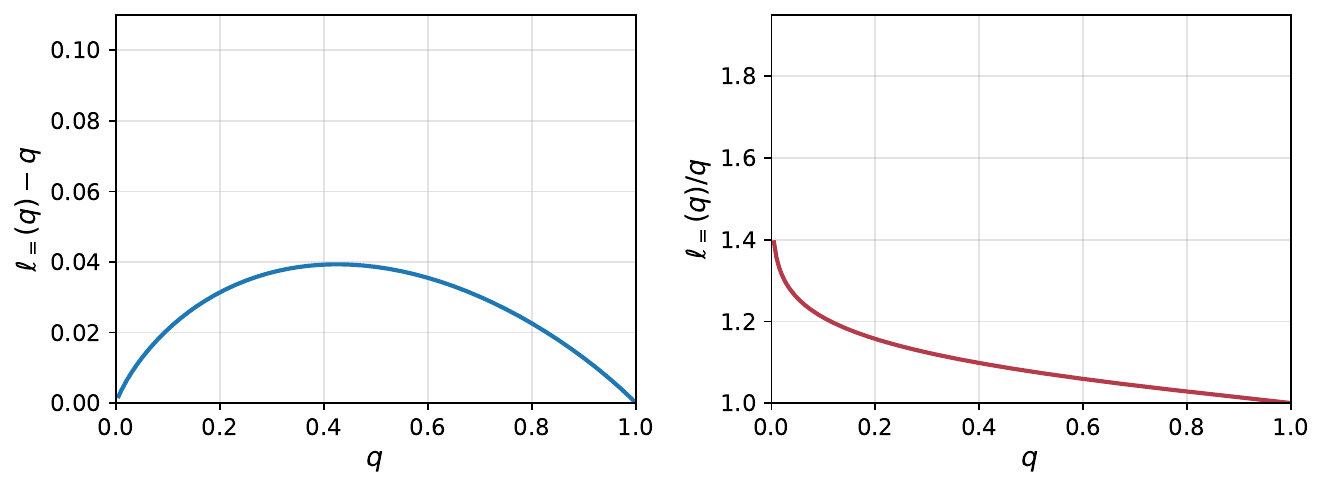}
\caption{Numerical evaluation of ${\ell_{=}}(q)-q$ and ${\ell_{=}}(q)/q$ over the grid $q=0.005k$, $k=1,\ldots,199$.}
\label{fig:ellq}
\end{figure}

\begin{table}[H]
\centering

\begin{tabular}{@{}ccccc@{}}
\toprule
$q$&0.01&0.05&0.10&0.20\\
\midrule
${\ell_{=}}(q)-q$&0.0036&0.0129&0.0209&0.0314\\
${\ell_{=}}(q)/q$&1.3553&1.2571&1.2094&1.1570\\
\bottomrule
\end{tabular}
\caption{The absolute inflation ${\ell_{=}}(q)-q$ and relative
inflation ${\ell_{=}}(q)/q$ at commonly used nominal levels.}
\label{tab:common-q}
\end{table}

For one-sided Gaussian tests, the classical negatively correlated examples
of \citet{hochberg1995extensions} and \citet{samuel1996simes}
show that the Simes inequality, and hence BH under the global
null, can be anticonservative.  \citet{chi2025multiple} prove useful upper bounds under
several notions of negative dependence, but the adversarial construction used
in their sharpness comparison, due to \citet{Su2018}, is not jointly
Gaussian.  Our one-sided construction instead retains positively
correlated nulls and reverses the loading of the non-nulls.
Theorem~\ref{thm:one-lower} and
Proposition~\ref{prop:one-strict} give, for every $q\in(0,1)$,
\[
 \ell_{\le}(q)
 =\frac{q\sqrt{\log(1/q)}}{\sqrt\pi}+\frac q2
  +\Phib\!\left(\sqrt{2\log(1/q)}\right)>q,
\]
and Proposition~\ref{prop:one-expansion} gives
\[
 \ell_{\le}(q)
 =\frac{q\sqrt{\log(1/q)}}{\sqrt\pi}+\frac q2+o(q).
\]
This is a larger leading constant than in the two-sided construction and
again implies that the worst-case FDR ratio diverges.
Figure~\ref{fig:one-ellq} illustrates the resulting one-sided FDR
inflation.
\begin{figure}[H]
\centering
\includegraphics[width=\linewidth]{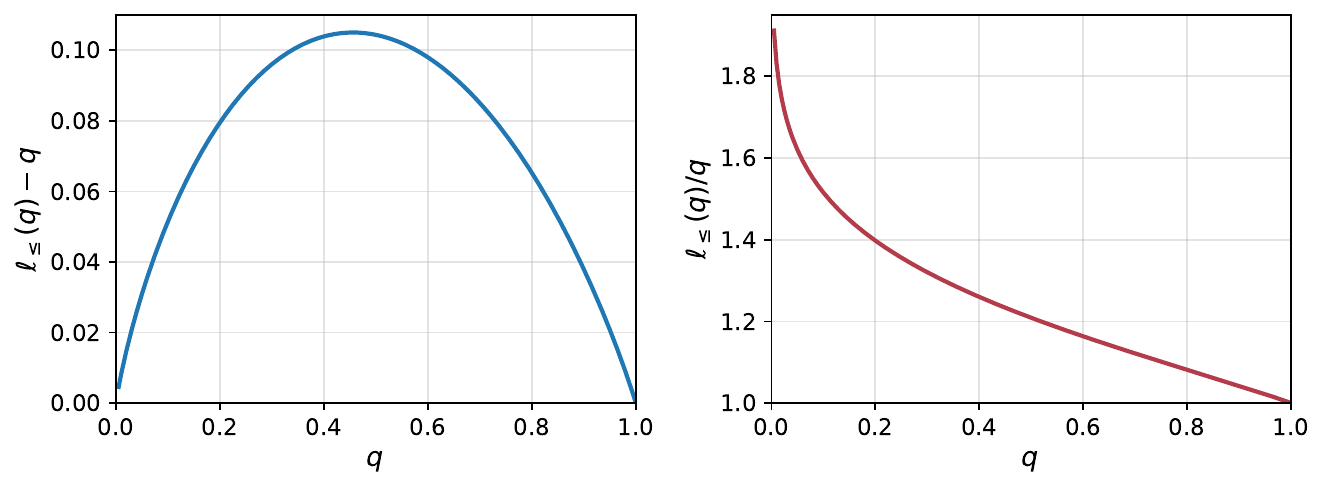}
\caption{Numerical evaluation of $\ell_{\le}(q)-q$ and
$\ell_{\le}(q)/q$ over the grid $q=0.005k$, $k=1,\ldots,199$.}
\label{fig:one-ellq}
\end{figure}

\begin{table}[H]
\centering
\caption{The absolute inflation $\ell_{\le}(q)-q$ and relative
inflation $\ell_{\le}(q)/q$ at commonly used nominal levels.}
\label{tab:one-ell}
\begin{tabular}{ccccc}
\toprule
$q$ & $0.01$ & $0.05$ & $0.10$ & $0.20$\\
\midrule
$\ell_{\le}(q)-q$ & $0.00831$ & $0.03101$ & $0.05155$ & $0.07955$\\
$\ell_{\le}(q)/q$ & $1.8311$ & $1.6203$ & $1.5155$ & $1.3977$\\
\bottomrule
\end{tabular}
\end{table}

The dependence structures differ sharply between the two constructions.
Every covariance in the two-sided construction is positive, but the map
$T_i\mapsto2\Phib(|T_i|)$ is not monotone and
need not make the full vector of two-sided $p$-values satisfy
positive regression dependence on a subset (PRDS) with respect to the true
nulls.  For
one-sided Gaussian $p$-values, by contrast, nonnegative correlations between
every true null and every other hypotheses imply that $p$-values are PRDS
and hence FDR control
\citep{BenjaminiYekutieli2001}.  Negative dependence among the test statistics is
therefore essential for the one-sided lower bound. In particular, our example
has positive correlations between null pairs and negative correlations
between a pair of null and non-null.  Thus,
{the full vector of one-sided $p$-values in our construction is
not PRDS; Section~\ref{sec:one-sided-lower} verifies this directly}.

The lower-bound rate $q\sqrt{\log(1/q)}$ should be compared
with the FDR-linking theorem of \citet{Su2018}.
Positive regression dependence within nulls (PRDN) relaxes PRDS by imposing the relevant positive
dependence assumption only on the null $p$-values.  Under PRDN, the
FDR-linking argument gives an $O(q\log(1/q))$ bound.
Section~\ref{sec:upper-bound} proves sharper Gaussian bounds.  For two-sided
one-common-factor models,
\[
 \sup_{\substack{m\in\N,\ \theta\in\R^m,\\
                    a\in[-1,1]^m}}
 \FDR_{\theta,a}(\BH_q)
 =\left\{\frac1{2\sqrt\pi}+o(1)\right\}
 q\sqrt{\log(1/q)},
\]
so the lower construction attains the sharp leading constant in this
class.  This bound improves on the general PRDN linking rate.  Moreover,
as explained in
Remark~\ref{rem:prdn-pvalues}, the two-sided null $p$-values in this class
satisfy the PRDN property by the
absolute-value Gaussian multivariate total positivity of order
two (MTP$_2$) criterion of \citet{KarlinRinott1981}.
For one-sided one-common-factor models,
Theorem~\ref{thm:one-upper} and Corollary~\ref{cor:one-sharp} give
\[
 \sup_{\substack{m\in\N,\ \theta\in\R^m,\\
                    a\in[-1,1]^m}}
 \FDR_{\theta,a}(\BH_q)
 =\left\{\frac1{\sqrt\pi}+o(1)\right\}
 q\sqrt{\log(1/q)}.
\]
This matches the leading constant in $\ell_{\le}(q)$.  Both bounds are
uniform over the number of hypotheses, means, and scalar loadings.

We conclude the section with a new conjecture that neither the author nor
GPT 5.6 Sol was able to disprove after significant efforts.

\begin{conjecture}
The unrestricted Gaussian worst-case FDR satisfies
\[
\lim_{q\downarrow0}
\frac{
\displaystyle
\sup_{\substack{N\in\N,\ \theta\in\R^N,\ \Sigma\in\mathbb S_+^N\\
\Sigma_{ii}=1,\ 1\le i\le N}}
\FDR_{\theta,\Sigma}(\BH_q)}
{q\sqrt{\log(1/q)}}
=
\begin{cases}
\dfrac{1}{2\sqrt{\pi}},&\text{for two-sided tests},\\[4pt]
\dfrac{1}{\sqrt{\pi}},&\text{for one-sided tests}.
\end{cases}
\]
\end{conjecture}

\section{\texorpdfstring{Lower bound for two-sided tests}{Lower bound for two-sided tests}}

\subsection{Main result}

\begin{equation}
M(y)=\sup_{a\ge0}e^{-a^2/2}\cosh(ay).
\end{equation}
Appendix~\ref{app:envelope} identifies the likelihood-ratio envelope: $M=1$ on $[0,1]$, and $M$ is continuous, strictly increasing, and unbounded on $(1,\infty)$.  Hence there is a unique $y_q>1$ such that
\begin{equation}
qM(y_q)=1.
\label{eq:yq}
\end{equation}
Set
\begin{equation}
\ell_{=}(q)=q\Phi(1)+q\int_1^{y_q}M(y)\phi(y)\dd y+\Phib(y_q).
\label{eq:ell}
\end{equation}

The next theorem shows that the lower bound can be approached
arbitrarily closely by finite-dimensional Gaussian models with positive
definite correlation matrices.
\begin{theorem}
\label{thm:main}
For $T\sim N(\theta,\Sigma)$, let
$\FDR_{\theta,\Sigma}(\BH_q)$ denote the FDR of $\BH_q$ applied to the
two-sided $p$-values $p(|T_1|),\ldots,p(|T_N|)$.
For every $q\in(0,1)$ and $\delta>0$, there are $N\in\N$,
$\theta\in\R^N$, and a positive definite correlation matrix
$\Sigma$ such that
\begin{equation}
\FDR_{\theta,\Sigma}(\BH_q)>{\ell_{=}}(q)-\delta.
\label{eq:main}
\end{equation}
Consequently,
\begin{equation}
\sup_{\substack{N\in\N,\ \theta\in\R^N,\\
\Sigma\succ0,\ \diag(\Sigma)=\mathbf 1_N}}
\FDR_{\theta,\Sigma}(\BH_q)
\ge \ell_{=}(q).
\label{eq:sup}
\end{equation}
\end{theorem}

The next proposition gives the small-$q$ expansion of the lower bound.
\begin{proposition}
\label{prop:ell-expansion}
As $q\downarrow0$,
\[
\ell_{=}(q)
=
\frac{q\sqrt{\log(1/q)}}{2\sqrt{\pi}}
+{c_\ell}q+o(q),
\]
where
\[
\begin{aligned}
c_\ell
&=
\Phi(1)-\frac{1}{2\sqrt{2\pi}}
+\int_1^\infty
\left\{
M(y)\phi(y)-\frac{1}{2\sqrt{2\pi}}
\right\}\dd y\\
&=0.6492828\ldots.
\end{aligned}
\]
\end{proposition}

Proposition~\ref{prop:ell-strict} further shows that
${\ell_{=}}(q)>q$ for every $q\in(0,1)$.
Together with Theorem~\ref{thm:main}, this implies that BH fails
to control the FDR at every nontrivial nominal level in the correlated
two-sided Gaussian setting.
\begin{proposition}
\label{prop:ell-strict}
For every $q\in(0,1)$,
\[
\ell_{=}(q)>q.
\]
\end{proposition}

Figure~\ref{fig:ellq} in the Introduction displays the numerical
behavior of the two-sided lower bound.
Table~\ref{tab:common-q} gives the absolute inflation
${\ell_{=}}(q)-q$ and relative inflation ${\ell_{=}}(q)/q$ for commonly
used nominal levels.

\subsection{Roadmap}
\label{sec:roadmap}

Fix $q\in(0,1)$ and a design law $\nu_q$.  Let
$\Omega_N=((\Theta_1,B_1),\ldots,(\Theta_N,B_N))$, where the pairs are
i.i.d. draws from $\nu_q$, independently of
$Z,\varepsilon_1,\ldots,\varepsilon_N
{\stackrel{\text{i.i.d.}}{\sim}}N(0,1)$,
and set
\[
T_i=\Theta_i+B_iZ+\sqrt{1-B_i^2}\,\varepsilon_i,
\quad i=1,\ldots,N.
\]
Conditional on a realization of $\Omega_N$, this is a finite deterministic
Gaussian design.  Since BH is a step-up procedure, its rejection set is
$\{i:|T_i|\ge\widehat C_N\}$.  We will show in
Subsection~\ref{sec:finite-transfer} that, conditional on $Z=z$,
$\widehat C_N\to C(z)$ and $\FDP\to D(z)$, where convergence is defined
in an appropriate sense.  Consequently,
\[
\mathbb E_{\Omega_N}\!\left[\FDR_{\Omega_N}(\BH_q)\right]
=
\mathbb E_{\Omega_N,Z,\varepsilon}[\FDP]
\longrightarrow
\mathbb E[D(Z)].
\]
Thus, for any $\delta>0$ and all sufficiently large $N$, at least one
deterministic realization $\omega_N$ of $\Omega_N$ satisfies
\[
\FDR_{\omega_N}(\BH_q)\ge\mathbb E[D(Z)]-\delta.
\]

We next construct a family of laws $\nu_{q,r}$.  Write $D_r$ for the
conditional FDP limit associated with $\nu_{q,r}$.  For every
$z\notin\{-1,-y_q\}$, we prove that $D_r(z)\to D_*(z)$ as
$r\downarrow0$, where
\[
D_*(z)=
\begin{cases}
q,&z\ge-1,\\
qM(-z),&-y_q<z<-1,\\
1,&z\le-y_q.
\end{cases}
\]
Because $0\le D_r,D_*\le1$, dominated convergence then gives
\[
\mathbb E[D_r(Z)]\longrightarrow\mathbb E[D_*(Z)]=\ell_{=}(q).
\]
Then we can prove the lower bound \eqref{eq:sup} in
Theorem~\ref{thm:main} by taking $N$ to approach infinity first and $r$
to approach zero. To convert the supremum to a bound in \eqref{eq:main},
we just need to The argument therefore takes the iterated limit
$\lim_{r\downarrow0}\lim_{N\to\infty}$: first $N\to\infty$ for each
fixed $r$, and then $r\downarrow0$.  To obtain a single finite example,
we choose $r$ sufficiently small and then $N$ sufficiently large.  A
small independent Gaussian perturbation finally makes its covariance matrix
positive definite.

\subsection{Population quantities}
\label{sec:generic}

Put $u_q=p^{-1}(q)$.  Let $\nu_q$ be a probability measure on $\R\times[-1,1]$.  Write
\begin{equation}
\pi_0=\nu_q(\{0\}\times[-1,1]),
\label{eq:pi0-generic}
\end{equation}
and assume $\pi_0>0$.

For $(\theta,b)\in\R\times[-1,1]$, $z\in\R$, and $s\in(0,1]$, let $u_s=p^{-1}(s)$ and define
\[
F_{\theta,b,z}(s)
=
\mathbb P\!\left\{
p(|T_i|)\le s
\,\middle|\,
\Theta_i=\theta,\ B_i=b,\ Z=z
\right\}.
\]
If $|b|<1$, then
\begin{equation}
F_{\theta,b,z}(s)
=
\Phib\!\left(\frac{u_s-\theta-bz}{\sqrt{1-b^2}}\right)
+
\Phib\!\left(\frac{u_s+\theta+bz}{\sqrt{1-b^2}}\right).
\label{eq:Ftheta-formula}
\end{equation}
If $|b|=1$, then
\begin{equation}
F_{\theta,b,z}(s)
=
\1\{p(|\theta+bz|)\le s\}.
\label{eq:Ftheta-degenerate}
\end{equation}
Set $F_{\theta,b,z}(0)=0$.

Conditional on $Z=z$, the CDFs of the mixture $p$-value
and the null $p$-value are
\begin{align}
R_z(s)
&=
\int F_{\theta,b,z}(s)\,\nu_q(\dd\theta,\dd b),
\label{eq:Rz-generic}\\
R_{0,z}(s)
&=
\frac1{\pi_0}
\int_{\{0\}\times[-1,1]}
F_{0,b,z}(s)\,\nu_q(\dd\theta,\dd b).
\end{align}
For $c\ge u_q$, define the normalized quantities
\begin{align}
\bar R_z(c)
&=
\frac{qR_z\{p(c)\}}{p(c)},
\\
\bar R_{0,z}(c)
&=
\frac{q\pi_0R_{0,z}\{p(c)\}}{p(c)}.
\end{align}
Then the limiting cutoff can be written as
\begin{equation}
C(z)
=
\inf\{c\ge u_q:\bar R_z(c)\ge1\}.
\end{equation}
Here and below, $\inf\varnothing=\infty$.  The limiting conditional
FDP is
defined by
\begin{equation}
D(z)=
\begin{cases}
\bar R_{0,z}\{C(z)\},&C(z)<\infty,\\
0,&C(z)=\infty.
\end{cases}
\label{eq:D-generic}
\end{equation}

We call a finite cutoff a regular first crossing when it satisfies the following requirements.
\begin{definition}
\label{def:regular}
A finite cutoff $C(z)>u_q$ is regular if
\begin{align}
\sup_{u_q\le c\le C(z)-\rho}\bar R_z(c)&<1
&&\text{for every }0<\rho<C(z)-u_q,
\label{eq:regular-left}\\
\sup_{C(z)<c<C(z)+\rho}\bar R_z(c)&>1
&&\text{for every }\rho>0,
\label{eq:regular-right}
\end{align}
and $c\mapsto\bar R_{0,z}(c)$ is continuous at $C(z)$.
\end{definition}

The following proposition reduces regularity to a local monotonicity
check.
\begin{proposition}
\label{prop:local-regularity}
Let $C(z)>u_q$ be finite.  Suppose there exist
$u_q<a<C(z)<b$ such that the map
$c\mapsto\bar R_z(c)$ is differentiable on $(a,b)$ and
satisfies
\begin{equation}
 \bar R_z'(c)>0,
 \qquad a<c<b.
 \label{eq:local-regularity-positive-derivative}
\end{equation}
Suppose also that $c\mapsto\bar R_{0,z}(c)$ is continuous at $C(z)$.
Then $C(z)$ is regular.
\end{proposition}

\begin{proof}
The map $c\mapsto R_z\{p(c)\}$ is a survival function and hence is
upper semicontinuous.  Since $p$ is positive and continuous,
$c\mapsto\bar R_z(c)$ is also upper semicontinuous.  By the definition
of $C(z)$, $\bar R_z(c)<1$ for every $c<C(z)$.  Therefore, for every
$0<\rho<C(z)-u_q$, upper semicontinuity and compactness give
\[
 \sup_{u_q\le c\le C(z)-\rho}\bar R_z(c)<1,
\]
which proves \eqref{eq:regular-left}.

By the definition of $C(z)$, there is a sequence
$c_n\geq C(z)$ decreasing to $C(z)$ such that
$\bar R_z(c_n)\geq1$.  Continuity on $(a,b)$ gives
$\bar R_z\{C(z)\}\geq1$.  On the other hand, for any sequence
$d_n<C(z)$ increasing to $C(z)$, the first-crossing definition gives
$\bar R_z(d_n)<1$, so continuity gives
$\bar R_z\{C(z)\}\leq1$.  Hence
$\bar R_z\{C(z)\}=1$.  The mean value theorem and
\eqref{eq:local-regularity-positive-derivative} imply strict monotonicity
on $(a,b)$, which yields
\[
 \bar R_z(c)>1\quad\text{for }C(z)<c<b.
\]
This gives \eqref{eq:regular-right}.  Continuity of the null contribution at
$C(z)$ completes Definition~\ref{def:regular}.
\end{proof}

\subsection{Limiting cutoff and FDR}
\label{sec:finite-transfer}

Let
\begin{equation}
\Omega_N=((\Theta_1,B_1),\ldots,(\Theta_N,B_N)),
\quad
(\Theta_i,B_i){\stackrel{\text{i.i.d.}}{\sim}}\nu_q.
\end{equation}
Since the distribution of $(T_1,\ldots,T_N)$ is fully determined by $\Omega_N$, we write $\FDR_{\theta,\Sigma}(\BH_q)$ as $\FDR_{\Omega_N}(\BH_q)$.

Define $P_i=p(|T_i|)$.

\begin{align}
\widehat R_N(s)
&=
\frac1N\sum_{i=1}^N\1\{P_i\le s\},
\\
\widehat R_{0,N}(s)
&=
\frac1N\sum_{i=1}^N\1\{\Theta_i=0,\ P_i\le s\}.
\end{align}

The following lemma gives a uniform conditional empirical-CDF bound.
For a function $f$ on $[0,1]$, write
$\|f\|_{\infty,[0,1]}=\sup_{0\le s\le1}|f(s)|$.
\begin{lemma}
\label{lem:DKW}
For every $z\in\R$ and $t>0$,
\begin{equation}
\mathbb P\!\left(
\max\left\{
\|\widehat R_N-R_z\|_{\infty,[0,1]},
\|\widehat R_{0,N}-\pi_0R_{0,z}\|_{\infty,[0,1]}
\right\}>t
\,\middle|\,Z=z
\right)
\le4e^{-2Nt^2}.
\end{equation}
\end{lemma}

\begin{proof}
Conditional on $Z=z$, the $P_i$ are i.i.d. with CDF $R_z$.  Also define
\[
W_i=P_i\1\{\Theta_i=0\}+2\1\{\Theta_i\ne0\}.
\]
For $0\le s\le1$,
\[
\1\{W_i\le s\}
=
\1\{\Theta_i=0,\ P_i\le s\},
\quad
\mathbb P(W_i\le s\mid Z=z)=\pi_0R_{0,z}(s).
\]
Apply the Dvoretzky-Kiefer-Wolfowitz inequality twice and take a union bound.
\end{proof}

The next lemma gives an exact characterization of the BH cutoff and FDP through the empirical CDFs.
\begin{lemma}
\label{lem:exact-BH}
Let $R_N$ be the number of BH rejections and put
\begin{equation}
\widehat\tau_N=\frac{qR_N}{N}.
\end{equation}
If $R_N\ge1$, then
\begin{equation}
\widehat R_N(\widehat\tau_N)=\frac{R_N}{N},
\quad
\FDP
=
\frac{q\widehat R_{0,N}(\widehat\tau_N)}{\widehat\tau_N}.
\label{eq:exact-FDP}
\end{equation}
With $\widehat C_N=p^{-1}(\widehat\tau_N)$,
\begin{equation}
\widehat C_N
=
\inf\left\{
 c\ge u_q:
 \frac{q\widehat R_N\{p(c)\}}{p(c)}\ge1
\right\}.
\label{eq:exact-C}
\end{equation}
Then the BH rejection set is $\{i:|T_i|\ge\widehat C_N\}$.
If $R_N=0$, the set in \eqref{eq:exact-C} is empty and $\widehat C_N=\infty$.
\end{lemma}

\begin{proof}
Let $P_{(1)}\le\cdots\le P_{(N)}$.  If more than $R_N$ values were at most $qR_N/N$, then
\[
P_{(R_N+1)}\le qR_N/N<q(R_N+1)/N,
\]
contradicting the maximality of $R_N$.  Hence $N\widehat R_N(\widehat\tau_N)=R_N$, and \eqref{eq:exact-FDP} follows.  Since $p$ is decreasing, the BH step-up definition is equivalent to \eqref{eq:exact-C}.
\end{proof}

The next theorem establishes the limiting cutoff conditional
on the common factor $Z$.
\begin{theorem}
\label{thm:population-transfer}
Assume that $C(z)$ is regular for almost every $z$.  Then, for every
$z$ at which $C(z)$ is regular, conditional on $Z=z$,
\begin{equation}
\widehat C_N\longrightarrow C(z),
\quad
\FDP\longrightarrow D(z)
=\bar R_{0,z}\{C(z)\}
=\frac{q\pi_0R_{0,z}[p\{C(z)\}]}{p\{C(z)\}}
\label{eq:conditional-transfer}
\end{equation}
in probability.  Consequently,
\begin{equation}
\lim_{N\to\infty}
\mathbb E_{\Omega_N}
\bigl[\FDR_{\Omega_N}(\BH_q)\bigr]
=
\mathbb E[D(Z)],
\label{eq:expected-transfer}
\end{equation}
where
\begin{equation}
\FDR_{\Omega_N}(\BH_q)
=
\mathbb E_{Z,\varepsilon}\!\left[
\FDP\mid\Omega_N
\right].
\label{eq:FDRomega}
\end{equation}
\end{theorem}

\begin{proof}
Fix a regular $z$ and $\rho,\tau>0$.  Since $D(z)=\bar R_{0,z}\{C(z)\}$ and $c\mapsto\bar R_{0,z}(c)$ is continuous at $C(z)$, choose
\[
0<\rho(z)<\min\{\rho,C(z)-u_q\}
\]
so that
\begin{equation}
\sup_{C(z)-\rho(z)\le c\le C(z)+\rho(z)}
|\bar R_{0,z}(c)-D(z)|\le\tau.
\label{eq:crossing-bracket}
\end{equation}
Set $A(z)=C(z)-\rho(z)$.  Equation~\eqref{eq:regular-left} gives
$\sup_{u_q\le c\le A(z)}\bar R_z(c)<1$, and
\eqref{eq:regular-right} permits a choice of
$B(z)\in(C(z),C(z)+\rho(z))$ with $\bar R_z\{B(z)\}>1$.
Set
\[
\gamma(z)
=
\frac12\min\left\{
1-\sup_{u_q\le c\le A(z)}\bar R_z(c),\quad
\bar R_z\{B(z)\}-1
\right\}>0.
\]
Put
\[
E_N(z)=
\max\left\{
\|\widehat R_N-R_z\|_{\infty,[0,1]},
\|\widehat R_{0,N}-\pi_0R_{0,z}\|_{\infty,[0,1]}
\right\}.
\]
Lemma~\ref{lem:DKW} gives $E_N(z)\to0$ in probability conditional on
$Z=z$.  For every $u_q\le c\le B(z)$,
\begin{equation}
\begin{aligned}
\left|
\frac{q\widehat R_N\{p(c)\}}{p(c)}-\bar R_z(c)
\right|
&\le \frac{qE_N(z)}{p\{B(z)\}},\\
\left|
\frac{q\widehat R_{0,N}\{p(c)\}}{p(c)}-\bar R_{0,z}(c)
\right|
&\le \frac{qE_N(z)}{p\{B(z)\}}.
\end{aligned}
\label{eq:empirical-uniform-bound}
\end{equation}
By the empirical curve we mean
\[
c\longmapsto
\frac{q\widehat R_N\{p(c)\}}{p(c)}.
\]
On the event $qE_N(z)/p\{B(z)\}<\gamma(z)$, the definition of
$\gamma(z)$ and \eqref{eq:empirical-uniform-bound} give
\[
\begin{aligned}
\sup_{u_q\le c\le A(z)}
\frac{q\widehat R_N\{p(c)\}}{p(c)}
&\le
\sup_{u_q\le c\le A(z)}\bar R_z(c)
+\frac{qE_N(z)}{p\{B(z)\}}\\
&<1-2\gamma(z)+\gamma(z)<1,
\end{aligned}
\]
whereas
\[
\frac{q\widehat R_N\{p(B(z))\}}{p(B(z))}
\ge
\bar R_z\{B(z)\}
-\frac{qE_N(z)}{p\{B(z)\}}
>1+2\gamma(z)-\gamma(z)>1.
\]
Lemma~\ref{lem:exact-BH} and \eqref{eq:crossing-bracket} therefore give
\[
A(z)<\widehat C_N\le B(z),
\qquad
|\FDP-D(z)|
\le
\tau+\frac{qE_N(z)}{p\{B(z)\}}.
\]
Because $\rho,\tau>0$ were arbitrary, this proves
\eqref{eq:conditional-transfer} at every regular $z$.

For completeness, the quantities in the expectation can be taken to be
measurable without making a measurable selection of $B(z)$.  The maps
$(z,c)\mapsto\bar R_z(c)$ and
$(z,c)\mapsto\bar R_{0,z}(c)$ are jointly Borel measurable, and the first
is left-continuous in $c$.  Define
\[
C^\circ(z)
=
\inf\{c\in\mathbb Q:c\ge u_q,\ \bar R_z(c)>1\}.
\]
At every regular $z$, left continuity and
\eqref{eq:regular-right} imply $C^\circ(z)=C(z)$.  Thus $C$ has a
measurable version.  Moreover, for $c<C(z)$,
$\bar R_{0,z}(c)\le\bar R_z(c)<1$; continuity of the null contribution at the
crossing gives $0\le D(z)\le1$.  Hence
\[
D^\circ(z)
=
\begin{cases}
\min\bigl[1,\bar R_{0,z}\{C^\circ(z)\}\bigr],&C^\circ(z)<\infty,\\
0,&C^\circ(z)=\infty,
\end{cases}
\]
is a measurable $[0,1]$-valued version of $D(z)$ at almost every
$z$.  Use these versions below.  Conditional convergence in probability,
together with boundedness, now yields
\[
\mathbb E\bigl[|\FDP-D(z)|\mid Z=z\bigr]\longrightarrow0
\quad\text{for almost every }z.
\]
Dominated convergence and the tower property give
\[
\mathbb E_{\Omega_N,Z,\varepsilon}
\bigl[|\FDP-D(Z)|\bigr]\longrightarrow0,
\]
and a second application of the tower property yields
\eqref{eq:expected-transfer}.
\end{proof}

Lastly, we show that $\Sigma$ can be made positive definite
by an independent Gaussian perturbation.  For a fixed realization
$\omega=((\theta_i,b_i))_{i=1}^N$, let $\xi\sim N(0,I_N)$ be independent and define
\begin{equation}
T_i^{(\eta)}
=
\theta_i
+
\sqrt{1-\eta^2}
\left(
 b_iZ+\sqrt{1-b_i^2}\,\varepsilon_i
\right)
+
\eta\xi_i,
\quad 0<\eta<1.
\label{eq:eta-model}
\end{equation}
Then
\begin{equation}
\Sigma_{\omega,\eta}
=
(1-\eta^2)
\left[
bb^\top+\diag(1-b_1^2,\ldots,1-b_N^2)
\right]
+
\eta^2I_N.
\label{eq:eta-cov}
\end{equation}

The following lemma shows that the perturbation makes the
correlation matrix positive definite while preserving the FDR in the limit.
\begin{lemma}
\label{lem:eta}
For every fixed $\omega$,
\begin{equation}
\diag(\Sigma_{\omega,\eta})=\mathbf 1_N,
\quad
\Sigma_{\omega,\eta}\succeq\eta^2I_N,
\label{eq:eta-PD}
\end{equation}
and
\begin{equation}
\FDR_{\omega}^{(\eta)}(\BH_q)
\longrightarrow
\FDR_{\omega}(\BH_q)
\quad(\eta\downarrow0).
\label{eq:eta-continuity}
\end{equation}
\end{lemma}

\begin{proof}
Equation \eqref{eq:eta-PD} is immediate from \eqref{eq:eta-cov}.  Under the coupling \eqref{eq:eta-model}, $T^{(\eta)}\to T$ almost surely.  Define
$P_i^{(\eta)}=p(|T_i^{(\eta)}|)$.  Then
$P_i^{(\eta)}\to P_i=p(|T_i|)$ almost surely.  Each $P_i$ has a continuous marginal law, so
\[
\mathbb P\left\{
P_i\in\left\{\frac{qk}{N}:1\le k\le N\right\}
\text{ for some }i
\right\}=0.
\]
Outside of this event,
$P_i^{(\eta)}\le qk/N$ if and only if
$P_i\le qk/N$ for when $\eta$ is sufficiently small.  Hence the BH rejection
set remains constant for sufficiently small $\eta$.  Bounded convergence proves \eqref{eq:eta-continuity}.
\end{proof}

{\subsection{Construction of
\texorpdfstring{$\nu_{q,r}$}{nu(q,r)}}}
\label{sec:continuum-design}

We construct a family of laws $\nu_{q,r}$ indexed by $r\downarrow0$.  For $a,y\ge0$, put
{\[
L_a(y)=e^{-a^2/2}\cosh(ay).
\]}
For $y>1$, let
{\[
a(y)=\operatorname*{arg\,max}_{\alpha\ge0} L_\alpha(y),
\quad
y(\alpha)=a^{-1}(\alpha),\quad \alpha>0.
\]}
Set
\begin{equation}
a_q=a(y_q).
\label{eq:aq-def}
\end{equation}
Appendix~\ref{app:envelope} establishes the properties of these functions that we use below.

The quantities $b_r,\tau_r,\mu_r,\pi_{0,r}$, and $c_r$ are defined below.  The law $\nu_{q,r}$ has the three components displayed in Table~\ref{tab:two-construction}.

\begin{table}[ht]

\centering
\caption{Two-sided population construction defining $\nu_{q,r}$.}
\label{tab:two-construction}
\begin{tabular}{lccc}
\toprule
Type & Mass & $\Theta$ & $B$\\
\midrule
Nulls & $\pi_{0,r}$ & $0$ & $r$\\
Primary non-nulls & $r$ & $c_r+y(b_r)$ & $1$\\
Secondary non-nulls & $\tau_r\mu_r(\dd a)$ & $a/r+y(a)$ &
$1$, $a\in[b_r,a_q]$\\
\bottomrule
\end{tabular}
\end{table}

Equivalently, for every bounded measurable $f:\R\times[-1,1]\to\R$,
\begin{equation}
\int f\,\dd\nu_{q,r}
=\pi_{0,r}f(0,r)+r f(c_r+y(b_r),1)
+\tau_r\int_{b_r}^{a_q}f\!\left(\frac ar+y(a),1\right)\mu_r(\dd a).
\label{eq:nuqr}
\end{equation}

\noindent\textbf{Definition of secondary non-nulls.}
For small $r$, set
{\[
b_r=r\,p^{-1}(r^2).
\]}
For the remaining definitions, take $r$ small enough that $0<b_r<a_q$.
For $a\in[b_r,a_q]$, define
\begin{equation}
W_r(a)=\frac{\bigl[1-qM\{y(a)\}\bigr]p(a/r)}q.
\label{eq:Wr}
\end{equation}
By Lemma~\ref{lem:envelope} and \eqref{eq:deficit-properties}, the factor $1-qM\{y(a)\}$ is nonnegative and decreasing on $[b_r,a_q]$; the factor $p(a/r)$ is decreasing in $a$.  Hence $W_r$ is continuous and decreasing, with $W_r(a_q)=0$.
Set
\begin{equation}
\tau_r
:=
W_r(b_r)
=
\frac{\bigl[1-qM\{y(b_r)\}\bigr]r^2}{q}
=O(r^2).
\label{eq:taur}
\end{equation}
We therefore call this the secondary non-null block; its total
mass is $O(r^2)=o(r)$.
Let $\mu_r$ be the probability measure on $[b_r,a_q]$ determined by
\begin{equation}
\mu_r([a,a_q])=\frac{W_r(a)}{\tau_r},
\quad b_r\le a\le a_q.
\label{eq:mur-survival}
\end{equation}
In particular, $\mu_r([b_r,a_q])=W_r(b_r)/\tau_r=1$.
~\\
~\\
\noindent\textbf{Definition of nulls.}
Put
\[
\pi_{0,r}=1-r-\tau_r.
\]
Since $\pi_{0,r}+r+\tau_r=1$, the masses in \eqref{eq:nuqr} define a probability measure.

\medskip
\noindent\textbf{Definition of primary non-nulls.}
For all sufficiently small $r$, $\pi_{0,r}>0$.  Define $c_r$ by
\begin{equation}
p(c_r)
=
\frac{q(r+\tau_r)}{1-q\pi_{0,r}}.
\label{eq:cr}
\end{equation}
The numerator is positive.  Since $\tau_r=O(r^2)$ and $\pi_{0,r}=1-r-\tau_r$,
\[
1-q\pi_{0,r}=1-q+qr+q\tau_r\longrightarrow1-q>0,
\quad
q(r+\tau_r)\longrightarrow0.
\]
Thus the right side of \eqref{eq:cr} is positive and smaller than one for small $r$, so $c_r$ is well-defined.

\subsection{Population FDR analysis under \texorpdfstring{$\nu_{q,r}$}{nu(q,r)}}

\subsubsection{Preliminaries}

Apply Subsection~\ref{sec:generic} with $\nu_q=\nu_{q,r}$.  Denote the resulting quantities by
\[
R_{r,z},
\quad
R_{0,r,z},
\quad
\bar R_{r,z},
\quad
\bar R_{0,r,z},
\quad
C_r(z).
\]
For $c\ge u_q$,
\begin{equation}
\bar R_{r,z}(c)
=\bar R_{0,r,z}(c)+\bar R_{1,r,z}(c)+\bar R_{2,r,z}(c),
\label{eq:population-curve}
\end{equation}
where
\begin{align}
\bar R_{0,r,z}(c)
&=
q\pi_{0,r}
\frac{R_{0,r,z}\{p(c)\}}{p(c)},
\\
\bar R_{1,r,z}(c)
&=
\frac{qr}{p(c)}
\1\{|c_r+y(b_r)+z|\ge c\},
\notag\\
\bar R_{2,r,z}(c)
&=
\frac{q\tau_r}{p(c)}
\int_{b_r}^{a_q}
\1\!\left\{
\left|\frac ar+y(a)+z\right|\ge c
\right\}
\mu_r(\dd a).
\notag
\end{align}
For a measurable definition at every $z$, put
\[
 C_r^\circ(z)
 =\inf\{c\in\mathbb Q:c\ge u_q,\ \bar R_{r,z}(c)>1\}
\]
and define
\[
 D_r(z)=
 \begin{cases}
 \min\bigl[1,\bar R_{0,r,z}\{C_r^\circ(z)\}\bigr],
     &C_r^\circ(z)<\infty,\\
 0,&C_r^\circ(z)=\infty.
 \end{cases}
\]
The joint Borel measurability of $(z,c)\mapsto\bar R_{r,z}(c)$ and
$(z,c)\mapsto\bar R_{0,r,z}(c)$ makes both $C_r^\circ$ and $D_r$
Borel measurable.  At every regular crossing, the rational-cutoff argument
in the proof of Theorem~\ref{thm:population-transfer} gives
$C_r^\circ(z)=C_r(z)$, and $\bar R_{0,r,z}\{C_r(z)\}\le1$.  Hence
\begin{equation}
D_r(z)=\bar R_{0,r,z}\{C_r(z)\}.
\label{eq:Dr}
\end{equation}
Moreover, $0\le D_r(z)\le1$ for every $z$.

We next state two asymptotic lemmas that will be applied in the following
analysis.  The first gives the order of $b_r$ and $c_r$ as
$r\downarrow0$.
\begin{lemma}
\label{lem:scale-properties}
As $r\downarrow0$,
\begin{equation}
b_r\downarrow0,
\quad
\frac{b_r}{r}=p^{-1}(r^2)\to\infty,
\quad
p(b_r/r)=r^2,
\quad
y(b_r)\downarrow1.
\label{eq:br-properties}
\end{equation}
Moreover,
\begin{equation}
p(c_r)\sim\frac{qr}{1-q},
\quad
c_r\to\infty,
\quad
rc_r\to0,
\quad
\frac{b_r}{r}-c_r\to\infty.
\label{eq:cr-properties}
\end{equation}
\end{lemma}

\begin{proof}
The identity $p(b_r/r)=r^2$ follows directly from $b_r=r\,p^{-1}(r^2)$.  Applying \eqref{eq:mills-quantile} in Lemma~\ref{lem:mills} gives
\[
\frac{b_r}{r}
=p^{-1}(r^2)
\sim
\sqrt{4\log(1/r)}
\to\infty,
\quad
b_r
\sim
2r\sqrt{\log(1/r)}
\to0.
\]
The last limit in \eqref{eq:br-properties} follows from Lemma~\ref{lem:envelope}: $a(y)$ is continuous and strictly increasing from $(1,\infty)$ onto $(0,\infty)$, with $\lim_{y\downarrow1}a(y)=0$.  Since $y(\cdot)$ is its inverse and $b_r\downarrow0$, we have $y(b_r)\downarrow1$.

By \eqref{eq:cr} and \eqref{eq:taur},
\[
p(c_r)
=
\frac{q(r+\tau_r)}{1-q\pi_{0,r}}
\sim
\frac{qr}{1-q}.
\]
{For every fixed
$\delta\in(0,q/(1-q))$, for all sufficiently small $r$,
\[
 \left\{\frac{q}{1-q}-\delta\right\}r
 \le p(c_r)\le
 \left\{\frac{q}{1-q}+\delta\right\}r.
\]
Since $p^{-1}$ is decreasing,
\[
 p^{-1}\!\left(\left\{\frac{q}{1-q}+\delta\right\}r\right)
 \le c_r
 \le p^{-1}\!\left(\left\{\frac{q}{1-q}-\delta\right\}r\right).
\]
Equation~\eqref{eq:mills-quantile} in Lemma~\ref{lem:mills}, applied to
the two fixed positive constants, shows that both endpoints are
asymptotic to $\sqrt{2\log(1/r)}$.  Hence
\[
c_r\sim\sqrt{2\log(1/r)}.
\]}
Hence $c_r\to\infty$ and $rc_r\to0$.  Combining this with
\[
\frac{b_r}{r}
=p^{-1}(r^2)
\sim
\sqrt{4\log(1/r)}
\]
gives
\[
\frac{b_r}{r}-c_r
\sim
\left(2-\sqrt2\right)\sqrt{\log(1/r)}
\longrightarrow\infty.
\]
\end{proof}

The next lemma gives a uniform approximation to the normalized
conditional FDP contribution from nulls.
\begin{lemma}
\label{lem:null-tail}
For every $A>0$ and $T\ge0$, as $r\downarrow0$,
\begin{equation}
\sup_{|z|\le T}
\sup_{0\le c\le A/r}
\left|
\frac{R_{0,r,z}\{p(c)\}}{p(c)}
-L_{rc}(|z|)
\right|
\longrightarrow0.
\label{eq:null-tail}
\end{equation}
\end{lemma}

\begin{proof}
For $0<r<1$, the null pair $(\Theta,B)=(0,r)$ gives
\[
R_{0,r,z}\{p(c)\}
=
\mathbb P\{|rz+\sqrt{1-r^2}\,\varepsilon|\ge c\}.
\]
Write
\[
x_\pm=\frac{c\mp rz}{\sqrt{1-r^2}}.
\]
Then
\[
R_{0,r,z}\{p(c)\}=\Phib(x_+)+\Phib(x_-).
\]
Fix $C>0$.  Uniformly for $|z|\le T$ and $C\le c\le A/r$,
\[
\begin{aligned}
x_\pm-c
&=c\left\{\frac1{\sqrt{1-r^2}}-1\right\}
\mp\frac{rz}{\sqrt{1-r^2}}\\
&=\frac{cr^2}{\sqrt{1-r^2}\{1+\sqrt{1-r^2}\}}
\mp\frac{rz}{\sqrt{1-r^2}}.
\end{aligned}
\]
Therefore,
\[
|x_\pm-c|
\le
\frac{Ar}{\sqrt{1-r^2}\{1+\sqrt{1-r^2}\}}
+\frac{Tr}{\sqrt{1-r^2}}
=O_{A,T}(r).
\]
Consequently, for all sufficiently small $r$, the shifts $x_\pm-c$ are
uniformly bounded and $x_\pm\ge C/2$.  The uniform expansion
\eqref{eq:mills-local-ratio} in Lemma~\ref{lem:mills} therefore gives
remainders $\rho_{r,C,\pm}(c,z)$ such that
\begin{equation}
\begin{aligned}
\frac{\Phib(x_\pm)}{\Phib(c)}
&=
\frac{c}{x_\pm}
\exp\left\{-\frac{x_\pm^2-c^2}{2}\right\}
\{1+\rho_{r,C,\pm}(c,z)\},\\
-\frac{x_\pm^2-c^2}{2}
&=
-\frac{(rc)^2}{2}\pm rcz+O_{A,T}(r^2).
\end{aligned}
\label{eq:null-tail-expansion}
\end{equation}
and
\begin{equation}
\lim_{C\to\infty}\limsup_{r\downarrow0}
\sup_{|z|\le T}\sup_{C\le c\le A/r}
\max_{\pm}|\rho_{r,C,\pm}(c,z)|=0.
\label{eq:null-tail-remainder}
\end{equation}
For each fixed $C>0$, $c/x_\pm\to1$ {uniformly for $|z|\le T$ and $C\le c\le A/r$} as
$r\downarrow0$.  To identify the leading term, put $t=rc$, so that
$0\le t\le A$.  Then
\[
\frac12\left[
\exp\left\{-\frac{t^2}{2}+tz\right\}
+\exp\left\{-\frac{t^2}{2}-tz\right\}
\right]
=e^{-t^2/2}\cosh(tz)
=L_t(|z|).
\]
Because $t$ and $z$ range over compact sets, these exponential factors
are uniformly bounded.  Thus, after summing the two tails in
\eqref{eq:null-tail-expansion} and using $p(c)=2\Phib(c)$, the errors from
$c/x_\pm-1$ and from the $O_{A,T}(r^2)$ term vanish uniformly for each
fixed $C$.  Consequently, there is a constant $K_{A,T}<\infty$, independent
of $C$, such that, for
\[
\eta(C)
=
\limsup_{r\downarrow0}
\sup_{|z|\le T}
\sup_{C\le c\le A/r}
\left|
\frac{R_{0,r,z}\{p(c)\}}{p(c)}
-L_{rc}(|z|)
\right|,
\]
we have
\[
\eta(C)
\le
K_{A,T}
\limsup_{r\downarrow0}
\sup_{|z|\le T}\sup_{C\le c\le A/r}
\max_{\pm}|\rho_{r,C,\pm}(c,z)|.
\]
The right side tends to zero as $C\to\infty$ by
\eqref{eq:null-tail-remainder}.  Hence $\eta(C)\to0$.

It remains to control $0\le c\le C$.  On this range, $p(c)\ge p(C)>0$.  Also, uniformly over $0\le c\le C$ and $|z|\le T$,
\[
R_{0,r,z}\{p(c)\}
=
\Phib\!\left(\frac{c-rz}{\sqrt{1-r^2}}\right)
+
\Phib\!\left(\frac{c+rz}{\sqrt{1-r^2}}\right)
=p(c)+O_{C,T}(r),
\]
because the two arguments differ from $c$ by $O_{C,T}(r)$ and $\Phib$ has bounded derivative.  Hence
\[
\sup_{|z|\le T}\sup_{0\le c\le C}
\left|
\frac{R_{0,r,z}\{p(c)\}}{p(c)}-1
\right|
\longrightarrow0.
\]
Moreover, $L_{rc}(|z|)=e^{-(rc)^2/2}\cosh(rc|z|)\to1$ {uniformly for $|z|\le T$ and $0\le c\le C$}.  Thus
\[
\limsup_{r\downarrow0}
\sup_{|z|\le T}
\sup_{0\le c\le A/r}
\left|
\frac{R_{0,r,z}\{p(c)\}}{p(c)}-L_{rc}(|z|)
\right|
\le \eta(C).
\]
Letting $C\to\infty$ proves the claim.
\end{proof}

The following lemma provides a generic criterion for deriving the limit of
$C_r(z)$.
\begin{lemma}
\label{lem:localized-crossing}
Let $K$ be compact.  For $z\in K$, let
$m_r(z)>u_q$ and $w_r(z)>0$.
Suppose that, for every fixed $H>0$,
\begin{equation}
 \inf_{z\in K}\{m_r(z)-Hw_r(z)\}>u_q
 \label{eq:localized-crossing-domain}
\end{equation}
for all sufficiently small $r$, and
\begin{equation}
 \sup_{z\in K}\sup_{|h|\le H}
 \left|
 \bar R_{r,z}\{m_r(z)+w_r(z)h\}
 -d(z)-\{1-d(z)\}e^{\xi(z)h}
 \right|\longrightarrow0,
 \label{eq:localized-crossing-full}
\end{equation}
where $d:K\to[0,1)$ and $\xi:K\to(0,\infty)$ are continuous.
Suppose also that, for every $h_0>0$,
\begin{equation}
 \limsup_{r\downarrow0}\sup_{z\in K}
 \sup_{u_q\le c\le m_r(z)-w_r(z)h_0}\bar R_{r,z}(c)<1.
 \label{eq:localized-crossing-left}
\end{equation}
Then
\begin{equation}
 \sup_{z\in K}
 \left|\frac{C_r(z)-m_r(z)}{w_r(z)}\right|
 \longrightarrow0.
 \label{eq:localized-crossing-cutoff}
\end{equation}
\end{lemma}

\begin{proof}
Fix $h_0>0$.  Compactness gives
\[
 \inf_{z\in K}
 \{1-d(z)\}\{e^{\xi(z)h_0}-1\}>0.
\]
Hence \eqref{eq:localized-crossing-full} gives
$\bar R_{r,z}\{m_r(z)+w_r(z)h_0\}>1$ uniformly on $K$, while
\eqref{eq:localized-crossing-left} excludes a crossing at or before
$m_r(z)-w_r(z)h_0$.  Thus
\[
 m_r(z)-w_r(z)h_0
 <C_r(z)
 \le m_r(z)+w_r(z)h_0
\]
uniformly on $K$.  Letting $h_0\downarrow0$ proves
\eqref{eq:localized-crossing-cutoff}.
\end{proof}

The next lemma gives regularity and the limiting null
contribution once the localized cutoff has been identified.
\begin{lemma}
\label{lem:localized-crossing-regularity}
Under the assumptions and notation of
Lemma~\ref{lem:localized-crossing}, assume in addition that, for every fixed
$H>0$,
\begin{equation}
 \sup_{z\in K}\sup_{|h|\le H}
 \left|\bar R_{0,r,z}\{m_r(z)+w_r(z)h\}-d(z)\right|
 \longrightarrow0.
 \label{eq:localized-crossing-null}
\end{equation}
If, for some $H_*>0$ and all sufficiently small $r$, the map
\[
 h\longmapsto \bar R_{r,z}\{m_r(z)+w_r(z)h\}
\]
is differentiable with positive derivative on $(-H_*,H_*)$, and
the map
$h\mapsto\bar R_{0,r,z}\{m_r(z)+w_r(z)h\}$ is continuous there for
every $z\in K$, then $C_r(z)$ is regular and
\begin{equation}
 \sup_{z\in K}
 \left|D_r(z)-d(z)\right|
 \longrightarrow0.
 \label{eq:localized-crossing-fdp}
\end{equation}
\end{lemma}

\begin{proof}
{
For all sufficiently small $r$, \eqref{eq:localized-crossing-cutoff}
gives
\[
 m_r(z)-w_r(z)H_*<C_r(z)<m_r(z)+w_r(z)H_*.
\]}
With
$c=m_r(z)+w_r(z)h$, the chain rule gives
\[
 \frac{\dd}{\dd h}
 \bar R_{r,z}\{m_r(z)+w_r(z)h\}
 =w_r(z)\bar R_{r,z}'(c).
\]
Since $w_r(z)>0$,
\[
 0<\frac{\dd}{\dd h}
 \bar R_{r,z}\{m_r(z)+w_r(z)h\}
 =w_r(z)\bar R_{r,z}'(c)
 \quad\Longrightarrow\quad
 \bar R_{r,z}'(c)>0.
\]
Thus
\begin{equation}
 \bar R_{r,z}'(c)>0,
 \qquad
 m_r(z)-w_r(z)H_*<c<m_r(z)+w_r(z)H_*.
 \label{eq:localized-positive-derivative}
\end{equation}
Equations~\eqref{eq:localized-crossing-cutoff} and
\eqref{eq:localized-crossing-domain}, together with
\eqref{eq:localized-positive-derivative}, verify the hypotheses of
Proposition~\ref{prop:local-regularity},
which gives regularity.
Finally, put
\[
 h_r(z)=\frac{C_r(z)-m_r(z)}{w_r(z)}.
\]
Equation~\eqref{eq:localized-crossing-cutoff} gives $h_r(z)\to0$
uniformly.  By \eqref{eq:Dr},
\[
 D_r(z)=\bar R_{0,r,z}\{C_r(z)\}
 =\bar R_{0,r,z}\{m_r(z)+w_r(z)h_r(z)\}.
\]
Evaluating \eqref{eq:localized-crossing-null} at $h=h_r(z)$ therefore
proves \eqref{eq:localized-crossing-fdp}.
\end{proof}

The next lemma derives the derivatives of the null and secondary
non-null contributions.
\begin{lemma}
\label{lem:two-block-derivatives}
For $c>0$, put
\[
 x_\pm=\frac{c\mp rz}{\sqrt{1-r^2}},
 \qquad
 Q_\pm(c,z)=\frac{\Phib(x_\pm)}{\Phib(c)}.
\]
The null contribution $\bar R_{0,r,z}(c)$ satisfies
\begin{align}
 \bar R_{0,r,z}'(c)
 &=\frac{q\pi_{0,r}}2
 \sum_{\sigma\in\{-,+\}}Q_\sigma(c,z)
 \left\{\lambda(c)-\frac{\lambda(x_\sigma)}{\sqrt{1-r^2}}\right\},
 \label{eq:two-null-block-derivative}\\
 \frac{\dd}{\dd c}\log\bar R_{0,r,z}(c)
 &=\lambda(c)-
 \frac{Q_+(c,z)\lambda(x_+)+Q_-(c,z)\lambda(x_-)}
 {\sqrt{1-r^2}\{Q_+(c,z)+Q_-(c,z)\}}.
 \label{eq:two-null-block-log-derivative}
\end{align}

Define
\begin{equation}
 g_r(a)=a+r y(a).
 \label{eq:gr}
\end{equation}
{
For $0<a<a_q$,
\[
 g_r'(a)=1+r y'(a)>1,
\]
so $g_r$ is strictly increasing on $[b_r,a_q]$.}
Finally, fix $y>1$ and $c>0$, and suppose that
\begin{equation}
 g_r(b_r)>r(y-c).
\label{eq:two-secondary-lower-tail}
\end{equation}
If
\[
 g_r(b_r)<r(c+y)<g_r(a_q),
\]
let $\beta_{r,y}(c)\in(b_r,a_q)$ be the unique solution of
\begin{equation}
 g_r\{\beta_{r,y}(c)\}=r(c+y).
\label{eq:two-secondary-threshold-equation}
\end{equation}
Conditional on $Z=-y$, a secondary non-null indexed by
$a\in[b_r,a_q]$ is then rejected at cutoff $c$ if and only if
\begin{equation}
 a\in[\beta_{r,y}(c),a_q].
\label{eq:two-secondary-rejection-set}
\end{equation}
Whenever
$1-qM\{y(\beta_{r,y}(c))\}>0$, the secondary non-null contribution
$\bar R_{2,r,-y}(c)$ satisfies
\begin{align}
 \frac{\partial\beta_{r,y}(c)}{\partial c}
 &=\frac{r}{1+r y'\{\beta_{r,y}(c)\}},
 \label{eq:two-secondary-threshold-derivative}\\
 \frac{\dd}{\dd c}\log\bar R_{2,r,-y}(c)
 &=-\frac{qM'\{y(\beta_{r,y}(c))\}y'\{\beta_{r,y}(c)\}}
 {1-qM\{y(\beta_{r,y}(c))\}}
 \frac{r}{1+r y'\{\beta_{r,y}(c)\}}
 \notag\\
 &\quad
 -\frac{\lambda\{\beta_{r,y}(c)/r\}}
 {1+r y'\{\beta_{r,y}(c)\}}+\lambda(c),
 \label{eq:two-secondary-block-log-derivative}\\
 \bar R_{2,r,-y}'(c)
 &=\bar R_{2,r,-y}(c)
 \frac{\dd}{\dd c}\log\bar R_{2,r,-y}(c).
 \label{eq:two-secondary-block-derivative}
\end{align}
If $r(c+y)\le g_r(b_r)$, the secondary non-null block is fully rejected.
Under the strict inequality $r(c+y)<g_r(b_r)$,
\[
 \bar R_{2,r,-y}(c)=\frac{q\tau_r}{p(c)},
 \qquad
 \bar R_{2,r,-y}'(c)
 =-\frac{q\tau_r p'(c)}{p(c)^2}
 =\lambda(c)\bar R_{2,r,-y}(c).
\]
If $r(c+y)\ge g_r(a_q)$, the secondary non-null contribution vanishes:
at equality the rejected set is $\{a_q\}$, which has $\mu_r$-measure
zero by \eqref{eq:mur-survival} and $W_r(a_q)=0$.  Under the strict
inequality $r(c+y)>g_r(a_q)$,
\[
 \bar R_{2,r,-y}(c)=\bar R_{2,r,-y}'(c)=0.
\]
As $r(c+y)\downarrow g_r(b_r)$,
$\beta_{r,y}(c)\downarrow b_r$; as
$r(c+y)\uparrow g_r(a_q)$,
$\beta_{r,y}(c)\uparrow a_q$.  Since $W_r$ is continuous,
$W_r(b_r)=\tau_r$, and $W_r(a_q)=0$, the interior expression
\[
 \bar R_{2,r,-y}(c)=\frac{qW_r\{\beta_{r,y}(c)\}}{p(c)}
\]
matches the endpoint values.  Thus the contribution is continuous at both
equalities.  Its one-sided derivatives need not agree there, so a two-sided
derivative is not asserted.
\end{lemma}

\begin{proof}
Since
\[
 \frac{\dd x_\pm}{\dd c}=\frac1{\sqrt{1-r^2}},
 \qquad
 \frac{\dd}{\dd t}\log\Phib(t)=-\lambda(t),
\]
we have
\[
 \frac{\dd}{\dd c}\log Q_\pm(c,z)
 =-\frac{\lambda(x_\pm)}{\sqrt{1-r^2}}+\lambda(c).
\]
Multiplying by $Q_\pm(c,z)$ and summing gives
\begin{align*}
 \bar R_{0,r,z}'(c)
 &=\frac{q\pi_{0,r}}2\{Q_+'(c,z)+Q_-'(c,z)\}\\
 &=\frac{q\pi_{0,r}}2
 \sum_{\sigma\in\{-,+\}}Q_\sigma(c,z)
 \left\{\lambda(c)-
 \frac{\lambda(x_\sigma)}{\sqrt{1-r^2}}\right\},
\end{align*}
which is \eqref{eq:two-null-block-derivative}, and
\begin{align*}
 \frac{\bar R_{0,r,z}'(c)}{\bar R_{0,r,z}(c)}
 &=\frac{Q_+'(c,z)+Q_-'(c,z)}{Q_+(c,z)+Q_-(c,z)}\\
 &=\lambda(c)-
 \frac{Q_+(c,z)\lambda(x_+)+Q_-(c,z)\lambda(x_-)}
 {\sqrt{1-r^2}\{Q_+(c,z)+Q_-(c,z)\}},
\end{align*}
which is \eqref{eq:two-null-block-log-derivative}.

For a secondary non-null indexed by $a$, conditional on $Z=-y$,
\[
 T(a)=\frac ar+y(a)-y,
 \qquad
 r\{T(a)+y\}=g_r(a).
\]
By \eqref{eq:gr} and the positivity of $y'$ in
Lemma~\ref{lem:envelope}, $g_r'(a)=1+ry'(a)>1$, proving the asserted
strict monotonicity.  Equation~\eqref{eq:two-secondary-lower-tail} gives
\[
 g_r(a)\ge g_r(b_r)>r(y-c)
 \quad\Longrightarrow\quad T(a)>-c,
\]
so rejection can occur only through $T(a)\ge c$.  In the interior case,
\eqref{eq:two-secondary-threshold-equation} and strict monotonicity give
\[
 T(a)\ge c
 \quad\Longleftrightarrow\quad
 g_r(a)\ge r(c+y)
 \quad\Longleftrightarrow\quad
 a\ge\beta_{r,y}(c).
\]
This proves \eqref{eq:two-secondary-rejection-set}.
Thus
\begin{align*}
 \bar R_{2,r,-y}(c)
 &=\frac{q\tau_r}{p(c)}
 \mu_r\bigl([\beta_{r,y}(c),a_q]\bigr)\\
 &=\frac{qW_r\{\beta_{r,y}(c)\}}{p(c)}\\
 &=\bigl[1-qM\{y(\beta_{r,y}(c))\}\bigr]
 \frac{p\{\beta_{r,y}(c)/r\}}{p(c)},
\end{align*}
by \eqref{eq:mur-survival} and \eqref{eq:Wr}.

Differentiating the threshold equation yields
\[
 \bigl[1+ry'\{\beta_{r,y}(c)\}\bigr]
 \frac{\partial\beta_{r,y}(c)}{\partial c}=r,
\]
which proves \eqref{eq:two-secondary-threshold-derivative}.  Since
$p'(t)/p(t)=-\lambda(t)$,
\begin{align*}
 \frac{\dd}{\dd c}\log\bar R_{2,r,-y}(c)
 &=-\frac{qM'\{y(\beta_{r,y}(c))\}y'\{\beta_{r,y}(c)\}}
 {1-qM\{y(\beta_{r,y}(c))\}}
 \frac{\partial\beta_{r,y}(c)}{\partial c}\\
 &\quad
 -\lambda\{\beta_{r,y}(c)/r\}
 \frac1r\frac{\partial\beta_{r,y}(c)}{\partial c}
 +\lambda(c).
\end{align*}
Substitution of \eqref{eq:two-secondary-threshold-derivative} gives
\eqref{eq:two-secondary-block-log-derivative}, and
\[
 \bar R_{2,r,-y}'(c)
 =\bar R_{2,r,-y}(c)
 \frac{\dd}{\dd c}\log\bar R_{2,r,-y}(c)
\]
gives \eqref{eq:two-secondary-block-derivative}.
The same equivalence
$T(a)\ge c\Longleftrightarrow g_r(a)\ge r(c+y)$ gives the endpoint cases:
\[
 \bar R_{2,r,-y}(c)
 =\begin{cases}
 q\tau_r/p(c),&r(c+y)\le g_r(b_r),\\
 0,&r(c+y)\ge g_r(a_q),
 \end{cases}
 \qquad
 \bar R_{2,r,-y}'(c)
 =\begin{cases}
 \lambda(c)\bar R_{2,r,-y}(c),&r(c+y)<g_r(b_r),\\
 0,&r(c+y)>g_r(a_q).
 \end{cases}
\]
Because $W_r$ is continuous, \eqref{eq:mur-survival} implies that
$\mu_r$ has no atoms.  Thus the displayed endpoint values agree with the
limits from the interior, so the contribution is continuous at both
thresholds.
\end{proof}

\subsubsection{\texorpdfstring{The regime $z>-1$}{The regime z > -1}}

\begin{lemma}
\label{lem:primary}
For every compact $K\subset(-1,\infty)$ and every $H>0$,
uniformly for $z\in K$ and $|h|\le H$,
\begin{equation}
\bar R_{r,z}\left(c_r+\frac h{c_r}\right)
\longrightarrow
q+(1-q)e^h.
\label{eq:primary-local}
\end{equation}
Moreover, for every $h_0>0$,
\begin{equation}
\limsup_{r\downarrow0}
\sup_{z\in K}
\sup_{u_q\le c\le c_r-h_0/c_r}
\bar R_{r,z}(c)
\le
q+(1-q)e^{-h_0}<1.
\label{eq:primary-left-bound}
\end{equation}
Consequently,
\begin{equation}
C_r(z)=c_r+o(c_r^{-1})
\label{eq:primary-result}
\end{equation}
and
\begin{equation}
D_r(z)\longrightarrow q,
\label{eq:primary-fdp-result}
\end{equation}
uniformly on $K$.  Moreover, for sufficiently small $r$, $C_r(z)$ is regular for every $z\in K$.
\end{lemma}

\begin{proof}
\textbf{Proof of \eqref{eq:primary-local}.}
Fix $H>0$, take $z\in K$, and take $|h|\le H$.  The primary non-null has $B=1$ and $\Theta=c_r+y(b_r)$, so conditional on $Z=z$ its value is $c_r+y(b_r)+z$.  Since $K\subset(-1,\infty)$ is compact,
\[
\eta_K:=\inf_{z\in K}(1+z)>0.
\]
By Lemma~\ref{lem:scale-properties}, $c_r\to\infty$ and $y(b_r)\downarrow1$.  For the fixed $H$, this implies $H/c_r\to0$.  Combining these limits, for sufficiently small $r$,
\begin{equation}
|y(b_r)-1|\le\frac{\eta_K}{4},
\quad
\frac{H}{c_r}\le\frac{\eta_K}{4}.
\label{eq:primary-eta-bounds}
\end{equation}
For $z\in K$ and $|h|\le H$, the definition of $\eta_K$ gives $z\ge -1+\eta_K$, while \eqref{eq:primary-eta-bounds} gives $y(b_r)\ge1-\eta_K/4$ and $-h/c_r\ge-H/c_r\ge-\eta_K/4$.  Therefore
\[
y(b_r)+z-\frac h{c_r}
\ge
\left(1-\frac{\eta_K}{4}\right)
+\left(-1+\eta_K\right)
-\frac{\eta_K}{4}
=\frac{\eta_K}{2},
\quad z\in K,\ |h|\le H.
\]
Equivalently, $y(b_r)+z\ge h/c_r+\eta_K/2$.  Hence the conditional value of the primary non-null satisfies
\begin{equation}
\begin{aligned}
|c_r+y(b_r)+z|
&=c_r+y(b_r)+z\\
&=c_r+\frac h{c_r}
+\left\{y(b_r)+z-\frac h{c_r}\right\}
\\
&\ge c_r+\frac h{c_r}+\frac{\eta_K}{2}
>c_r+\frac h{c_r},
\end{aligned}
\label{eq:primary-rejection-bound}
\end{equation}
so the primary non-null is rejected at cutoff $c_r+h/c_r$, uniformly for $z\in K$ and $|h|\le H$.

Next consider a secondary non-null indexed by $a\in[b_r,a_q]$.  Its conditional value at $Z=z$ is $a/r+y(a)+z$.  Since $a\ge b_r$, we have $a/r\ge b_r/r$.  Also, Lemma~\ref{lem:envelope} shows that $y(\cdot)$, the inverse of the strictly increasing map $a(\cdot)$, is increasing; hence $y(a)\ge y(b_r)$.  Together with $z\ge\inf_{u\in K}u$ and $h\le H$, this gives
\begin{equation}
\frac ar+y(a)+z-\left(c_r+\frac h{c_r}\right)
\ge
\left(\frac{b_r}{r}-c_r\right)
+y(b_r)+\inf_{u\in K}u-\frac{H}{c_r}.
\label{eq:secondary-rejection-bound}
\end{equation}
Here $\frac{b_r}{r}-c_r\to\infty$ by \eqref{eq:cr-properties}, while
\begin{equation}
y(b_r)+\inf_{u\in K}u-\frac{H}{c_r}
\longrightarrow
1+\inf_{u\in K}u
\label{eq:secondary-offset-limit}
\end{equation}
because $y(b_r)\downarrow1$ and $H/c_r\to0$.  Equations~\eqref{eq:cr-properties} and~\eqref{eq:secondary-offset-limit} show that the right side of \eqref{eq:secondary-rejection-bound} tends to $+\infty$.  Hence it is positive for sufficiently small $r$, uniformly in $z\in K$, $|h|\le H$, and $a\in[b_r,a_q]$.  Thus every secondary non-null is rejected at cutoff $c_r+h/c_r$.

Lastly, we analyze the nulls.  Since $|h|\le H$ and $rc_r\to0$ by \eqref{eq:cr-properties},
\begin{equation}
r\left(c_r+\frac h{c_r}\right)
=rc_r+\frac{rh}{c_r}
\longrightarrow0.
\label{eq:primary-rescaled-cutoff}
\end{equation}
Let $T_K=\sup_{z\in K}|z|<\infty$.  Since
$c_r\to\infty$ by \eqref{eq:cr-properties} in
Lemma~\ref{lem:scale-properties}, while
\eqref{eq:primary-rescaled-cutoff} gives a vanishing rescaled cutoff, we have
\[
u_q\le c_r+\frac h{c_r}\le\frac1r
\]
for all sufficiently small $r$, uniformly for $z\in K$ and
$|h|\le H$.  Applying \eqref{eq:null-tail} in Lemma~\ref{lem:null-tail}
with $A=1$ and $T=T_K$ gives
\[
\frac{R_{0,r,z}\{p(c_r+h/c_r)\}}{p(c_r+h/c_r)}
=
L_{r(c_r+h/c_r)}(|z|)+o(1).
\]
Moreover,
\[
L_{r(c_r+h/c_r)}(|z|)
=
\exp\left\{-\frac{r^2(c_r+h/c_r)^2}{2}\right\}
\cosh\!\left(r(c_r+h/c_r)|z|\right)
=1+o(1),
\]
uniformly for $z\in K$ and $|h|\le H$.  Since $\pi_{0,r}=1-r-\tau_r\to1$, it follows that
\[
\bar R_{0,r,z}(c_r+h/c_r)
=q+o(1).
\]
By \eqref{eq:cr},
\[
\frac{q(r+\tau_r)}{p(c_r+h/c_r)}
=
(1-q\pi_{0,r})
\frac{p(c_r)}{p(c_r+h/c_r)}
\longrightarrow
(1-q)e^h
\]
by \eqref{eq:h-over-x} in Lemma~\ref{lem:mills}.  This proves \eqref{eq:primary-local}.

\par\medskip\noindent
\textbf{Proof of \eqref{eq:primary-left-bound}.}
Fix $h_0>0$.  Equation~\eqref{eq:primary-rejection-bound}, evaluated at $h=h_0$, gives
\[
|c_r+y(b_r)+z|>c_r+\frac{h_0}{c_r},
\]
so the primary non-null is rejected at cutoff $c_r+h_0/c_r$.  Equations~\eqref{eq:secondary-rejection-bound}, \eqref{eq:cr-properties}, and~\eqref{eq:secondary-offset-limit} give
\[
\frac ar+y(a)+z>c_r+\frac{h_0}{c_r},
\quad a\in[b_r,a_q],
\]
for sufficiently small $r$, uniformly in $z\in K$.  Thus every secondary non-null is also rejected at cutoff $c_r+h_0/c_r$.  Since rejection at cutoff $c$ means $|T_i|\ge c$, all non-nulls are rejected for every $u_q\le c\le c_r+h_0/c_r$.  Therefore, on this interval,
\[
\bar R_{r,z}(c)
=
\bar R_{0,r,z}(c)
+
\frac{q(r+\tau_r)}{p(c)}.
\]
For $u_q\le c\le c_r-h_0/c_r$, Lemma~\ref{lem:null-tail} and $r c_r\to0$ from \eqref{eq:cr-properties} give
\[
\sup_{z\in K}
\sup_{u_q\le c\le c_r-h_0/c_r}
\bar R_{0,r,z}(c)
\le q+o(1).
\]
Also $p(c)\ge p(c_r-h_0/c_r)$, so \eqref{eq:cr} and \eqref{eq:h-over-x} in Lemma~\ref{lem:mills} give
\[
\sup_{u_q\le c\le c_r-h_0/c_r}
\frac{q(r+\tau_r)}{p(c)}
\le
\frac{q(r+\tau_r)}{p(c_r-h_0/c_r)}
=
(1-q\pi_{0,r})
\frac{p(c_r)}{p(c_r-h_0/c_r)}
\longrightarrow
(1-q)e^{-h_0}.
\]
Taking the supremum over $z\in K$ and then the limsup in $r$ proves \eqref{eq:primary-left-bound}.

\par\medskip\noindent
\textbf{Proof of \eqref{eq:primary-result}.}
Apply Lemma~\ref{lem:localized-crossing} with
\[
 m_r(z)=c_r,
 \qquad w_r(z)=c_r^{-1},
 \qquad d(z)=q,
 \qquad \xi(z)=1.
\]
The set $K$ is compact by hypothesis.  Since $c_r\to\infty$, for all
sufficiently small $r$, $m_r(z)=c_r>u_q$ and
$w_r(z)=c_r^{-1}>0$.  Moreover, $d\equiv q$ and $\xi\equiv1$ are
continuous on $K$, with $d(z)\in[0,1)$ and $\xi(z)>0$.  For every
fixed $H>0$,
\[
 \inf_{z\in K}\{m_r(z)-Hw_r(z)\}
 =c_r-\frac H{c_r}>u_q
\]
for all sufficiently small $r$, which verifies
\eqref{eq:localized-crossing-domain}.  Under these choices,
\eqref{eq:localized-crossing-full} becomes
\[
 \sup_{z\in K}\sup_{|h|\le H}
 \left|
 \bar R_{r,z}\left(c_r+\frac h{c_r}\right)
 -q-(1-q)e^h
 \right|\longrightarrow0,
\]
which is exactly \eqref{eq:primary-local}.  Its left-of-crossing assumption
\eqref{eq:localized-crossing-left} becomes, for every $h_0>0$,
\[
 \limsup_{r\downarrow0}\sup_{z\in K}
 \sup_{u_q\le c\le c_r-h_0/c_r}\bar R_{r,z}(c)<1.
\]
This follows from \eqref{eq:primary-left-bound}, because
$q+(1-q)e^{-h_0}<1$.  Hence all assumptions of
Lemma~\ref{lem:localized-crossing} hold, and
\eqref{eq:localized-crossing-cutoff} gives
\[
 \sup_{z\in K}
 \left|\frac{C_r(z)-c_r}{c_r^{-1}}\right|
 =\sup_{z\in K}c_r|C_r(z)-c_r|\longrightarrow0,
\]
which is \eqref{eq:primary-result} uniformly on $K$.

\par\medskip\noindent
\textbf{Proof of \eqref{eq:primary-fdp-result}.}
Fix $H>0$.
On $|c-c_r|\le H/c_r$, all non-nulls are rejected, uniformly over
$z\in K$, and hence
\begin{equation}
\bar R_{r,z}(c)
=\bar R_{0,r,z}(c)+\frac{q(r+\tau_r)}{p(c)}.
\label{eq:primary-smooth-curve}
\end{equation}
{
In the notation of Lemma~\ref{lem:two-block-derivatives}, uniformly for
$|c-c_r|\le H/c_r$ and $z\in K$, \eqref{eq:cr-properties} gives
\[
 c\longrightarrow\infty,\qquad rc\longrightarrow0.
\]
For all sufficiently small $r$, this range lies in
$[C,1/r]$, with $C\to\infty$.  Thus
\eqref{eq:null-tail-expansion}-\eqref{eq:null-tail-remainder} in
Lemma~\ref{lem:null-tail} give, for either sign,
\[
\begin{aligned}
 Q_\pm(c,z)
 &=
 \frac{c}{x_\pm}
 \exp\left\{-\frac{(rc)^2}{2}\pm rcz+O(r^2)\right\}
 \{1+o(1)\}\\
 &=1+o(1),
\end{aligned}
\]
because $c/x_\pm\to1$, $rc\to0$, and $z$ ranges over the compact set
$K$.  Also $x_\pm/c_r\to1$, $c/c_r\to1$, and
$\sqrt{1-r^2}\to1$.
Thus \eqref{eq:mills-hazard} in Lemma~\ref{lem:mills}, applied directly
to \eqref{eq:two-null-block-derivative} in
Lemma~\ref{lem:two-block-derivatives}, gives
\[
\sup_{z\in K}\sup_{|c-c_r|\le H/c_r}
|\bar R_{0,r,z}'(c)|=o(c_r).
\]}
For $|c-c_r|\le H/c_r$, write $c=c_r+h/c_r$, where
$|h|\le H$.  By
\eqref{eq:cr},
\[
\frac{q(r+\tau_r)}{p(c)}
=(1-q\pi_{0,r})\frac{p(c_r)}{p(c)},
\]
and \eqref{eq:h-over-x} in Lemma~\ref{lem:mills} gives
\[
\frac{p(c_r)}{p(c)}
=
\frac{p(c_r)}{p(c_r+h/c_r)}
\longrightarrow e^h,
\]
uniformly for $|h|\le H$.  Since $e^h\ge e^{-H}$ on this range and
$1-q\pi_{0,r}\to1-q>0$, for all sufficiently small $r$,
\[
\inf_{|c-c_r|\le H/c_r}
\frac{q(r+\tau_r)}{p(c)}
\ge \frac{1-q}{2}\cdot\frac{e^{-H}}{2}.
\]
Thus we may take $m_H=(1-q)e^{-H}/4>0$.
Since all non-nulls are fully rejected for $|c-c_r|\le H/c_r$, their combined
contribution is $q(r+\tau_r)/p(c)$, whose derivative is
$\lambda(c)q(r+\tau_r)/p(c)$.  Together with
$\lambda(c)\sim c_r$ uniformly for $|c-c_r|\le H/c_r$, this gives
\[
\inf_{|c-c_r|\le H/c_r}
\frac{\dd}{\dd c}\frac{q(r+\tau_r)}{p(c)}
\ge \frac{m_H}{2}c_r
\]
for all sufficiently small $r$.  By contrast,
$\bar R_{0,r,z}'(c)=o(c_r)$ uniformly over $z\in K$ and
$|c-c_r|\le H/c_r$.  Differentiating \eqref{eq:primary-smooth-curve}
therefore shows that
\begin{equation}
\inf_{z\in K}\inf_{|c-c_r|\le H/c_r}
\bar R_{r,z}'(c)>0
\label{eq:primary-strict-increase}
\end{equation}
for all sufficiently small $r$.  By \eqref{eq:primary-result}, uniformly
for $z\in K$,
\[
 c_r-\frac{H}{c_r}<C_r(z)<c_r+\frac{H}{c_r}
\]
for all sufficiently small $r$; moreover,
$c_r-H/c_r>u_q$.  Equation~\eqref{eq:primary-strict-increase} and the
smoothness of the null contribution on this interval therefore verify the
hypotheses of Proposition~\ref{prop:local-regularity}, so $C_r(z)$ is
regular.

Equation~\eqref{eq:Dr} gives
\[
 D_r(z)=\bar R_{0,r,z}\{C_r(z)\}.
\]
Put $h_r(z)=c_r\{C_r(z)-c_r\}$.  Equation~\eqref{eq:primary-result}
gives $h_r(z)\to0$ uniformly on $K$.  The null limit established in
the proof of \eqref{eq:primary-local}, evaluated at $h=h_r(z)$, now gives
\[
 \sup_{z\in K}
 \left|D_r(z)-q\right|
 =\sup_{z\in K}
 \left|\bar R_{0,r,z}\{C_r(z)\}-q\right|
 \longrightarrow0,
\]
which is \eqref{eq:primary-fdp-result}.
\end{proof}

\subsubsection{\texorpdfstring{The regime $z\in(-y_q,-1)$}{The regime -yq < z < -1}}

Let $y=-z$.
For $1<y<y_q$, put
\begin{equation}
\kappa(y)=a(y)y'\{a(y)\}>0.
\label{eq:kappa}
\end{equation}
The positivity follows from Lemma~\ref{lem:envelope}, since $a(y)>0$ and $y'\{a(y)\}>0$ for $1<y<y_q$.
For the remainder of the population analysis, define
\begin{equation}
 s_r(y)=c_r+y(b_r)-y.
 \label{eq:sr-def}
\end{equation}

The next lemma separates the primary non-null value from the
cutoffs on the $r^{-1}$ scale.
\begin{lemma}
\label{lem:primary-scale-separation}
For every compact $I\subset(1,\infty)$ and every $H>0$, if
$a_I=\inf_{y\in I}a(y)>0$, then
\begin{equation}
\begin{aligned}
 \sup_{y\in I}r|s_r(y)|
 &\le rc_r+r\sup_{y\in I}|y(b_r)-y|\longrightarrow0,\\
 \inf_{\substack{y\in I\\|h|\le H}}
 r\left\{\frac{a(y)}r+h\right\}
 &=a_I-rH\longrightarrow a_I>0.
\end{aligned}
\label{eq:rare-primary-scale-separation}
\end{equation}
Consequently, the primary non-null is not rejected at cutoff
$a(y)/r+h$ for all sufficiently small $r$, uniformly for $y\in I$
and $|h|\le H$.
\end{lemma}

\begin{proof}
Conditional on $Z=-y$, the primary non-null value is $s_r(y)$
by \eqref{eq:sr-def}.  The two limits in
\eqref{eq:rare-primary-scale-separation} follow from
\eqref{eq:br-properties} and \eqref{eq:cr-properties} in
Lemma~\ref{lem:scale-properties}.  Their limiting bounds are separated by
$a_I>0$, which proves the rejection claim.
\end{proof}

\begin{lemma}
\label{lem:primary-active-range}
For every compact $I\subset(1,\infty)$,
\begin{equation}
 \sup_{y\in I}\sup_{u_q\le s\le s_r(y)}
 \bar R_{r,-y}(s)\le q+o(1)<1.
 \label{eq:primary-active-uniform}
\end{equation}
\end{lemma}

\begin{proof}
Put $y_I=\inf I>1$ and
$\delta_I^{\mathrm{shift}}=(y_I-1)/2>0$.  Since $y(b_r)\downarrow1$,
\begin{equation}
 \delta_I^{\mathrm{shift}}\le y-y(b_r)
 \le \sup_{u\in I}u-1=: \Delta_I<\infty
 \label{eq:primary-shift-bounds}
 \end{equation}
for all sufficiently small $r$, uniformly for $y\in I$.  Hence
$s_r(y)\le c_r-\delta_I^{\mathrm{shift}}$.  From
\eqref{eq:population-curve} and the bound
by one on each non-null rejection probability,
\[
 \bar R_{r,-y}(s)
 \le \bar R_{0,r,-y}(s)+\frac{q(r+\tau_r)}{p(s)},
 \qquad u_q\le s\le s_r(y).
\]
Uniformly on this range, $0\le rs\le rc_r\to0$.  Therefore
\eqref{eq:null-tail} in Lemma~\ref{lem:null-tail} gives
\[
 \bar R_{0,r,-y}(s)=q+o(1).
\]
Since $p$ is decreasing, \eqref{eq:cr} yields
\[
\begin{aligned}
 \frac{q(r+\tau_r)}{p(s)}
 &\le \frac{q(r+\tau_r)}{p\{s_r(y)\}}\\
 &=(1-q\pi_{0,r})
 \frac{p(c_r)}{p\{c_r+y(b_r)-y\}}.
\end{aligned}
\]
Here $s_r(y)\to\infty$ uniformly and the shift
$y-y(b_r)$ lies in
$[\delta_I^{\mathrm{shift}},\Delta_I]$.  Thus
\eqref{eq:fixed-shift} in Lemma~\ref{lem:mills} gives

\[
 \sup_{y\in I}
 \frac{p(c_r)}{p\{c_r+y(b_r)-y\}}\longrightarrow0.
\]
Combining the bounds proves
\eqref{eq:primary-active-uniform}.
\end{proof}

We first prove a lemma that is useful for the key result in this regime.
\begin{lemma}
\label{lem:secondary-threshold}
For every compact $I\subset(1,y_q)$ and every $H>0$,
there is an interval $U_I$ whose closure is contained in $(0,a_q)$ and contains $\{a(y):y\in I\}$.
For all sufficiently small $r$, with $g_r$ as in
\eqref{eq:gr}, the equation
\begin{equation}
g_r\{\tilde{\alpha}_r(y,h)\}=g_r\{a(y)\}+rh
\label{eq:tilde-alpha-equation}
\end{equation}
has a unique solution $\tilde{\alpha}_r(y,h)\in U_I$, uniformly for $y\in I$ and $|h|\le H$, and
\begin{equation}
\tilde{\alpha}_r(y,h)=a(y)+O(r).
\label{eq:alpha-close}
\end{equation}

Conditional on $Z=-y$, a secondary non-null with $a=\tilde a$ is rejected at cutoff $a(y)/r+h$ {if and only if}
\[
\tilde a\in[\tilde{\alpha}_r(y,h),a_q].
\]
Finally,
\begin{equation}
\log\frac{p\{\tilde{\alpha}_r(y,h)/r\}}{p(a(y)/r+h)}
=
\kappa(y)h+o(1),
\label{eq:tilde-alpha-log-ratio}
\end{equation}
uniformly for $y\in I$ and $|h|\le H$.
\end{lemma}

\begin{proof}
Fix a compact $I\subset(1,y_q)$ and $H>0$ as in the statement.
Let
\[
A_I=\{a(y):y\in I\}.
\]
Because $I\subset(1,y_q)$ is compact, $A_I$ is a compact subset of $(0,a_q)$.  Choose an interval $U_I$ whose closure is contained in $(0,a_q)$ and such that $A_I\subset U_I$.
{
By \eqref{eq:br-properties} in Lemma~\ref{lem:scale-properties},
$b_r\to0$, so $U_I\subset(b_r,a_q)$ for all sufficiently small
$r$.  Lemma~\ref{lem:two-block-derivatives} then shows that $g_r$ is
strictly increasing on $U_I$.}

Put $d_I=\mathrm{dist}(A_I,U_I^c)>0$.  If $rH<d_I/2$, then $a(y)\pm rH\in U_I$ for all $y\in I$.  Since $g_r'\ge1$ on $U_I$, the intermediate value theorem gives a unique solution to \eqref{eq:tilde-alpha-equation} in $U_I$, and
\[
|\tilde{\alpha}_r(y,h)-a(y)|\le rH.
\]
Hence, \eqref{eq:alpha-close} is proved.

Put $c=a(y)/r+h$ and $\beta=\tilde{\alpha}_r(y,h)$.  If
$a_I=\inf_{u\in I}a(u)>0$, then
\[
 r(y-c)
 =ry-a(y)-rh
 \le r\sup_{u\in I}u-a_I+rH<0<g_r(b_r)
\]
for all sufficiently small $r$, uniformly for $y\in I$ and
$|h|\le H$.  Thus \eqref{eq:two-secondary-lower-tail} holds uniformly.
Moreover,
\eqref{eq:tilde-alpha-equation} and $y\{a(y)\}=y$ give
\[
 g_r(\beta)
 =g_r\{a(y)\}+rh
 =r(c+y).
\]
This is \eqref{eq:two-secondary-threshold-equation}, so
\eqref{eq:two-secondary-rejection-set} shows that a secondary non-null is
rejected exactly when its index lies in
$[\tilde{\alpha}_r(y,h),a_q]$.

Write $\beta=\tilde{\alpha}_r(y,h)$.  Differentiating
\eqref{eq:tilde-alpha-equation} with respect to $h$ gives
\[
 \frac{\partial\beta}{\partial h}
 =\frac{r}{1+r y'(\beta)}.
\]
Since $p'(x)/p(x)=-\lambda(x)$,
\begin{equation}
 \frac{\partial}{\partial h}
 \log\frac{p(\beta/r)}{p\{a(y)/r+h\}}
 =
 \lambda\!\left(\frac{a(y)}r+h\right)
 -\frac{\lambda(\beta/r)}{1+r y'(\beta)}.
 \label{eq:tilde-alpha-log-ratio-derivative}
\end{equation}
Moreover, \eqref{eq:tilde-alpha-equation} and $y\{a(y)\}=y$ imply
\begin{equation}
 \frac{\beta}{r}
 =\frac{a(y)}r+h+y-y(\beta).
 \label{eq:tilde-alpha-threshold-identity}
\end{equation}
The function $a(y)$ is bounded away from zero on $I$, and
\eqref{eq:alpha-close} gives $\beta=a(y)+O(r)$ uniformly.  Hence both
arguments of $\lambda$ in
\eqref{eq:tilde-alpha-log-ratio-derivative} are of order $r^{-1}$.
Using $\lambda(x)=x+O(x^{-1})$ from
\eqref{eq:mills-hazard} in Lemma~\ref{lem:mills} and then
\eqref{eq:tilde-alpha-threshold-identity}, we obtain
\[
\begin{aligned}
 \lambda\!\left(\frac{a(y)}r+h\right)
 -\frac{\lambda(\beta/r)}{1+r y'(\beta)}
 &=
 \frac{a(y)}r+h
 -\frac{\beta/r}{1+r y'(\beta)}+O(r)\\
 &=
 \frac{\{a(y)+rh\}y'(\beta)+y(\beta)-y}
      {1+r y'(\beta)}+O(r)\\
 &\longrightarrow a(y)y'\{a(y)\}=\kappa(y)
\end{aligned}
\]
uniformly for $y\in I$ and $|h|\le H$.  At $h=0$, uniqueness in
\eqref{eq:tilde-alpha-equation} gives $\beta=a(y)$, so the logarithm in
\eqref{eq:tilde-alpha-log-ratio-derivative} is zero.  Integrating from
$0$ to $h$ proves \eqref{eq:tilde-alpha-log-ratio}.
\end{proof}

\begin{lemma}
\label{lem:rare-local}
For every compact $I\subset(1,y_q)$ and every $H>0$,
uniformly for $y\in I$ and $|h|\le H$,
\begin{equation}
\bar R_{r,-y}\left(\frac{a(y)}r+h\right)
\longrightarrow
qM(y)+\{1-qM(y)\}e^{\kappa(y)h}.
\label{eq:rare-local}
\end{equation}
Moreover, for every $h_0>0$, there is $\gamma>0$ such that, for all sufficiently small $r$,
\begin{equation}
\sup_{y\in I}
\sup_{u_q\le c\le a(y)/r-h_0}
\bar R_{r,-y}(c)
\le1-\gamma.
\label{eq:no-earlier}
\end{equation}
Consequently, uniformly for $y\in I$,
\begin{equation}
C_r(-y)-\frac{a(y)}r\longrightarrow0
\label{eq:rare-result}
\end{equation}
and
\begin{equation}
D_r(-y)\longrightarrow qM(y).
\label{eq:rare-fdp-result}
\end{equation}
Moreover, for sufficiently small $r$, $C_r(-y)$ is regular for every
$y\in I$.
\end{lemma}

\begin{proof}
\textbf{Proof of \eqref{eq:rare-local}.}
Fix $H>0$, take $y\in I$, and take $|h|\le H$.  By Lemma~\ref{lem:secondary-threshold}, a secondary non-null with $a=\tilde a$ is rejected at cutoff $a(y)/r+h$ {if and only if}
\[
\tilde a\in[\tilde{\alpha}_r(y,h),a_q].
\]
By \eqref{eq:mur-survival}, their mass under the secondary component is
\[
\tau_r\mu_r([\tilde{\alpha}_r(y,h),a_q])
=W_r\{\tilde{\alpha}_r(y,h)\}.
\]
The definition \eqref{eq:Wr} therefore gives
\[
\frac{qW_r\{\tilde{\alpha}_r(y,h)\}}{p(a(y)/r+h)}
=
\bigl[1-qM\{y(\tilde{\alpha}_r(y,h))\}\bigr]
\frac{p\{\tilde{\alpha}_r(y,h)/r\}}{p(a(y)/r+h)}.
\]
By \eqref{eq:alpha-close}, $y(\tilde{\alpha}_r(y,h))\to y$
uniformly for $y\in I$ and $|h|\le H$.  Equation~\eqref{eq:tilde-alpha-log-ratio} gives
\[
\frac{p\{\tilde{\alpha}_r(y,h)/r\}}{p(a(y)/r+h)}
\longrightarrow e^{\kappa(y)h},
\]
uniformly for $y\in I$ and $|h|\le H$.  Hence
\begin{equation}
\frac{qW_r\{\tilde{\alpha}_r(y,h)\}}{p(a(y)/r+h)}
=
\bigl[1-qM\{y(\tilde{\alpha}_r(y,h))\}\bigr]
\frac{p\{\tilde{\alpha}_r(y,h)/r\}}{p(a(y)/r+h)}
\longrightarrow
\{1-qM(y)\}e^{\kappa(y)h}
\label{eq:signal-local}
\end{equation}

Lemma~\ref{lem:primary-scale-separation}, through
\eqref{eq:rare-primary-scale-separation}, shows that the primary non-null is
not rejected at cutoff $a(y)/r+h$, uniformly for $y\in I$ and
$|h|\le H$.

Since $a(y)$ is bounded away from zero on $I$ and
$a(y)\le a_q$, for all sufficiently small $r$,
\[
u_q\le \frac{a(y)}r+h\le \frac{a_q+1}{r}
\]
uniformly for $y\in I$ and $|h|\le H$.  Thus
\eqref{eq:null-tail} in Lemma~\ref{lem:null-tail}, applied with
$A=a_q+1$ and $T=\sup_{u\in I}u$, gives
\begin{equation}
\bar R_{0,r,-y}(a(y)/r+h)
\longrightarrow
qL_{a(y)}(y)=qM(y).
\label{eq:rare-null-local}
\end{equation}
Combining \eqref{eq:signal-local} and \eqref{eq:rare-null-local} proves \eqref{eq:rare-local}.

\par\medskip\noindent
\textbf{Proof of \eqref{eq:no-earlier}.}
Lemma~\ref{lem:primary-active-range}, through
\eqref{eq:primary-active-uniform}, excludes a crossing on
$u_q\le s\le s_r(y)$, uniformly for $y\in I$.
Fix $h_0>0$.
The complementary range $s_r(y)<s\le a(y)/r-h_0$ is nonempty uniformly
on $I$, because
\[
\inf_{y\in I}
\left\{
\frac{a(y)}r-h_0-s_r(y)
\right\}
\ge
\frac{\inf_{y\in I}a(y)}r-c_r-O_I(1)
\longrightarrow\infty,
\]
where $rc_r\to0$ by \eqref{eq:cr-properties} in
Lemma~\ref{lem:scale-properties}.
It remains to prove that there is $\gamma>0$ such that, for all sufficiently
small $r$,
\begin{equation}
\sup_{y\in I}
\sup_{s_r(y)<s\le a(y)/r-h_0}
\bar R_{r,-y}(s)
\le 1-\gamma
\label{eq:secondary-range-target}
\end{equation}

For $s>s_r(y)$, the identity $T_i=s_r(y)>0$ gives $|T_i|=T_i<s$.  For sufficiently small $r$, every secondary non-null value is positive uniformly over $y\in I$ and $\tilde a\in[b_r,a_q]$.
{For every $y>1$ and $s>s_r(y)$, define
\begin{equation}
\mathcal A_{r,y,s}:=\left\{\tilde a\in[b_r,a_q]:g_r(\tilde a)\ge r(s+y)\right\},
\qquad \beta_{r,y,s}:=\inf\bigl(\mathcal A_{r,y,s}\cup\{a_q\}\bigr).
\label{eq:beta-rys-def}
\end{equation}}
On the range $s_r(y)<s\le a(y)/r-h_0$, the set is nonempty for
sufficiently small $r$, since $\tilde a=a_q$ satisfies the inequality
uniformly over $y\in I$.  By \eqref{eq:gr} and
Lemma~\ref{lem:two-block-derivatives}, $g_r$ is strictly increasing.
Thus a secondary non-null with $a=\tilde a$ is rejected if and only if
$\tilde a\in[\beta_{r,y,s},a_q]$.  Therefore, by
\eqref{eq:population-curve},
\begin{equation}
\bar R_{r,-y}(s)
=
\bar R_{0,r,-y}(s)
+
\frac{q\tau_r}{p(s)}
\mu_r([\beta_{r,y,s},a_q]).
\label{eq:secondary-only-curve}
\end{equation}

\noindent Suppose \eqref{eq:secondary-range-target} fails for this $h_0$.  Then there are
\begin{equation}
r_n\downarrow0,
\quad y_n\in I,
\quad s_{r_n}(y_n)<s_n\le a(y_n)/r_n-h_0,
\quad \liminf_{n\to\infty}\bar R_{r_n,-y_n}(s_n)\ge1.
\label{eq:secondary-contradiction-sequence}
\end{equation}
After taking a subsequence,
\begin{equation}
y_n\to y\in I,
\quad a(y_n)\to a(y),
\quad x_n=r_ns_n\to x\in[0,a(y)].
\label{eq:secondary-contradiction-limits}
\end{equation}

Define
\[
h_n:=s_n-\frac{a(y_n)}{r_n}.
\]
Then $h_n\le -h_0$ by \eqref{eq:secondary-contradiction-sequence}.  If $h_n=O(1)$, then after passing to a further subsequence $h_n\to h\le -h_0$.  Equation~\eqref{eq:rare-local} gives
\[
\bar R_{r_n,-y_n}(s_n)
\longrightarrow
qM(y)+\{1-qM(y)\}e^{\kappa(y)h}
\le
qM(y)+\{1-qM(y)\}e^{-\kappa(y)h_0}
<1,
\]
contradicting $\liminf_{n\to\infty}\bar R_{r_n,-y_n}(s_n)\ge1$.  Hence, after passing to a further subsequence if necessary, it remains to consider
\begin{equation}
h_n\to-\infty.
\label{eq:hn-minus-infty}
\end{equation}

We next identify the rejected secondary non-nulls at cutoff $s_n$.  A secondary non-null with index $\tilde a\in[b_{r_n},a_q]$ has $(\Theta,B)=(\tilde a/r_n+y(\tilde a),1)$.  Conditional on $Z=-y_n$, its value is
\[
T_{\tilde a,n}
=
\frac{\tilde a}{r_n}+y(\tilde a)-y_n.
\]
For all large $n$, this value is positive uniformly in $\tilde a\in[b_{r_n},a_q]$, because $b_{r_n}/r_n\to\infty$ and $y(\tilde a)-y_n$ is bounded below.  Hence $|T_{\tilde a,n}|=T_{\tilde a,n}$, and the rejection rule $|T_{\tilde a,n}|\ge s_n$ is equivalent to
\[
\frac{\tilde a}{r_n}+y(\tilde a)-y_n\ge s_n.
\]
The rejected set is nonempty because the endpoint $\tilde a=a_q$ is rejected.  Indeed, compactness of $I\subset(1,y_q)$ and continuity of $a(\cdot)$ give
\[
\Delta_I:=a_q-\sup_{u\in I}a(u)>0.
\]
Since $s_n\le a(y_n)/r_n-h_0$ by \eqref{eq:secondary-contradiction-sequence},
\[
\frac{a_q}{r_n}+y_q-y_n-s_n
\ge
\frac{a_q-a(y_n)}{r_n}+y_q-y_n+h_0>0
\]
for large $n$, because $a_q-a(y_n)\ge\Delta_I$ and $y_n\le\sup I<y_q$.  Define
\[
G_n(\tilde a)
:=
g_{r_n}(\tilde a)-r_n(s_n+y_n),
\quad
\tilde a\in[b_{r_n},a_q],
\]
where $g_r$ is defined in \eqref{eq:gr}.  Multiplying the rejection inequality by $r_n$, a secondary non-null with index $\tilde a$ is rejected if and only if
\[
G_n(\tilde a)\ge0.
\]
By \eqref{eq:gr} and Lemma~\ref{lem:two-block-derivatives}, $G_n$ is
continuous and strictly increasing on $[b_{r_n},a_q]$ for all large
$n$.  Therefore a secondary non-null with $a = \tilde{a}$ is rejected
if and only if $\tilde{a}\in[\beta_n,a_q]$, where
\begin{equation}
\beta_n
:=
\inf\left\{
\tilde a\in[b_{r_n},a_q]:
G_n(\tilde a)\ge0
\right\}.
\label{eq:beta-n-def}
\end{equation}
Put $d_n:=y_n-y(\beta_n)$.
If $\beta_n>b_{r_n}$, then $G_n(\beta_n)=0$.  Thus, in this case,
\begin{equation}
\frac{\beta_n}{r_n}-s_n
=d_n.
\label{eq:rejection-threshold}
\end{equation}
If $\beta_n=b_{r_n}$, then $G_n(\beta_n)\ge0$, so \eqref{eq:rejection-threshold} is replaced by
\begin{equation}
\frac{\beta_n}{r_n}-s_n
\ge
d_n.
\label{eq:rejection-threshold-endpoint}
\end{equation}
{
Moreover, $s_n\to\infty$, because
$s_n>s_{r_n}(y_n)=c_{r_n}+y(b_{r_n})-y_n$,
$c_{r_n}\to\infty$, $y(b_{r_n})\to1$, and $I$ is compact.
Also $d_n>0$ eventually.  If $\beta_n=b_{r_n}$, this follows from
the definition of $d_n$ and
\eqref{eq:primary-shift-bounds} in
Lemma~\ref{lem:primary-active-range}, which give
\[
d_n=y_n-y(b_{r_n})\ge\delta_I^{\mathrm{shift}}>0.
\]
If $\beta_n>b_{r_n}$, then
\[
g_{r_n}(\beta_n)
=r_n(s_n+y_n)
=g_{r_n}\{a(y_n)\}+r_nh_n
\le g_{r_n}\{a(y_n)\}-r_nh_0.
\]
The strict monotonicity of $g_{r_n}$ therefore gives
$\beta_n<a(y_n)$, and hence
$d_n=y_n-y(\beta_n)>0$.
}
The null contribution term in \eqref{eq:secondary-only-curve} already has a limit along this subsequence:
\begin{equation}
q\pi_{0,r_n}
\frac{R_{0,r_n,-y_n}\{p(s_n)\}}{p(s_n)}
\longrightarrow
qL_x(y)
\le qM(y)<1.
\label{eq:null-term-secondary-contradiction}
\end{equation}
Indeed, \eqref{eq:null-tail} in Lemma~\ref{lem:null-tail} applies because $r_ns_n\to x$, $y_n\to y$, and $\pi_{0,r_n}\to1$.  The inequality follows from $M(y)=\sup_{\alpha\ge0}L_\alpha(y)$, the strict monotonicity of $M$ on $(1,\infty)$, $y\in I\subset(1,y_q)$, and $qM(y_q)=1$.  Equations~\eqref{eq:secondary-only-curve} and~\eqref{eq:null-term-secondary-contradiction}, together with the threshold characterization in \eqref{eq:beta-n-def}, imply
\begin{equation}
\bar R_{r_n,-y_n}(s_n)
=
qL_x(y)+o(1)
+
\frac{q\tau_{r_n}}{p(s_n)}
\mu_{r_n}([\beta_n,a_q]).
\label{eq:secondary-sequence-curve}
\end{equation}
Moreover, \eqref{eq:mur-survival}, \eqref{eq:Wr}, \eqref{eq:deficit-properties} in Lemma~\ref{lem:envelope}, and the bound from \eqref{eq:rejection-threshold}-\eqref{eq:rejection-threshold-endpoint} give
\begin{equation}
\frac{q\tau_{r_n}}{p(s_n)}
\mu_{r_n}([\beta_n,a_q])
=
\frac{qW_{r_n}(\beta_n)}{p(s_n)}
=
\bigl[1-qM\{y(\beta_n)\}\bigr]
\frac{p(\beta_n/r_n)}{p(s_n)}
\le
\frac{p(s_n+d_n)}{p(s_n)}.
\label{eq:secondary-tail-bound}
\end{equation}
It remains only to prove that
\begin{equation}
s_nd_n\longrightarrow\infty.
\label{eq:sn-dn-diverges}
\end{equation}
Equations~\eqref{eq:secondary-tail-bound}, \eqref{eq:sn-dn-diverges}, and
\eqref{eq:variable-shift} in Lemma~\ref{lem:mills} imply
\[
\frac{q\tau_{r_n}}{p(s_n)}
\mu_{r_n}([\beta_n,a_q])
=o(1).
\]
Substituting this into \eqref{eq:secondary-sequence-curve} gives
\[
\bar R_{r_n,-y_n}(s_n)\longrightarrow qL_x(y)<1,
\]
which yields a contradiction to \eqref{eq:secondary-contradiction-sequence}.

\par\medskip\noindent
To prove \eqref{eq:sn-dn-diverges}, we now consider three cases under \eqref{eq:secondary-contradiction-sequence}, \eqref{eq:secondary-contradiction-limits}, and \eqref{eq:hn-minus-infty}.
\par\medskip\noindent

\begin{enumerate}[label=\textbf{Case (\alph*):},wide=0pt]
\item $0<x<a(y)$.
Suppose $x>0$.  Since $b_{r_n}\to0$ by \eqref{eq:br-properties} in Lemma~\ref{lem:scale-properties} and $r_ns_n\to x$, the endpoint case $\beta_n=b_{r_n}$ cannot occur for large $n$: multiplying \eqref{eq:rejection-threshold-endpoint} by $r_n$ would give
\[
b_{r_n}-r_ns_n
\ge
r_n\{y_n-y(b_{r_n})\},
\]
whose left side tends to $-x<0$, while the right side is nonnegative for large $n$.  Hence \eqref{eq:rejection-threshold} holds eventually.  Multiplying it by $r_n$ gives
\[
\beta_n-r_ns_n
=
r_nd_n.
\]
The factor $d_n$ is bounded.  Indeed, $I$ is compact, so $|y_n|\le \sup_{u\in I}|u|$.  Also $\beta_n\in[b_{r_n},a_q]$, and $y(\cdot)$ is increasing by Lemma~\ref{lem:envelope}; hence
\begin{equation}
y(b_{r_n})\le y(\beta_n)\le y(a_q)=y_q.
\label{eq:beta-y-bounds}
\end{equation}
Since $y(b_{r_n})\downarrow1$ by \eqref{eq:br-properties} in Lemma~\ref{lem:scale-properties}, the sequence $y(\beta_n)$ is bounded.  Thus $r_nd_n\to0$, so $\beta_n-r_ns_n\to0$.  Since $r_ns_n\to x$, we have $\beta_n\to x$.  Since $x<a(y)$ and $y(\cdot)$ is strictly increasing by Lemma~\ref{lem:envelope}, $y(x)<y$.  Together, $y_n\to y$ from \eqref{eq:secondary-contradiction-limits} and $\beta_n\to x$ imply, by continuity of $y(\cdot)$, that $d_n\to d:=y-y(x)>0$.  Also, \eqref{eq:rejection-threshold} gives $\beta_n/r_n=s_n+d_n$.  Therefore, for all large $n$,
\[
\frac{\beta_n}{r_n}=s_n+d_n,
\quad
0<\frac d2\le d_n\le 2d,
\quad
s_n\to\infty.
\]
As a result, \eqref{eq:sn-dn-diverges} is proved.

\item $x=0$.
Then $r_ns_n\to0$.  Since
\[
s_n>s_{r_n}(y_n)=c_{r_n}+y(b_{r_n})-y_n,
\]
\eqref{eq:cr-properties} gives $c_{r_n}\to\infty$, while \eqref{eq:br-properties} gives $y(b_{r_n})\downarrow1$ and compactness of $I$ keeps $y_n$ bounded.  Hence $s_n\to\infty$.  If $\beta_n=b_{r_n}$, then $\beta_n\to0$ by \eqref{eq:br-properties}.  If $\beta_n>b_{r_n}$, then \eqref{eq:rejection-threshold} holds; multiplying it by $r_n$, using compactness of $I$, and applying \eqref{eq:beta-y-bounds} gives
\[
\beta_n-r_ns_n=r_nd_n\to0.
\]
Since $r_ns_n\to0$, this also gives $\beta_n\to0$.  Thus $\beta_n\to0$ in both cases.  By \eqref{eq:secondary-contradiction-limits}, $y_n\to y$, and by the inverse relation together with \eqref{eq:envelope-limits} in Lemma~\ref{lem:envelope}, $y(\beta_n)\to1$.  Hence
\[
d_n\to y-1>0.
\]
Moreover, \eqref{eq:rejection-threshold} in the interior case and \eqref{eq:rejection-threshold-endpoint} in the endpoint case give $\beta_n/r_n\ge s_n+d_n$.  Since $d_n\to d:=y-1>0$ and $s_n\to\infty$, \eqref{eq:sn-dn-diverges} is proved.

\item $x=a(y)$.  By \eqref{eq:hn-minus-infty}, this is the only remaining subcase.  Since $r_ns_n\to a(y)$ and $a(y_n)\to a(y)$, we have
\[
r_nh_n=r_ns_n-a(y_n)\to a(y)-a(y)=0.
\]
The threshold equation $G_n(\tilde a)=0$ is
\[
g_{r_n}(\tilde a)
=
r_n(s_n+y_n)
=
a(y_n)+r_ny_n+r_nh_n
=
g_{r_n}\{a(y_n)\}+r_nh_n,
\]
where $y\{a(y_n)\}=y_n$.  Since $h_n<0$ and $g_{r_n}$ is strictly
increasing by \eqref{eq:gr} and Lemma~\ref{lem:two-block-derivatives}, the
corresponding threshold lies below $a(y_n)$.
To justify the existence and location of this threshold, choose
$\eta>0$ such that
\[
[a(y_n)-\eta,a(y_n)]\subset(0,a_q)
\]
for all large $n$; this is possible because $a(y_n)\to a(y)\in(0,a_q)$.
Since $-r_nh_n\to0$ and $g_{r_n}'\ge1$, for all large $n$,
\[
g_{r_n}\{a(y_n)-\eta\}
<g_{r_n}\{a(y_n)\}+r_nh_n
<g_{r_n}\{a(y_n)\}.
\]
Continuity and strict monotonicity of $g_{r_n}$ therefore give a unique
$\tilde\beta_n\in(a(y_n)-\eta,a(y_n))$ satisfying
\[
g_{r_n}(\tilde\beta_n)
=
g_{r_n}\{a(y_n)\}+r_nh_n.
\]
Moreover, the mean-value theorem and $g_{r_n}'\ge1$ give
\[
0<a(y_n)-\tilde\beta_n\le-r_nh_n\longrightarrow0.
\]
Because $a(y_n)$ remains bounded away from $0$ and $a_q$, and
$b_{r_n}\to0$ by \eqref{eq:br-properties} in
Lemma~\ref{lem:scale-properties}, it follows that
\[
b_{r_n}<\tilde\beta_n<a_q
\]
eventually.  Thus $\tilde\beta_n$ is the threshold $\beta_n$ in
\eqref{eq:beta-n-def}, so $G_n(\beta_n)=0$.  The mean-value theorem applied
to $g_{r_n}$ now gives, for some $\xi_n$ between $\beta_n$ and
$a(y_n)$,
\[
-r_nh_n
=
g_{r_n}\{a(y_n)\}-g_{r_n}(\beta_n)
=
(a(y_n)-\beta_n)\{1+r_ny'(\xi_n)\}.
\]
Therefore $\beta_n<a(y_n)$ and
\begin{equation}
a(y_n)-\beta_n
=
\frac{-r_nh_n}{1+r_ny'(\xi_n)}.
\label{eq:case-c-beta-gap}
\end{equation}
For some $\zeta_n$ between $\beta_n$ and $a(y_n)$,
\begin{equation}
\frac{\beta_n}{r_n}-s_n
=d_n
=y'(\zeta_n)\{a(y_n)-\beta_n\}.
\label{eq:case-c-dn-gap}
\end{equation}
Combining \eqref{eq:case-c-beta-gap} and~\eqref{eq:case-c-dn-gap} gives
\begin{equation}
s_nd_n
=
(r_ns_n)
\frac{y'(\zeta_n)}{1+r_ny'(\xi_n)}
(-h_n).
\label{eq:case-c-sn-dn-product}
\end{equation}
Here $r_ns_n\to a(y)>0$.  Also $\beta_n\to a(y)$ and $a(y_n)\to a(y)$, so
$\xi_n,\zeta_n\to a(y)$; hence $y'(\zeta_n)\to y'\{a(y)\}>0$, while
$1+r_ny'(\xi_n)\to1$.  Since $h_n\to-\infty$ by \eqref{eq:hn-minus-infty}, \eqref{eq:case-c-sn-dn-product} yields
\[
s_nd_n\longrightarrow\infty.
\]
Thus \eqref{eq:sn-dn-diverges} is proved.
\end{enumerate}

This proves \eqref{eq:secondary-range-target}.  Combining \eqref{eq:primary-active-uniform} and \eqref{eq:secondary-range-target}, and reducing $\gamma$ if necessary, proves \eqref{eq:no-earlier}.

\par\medskip\noindent
\textbf{Proof of \eqref{eq:rare-result}.}
{
Apply Lemma~\ref{lem:localized-crossing} to $\bar R_{r,-y}$, with
\[
 m_r(y)=\frac{a(y)}r,
 \qquad w_r(y)=1,
 \qquad d(y)=qM(y),
 \qquad \xi(y)=\kappa(y).
\]
The set $I$ is compact, and $a_I=\inf_{y\in I}a(y)>0$; hence, for every
fixed $H>0$,
\[
 \inf_{y\in I}\left\{\frac{a(y)}r-H\right\}
 \ge\frac{a_I}{r}-H>u_q
\]
for all sufficiently small $r$.  Because $I\subset(1,y_q)$ and $M$
is continuous and strictly increasing there,
\[
 0<d(y)=qM(y)<1,\qquad \xi(y)=\kappa(y)>0,
\]
and both functions are continuous on $I$.  Equation~\eqref{eq:rare-local}
is exactly \eqref{eq:localized-crossing-full}, while
\eqref{eq:no-earlier} is \eqref{eq:localized-crossing-left}.  Thus all
hypotheses of Lemma~\ref{lem:localized-crossing} hold, and
\eqref{eq:localized-crossing-cutoff} gives \eqref{eq:rare-result} uniformly
on $I$.}

\par\medskip\noindent
\textbf{Proof of \eqref{eq:rare-fdp-result}.}
{
Fix $H=1$.}
For
$y\in I$ and $|h|\le H$, let $\tilde{\alpha}_r(y,h)$ be the threshold
defined by \eqref{eq:tilde-alpha-equation}.  Write
\[
\bar R_{2,r,-y}\left(\frac{a(y)}r+h\right)
=\frac{qW_r\{\tilde{\alpha}_r(y,h)\}}{p(a(y)/r+h)}
=\bigl[1-qM\{y(\tilde{\alpha}_r(y,h))\}\bigr]
\frac{p\{\tilde{\alpha}_r(y,h)/r\}}{p(a(y)/r+h)}.
\]
Equation~\eqref{eq:rare-primary-scale-separation} shows that the
primary non-null is not rejected for $y\in I$ and $|h|\le H$.
Consequently, \eqref{eq:population-curve} gives
\[
\bar R_{r,-y}\left(\frac{a(y)}r+h\right)
=\bar R_{0,r,-y}\left(\frac{a(y)}r+h\right)
 +\bar R_{2,r,-y}\left(\frac{a(y)}r+h\right).
\]

We first differentiate the secondary non-null contribution
$\bar R_{2,r,-y}\{a(y)/r+h\}$.  Equation
\eqref{eq:two-secondary-lower-tail} holds uniformly, since
\[
 g_r(b_r)>0>ry-a(y)-rh
 =r\left\{y-\left(\frac{a(y)}r+h\right)\right\}
\]
for all sufficiently small $r$.
{
Apply \eqref{eq:two-secondary-block-log-derivative} in
Lemma~\ref{lem:two-block-derivatives} with
$c=a(y)/r+h$ and $\beta=\tilde{\alpha}_r(y,h)$.
The term in \eqref{eq:two-secondary-block-log-derivative},
\[
 -\frac{qM'\{y(\beta)\}y'(\beta)}
 {1-qM\{y(\beta)\}}
 \frac{r}{1+r y'(\beta)},
\]
is $O(r)$ uniformly by \eqref{eq:alpha-close} and
\eqref{eq:deficit-properties} in Lemma~\ref{lem:envelope}.  Indeed,
\eqref{eq:alpha-close} places $\beta$, uniformly over
$y\in I$ and $|h|\le H$, in a fixed compact subinterval of
$(0,a_q)$.  On this subinterval,
\[
 \frac{|M'\{y(\beta)\}|y'(\beta)}
 {1-qM\{y(\beta)\}}=O(1),
 \qquad
 \frac{r}{1+r y'(\beta)}=O(r),
\]
by \eqref{eq:deficit-properties} in Lemma~\ref{lem:envelope}; hence
their product is $O(r)$
uniformly.

The last two terms in \eqref{eq:two-secondary-block-log-derivative} are
\[
 -\frac{\lambda(\beta/r)}{1+r y'(\beta)}+\lambda(c).
\]
For these terms, \eqref{eq:two-secondary-threshold-equation} implies
\[
 \frac{\beta}{r}=c+y-y(\beta).
\]
Because $a(y)$ is bounded away from zero on $I$,
\eqref{eq:alpha-close} gives $c\asymp\beta/r\asymp r^{-1}$
uniformly.  Thus \eqref{eq:mills-hazard} in Lemma~\ref{lem:mills}
gives
\[
\begin{aligned}
 \lambda(c)-\frac{\lambda(\beta/r)}{1+r y'(\beta)}
 &=
 c-\frac{\beta/r}{1+r y'(\beta)}+O(r)\\
 &=
 \frac{rcy'(\beta)+y(\beta)-y}
 {1+r y'(\beta)}+O(r).
\end{aligned}
\]
Now $rc=a(y)+rh\to a(y)$, while
\eqref{eq:alpha-close} and continuity of $y$ and $y'$ give
\[
 y(\beta)-y=o(1),
 \qquad
 y'(\beta)=y'\{a(y)\}+o(1)
\]
uniformly.  Consequently,
\[
 \lambda(c)-\frac{\lambda(\beta/r)}{1+r y'(\beta)}
 \longrightarrow a(y)y'\{a(y)\}=\kappa(y)
\]
uniformly, by \eqref{eq:kappa}.  Combining this with
\eqref{eq:signal-local} yields
\begin{equation}
\frac{\partial}{\partial h}
\bar R_{2,r,-y}\left(\frac{a(y)}r+h\right)
\longrightarrow
\kappa(y)\{1-qM(y)\}e^{\kappa(y)h}
\label{eq:rare-signal-derivative}
\end{equation}
uniformly for $y\in I$ and $|h|\le H$.}

{
It remains to differentiate the null contribution.  Use the notation of
Lemma~\ref{lem:two-block-derivatives} at $c=a(y)/r+h$.  Since $a(y)$
is bounded away from zero on $I$ and
$a(y)\le a_q$, for all sufficiently small $r$,
\[
u_q\le \frac{a(y)}r+h\le \frac{a_q+1}{r}
\]
uniformly for $y\in I$ and $|h|\le H$.  Applying
\eqref{eq:null-tail} in Lemma~\ref{lem:null-tail} with
$A=a_q+1$ and $T=\sup_{u\in I}u$ shows that
\begin{equation}
 \sup_{y\in I}\sup_{|h|\le H}
 \bar R_{0,r,-y}\left(\frac{a(y)}r+h\right)=O(1).
\label{eq:rare-null-level-bound}
\end{equation}
Equation~\eqref{eq:two-null-block-log-derivative} in
Lemma~\ref{lem:two-block-derivatives} can be written as
\[
 \frac{\partial}{\partial h}\log\bar R_{0,r,-y}(c)
 =\sum_{\sigma\in\{-,+\}}
 \frac{Q_\sigma}{Q_++Q_-}
 \left\{\lambda(c)-\frac{\lambda(x_\sigma)}{\sqrt{1-r^2}}\right\},
\]
because $\partial c/\partial h=1$.  Moreover,
\eqref{eq:mills-hazard} in Lemma~\ref{lem:mills} gives, for either sign,
\[
\lambda(c)
-\frac{\lambda(x_\pm)}{\sqrt{1-r^2}}
=c-\frac{c\pm ry}{1-r^2}
+O(r)
=O(r)
\]
uniformly for $y\in I$ and $|h|\le H$.  Since the nonnegative weights
$Q_\sigma/(Q_++Q_-)$ sum to one,
\begin{equation}
 \sup_{y\in I}\sup_{|h|\le H}
 \left|\frac{\partial}{\partial h}
 \log\bar R_{0,r,-y}\left(\frac{a(y)}r+h\right)\right|=O(r).
\label{eq:rare-null-log-derivative-bound}
\end{equation}
Multiplying \eqref{eq:rare-null-level-bound} and
\eqref{eq:rare-null-log-derivative-bound} gives
\begin{equation}
\sup_{y\in I}\sup_{|h|\le H}
\left|
\frac{\partial}{\partial h}
\bar R_{0,r,-y}\left(\frac{a(y)}r+h\right)
\right|
\longrightarrow0.
\label{eq:rare-null-derivative}
\end{equation}}
Since $I$ is compact and
$\kappa(y)\{1-qM(y)\}e^{\kappa(y)h}$ is strictly positive on
$I\times[-H,H]$, \eqref{eq:rare-signal-derivative} and~\eqref{eq:rare-null-derivative} imply
\begin{equation}
\inf_{y\in I}\inf_{|h|\le H}
\frac{\partial}{\partial h}
\bar R_{r,-y}\left(\frac{a(y)}r+h\right)>0
\label{eq:rare-strict-increase}
\end{equation}
for all sufficiently small $r$.
For every fixed $H>0$, \eqref{eq:rare-null-local} verifies
\eqref{eq:localized-crossing-null} with $d(y)=qM(y)$.  The null
contribution is smooth, and \eqref{eq:rare-strict-increase}, with $H=1$,
gives the positive derivative required by
Lemma~\ref{lem:localized-crossing-regularity}.  That lemma
therefore shows that $C_r(-y)$ is regular for every $y\in I$ when $r$
is sufficiently small, and \eqref{eq:localized-crossing-fdp} gives
\[
 \sup_{y\in I}|D_r(-y)-qM(y)|\longrightarrow0,
\]
which is \eqref{eq:rare-fdp-result}.
\end{proof}

\subsubsection{\texorpdfstring{The regime $z<-y_q$}{The regime z < -yq}}

Let $y=-z$.
For $y>y_q$, let $\alpha_q(y)\in(0,a_q)$ be the unique solution of
\begin{equation}
qL_{\alpha_q(y)}(y)=1.
\label{eq:alpha}
\end{equation}
For fixed $y$, the map $a\mapsto L_a(y)$ is continuous and
\begin{equation}
\frac{\partial}{\partial a}\log L_a(y)=-a+y\tanh(ay).
\label{eq:La-log-derivative}
\end{equation}
By Lemma~\ref{lem:envelope}, this derivative is positive on $(0,a(y))$.  Since $y>y_q$, the same lemma gives $a(y)>a_q$, so $a\mapsto L_a(y)$ is strictly increasing on $[0,a_q]$.  Moreover, by the definition of $a_q$ in \eqref{eq:aq-def},
\[
qL_0(y)=q<1,
\quad
qL_{a_q}(y)>qL_{a_q}(y_q)=qM(y_q)=1,
\]
where the strict inequality follows from $a_q>0$, $y>y_q$, and the monotonicity of $\cosh(a_q y)$ in $y>0$.  Thus continuity and strict monotonicity give exactly one solution to \eqref{eq:alpha}.

\begin{lemma}
\label{lem:null-only}
For every compact $I\subset(y_q,\infty)$, uniformly for $y\in I$,
\begin{equation}
rC_r(-y)\longrightarrow\alpha_q(y)
\label{eq:null-only-result}
\end{equation}
and
\begin{equation}
D_r(-y)\longrightarrow1.
\label{eq:null-only-fdp-result}
\end{equation}
\end{lemma}

\begin{proof}
\textbf{Proof of \eqref{eq:null-only-result}.}
Fix a compact $I\subset(y_q,\infty)$, and put
\[
\delta_I=\inf_{y\in I}(y-y_q)>0.
\]
{
Use $s_r(y)$ from \eqref{eq:sr-def} and
$(\mathcal A_{r,y,s},\beta_{r,y,s})$ from
\eqref{eq:beta-rys-def}.  By \eqref{eq:gr} and
Lemma~\ref{lem:two-block-derivatives},
$g_r'(\tilde a)=1+r y'(\tilde a)>1$.
For all sufficiently small $r$, the primary non-null is not rejected
when $s>s_r(y)$.  To verify the corresponding sign assertion for the
secondary non-nulls, note that, conditional on $Z=-y$, the value associated
with $\tilde a\in[b_r,a_q]$ is
{
\begin{equation}
\frac{\tilde a}{r}+y(\tilde a)-y
\ge
\frac{b_r}{r}+y(b_r)-\sup_{u\in I}u.
\label{eq:null-only-secondary-positive-bound}
\end{equation}}
The monotonicity of $y(\cdot)$ gives
\eqref{eq:null-only-secondary-positive-bound}.  Its right-hand side tends to
infinity as $r\downarrow0$, by \eqref{eq:br-properties} in
Lemma~\ref{lem:scale-properties}.  Hence every
secondary non-null value is positive for all sufficiently small $r$,
uniformly over $y\in I$ and $\tilde a\in[b_r,a_q]$.  Thus the rejected
secondary non-nulls are indexed by $\mathcal A_{r,y,s}$.  If this
set is nonempty, continuity and strict monotonicity of $g_r$ show that it is
the closed interval $[\beta_{r,y,s},a_q]$ and that
\[
g_r(\beta_{r,y,s})\ge r(s+y).
\]
Consequently, in the nonempty case,
\begin{equation}
\frac{\beta_{r,y,s}}r-s
\ge y-y(\beta_{r,y,s})
\ge y-y_q
\ge\delta_I.
\label{eq:null-only-threshold-gap}
\end{equation}
If $\mathcal A_{r,y,s}=\varnothing$, its $\mu_r$-mass is zero.
Since $\mu_r(\{a_q\})=0$ by \eqref{eq:mur-survival}, the convention in
\eqref{eq:beta-rys-def} gives in both cases
\[
\mu_r(\mathcal A_{r,y,s})
=\mu_r([\beta_{r,y,s},a_q]).
\]
Using \eqref{eq:mur-survival}, \eqref{eq:deficit-properties} in
Lemma~\ref{lem:envelope}, and \eqref{eq:null-only-threshold-gap} when the set
is nonempty, we obtain
\[
0\le
\frac{q\tau_r}{p(s)}\mu_r([\beta_{r,y,s},a_q])
\le
\frac{p(s+\delta_I)}{p(s)}.
\]
Moreover,
\[
\inf_{y\in I}s_r(y)
=c_r+y(b_r)-\sup_{y\in I}y
\longrightarrow\infty,
\]
because $c_r\to\infty$, $y(b_r)\to1$, and $I$ is compact.
Therefore \eqref{eq:fixed-shift} in Lemma~\ref{lem:mills} gives
\begin{equation}
\sup_{y\in I}
\sup_{s>s_r(y)}
\frac{q\tau_r}{p(s)}
\mu_r([\beta_{r,y,s},a_q])
\longrightarrow0.
\label{eq:null-only-signal-small}
\end{equation}
}

The function $(a,y)\mapsto L_a(y)$ is $C^\infty$.
Equation~\eqref{eq:La-log-derivative}, together with
$\alpha_q(y)<a_q<a(y)$ and Lemma~\ref{lem:envelope}, shows that
\[
\frac{\partial}{\partial a}\log L_a(y)
\bigg|_{a=\alpha_q(y)}>0.
\]
The implicit function theorem therefore shows that
$y\mapsto\alpha_q(y)$ is continuous on $(y_q,\infty)$.  Since $I$ is
compact,
\begin{equation}
m_I
:=
\inf_{y\in I}
\min\{\alpha_q(y),a_q-\alpha_q(y)\}>0.
\label{eq:mI-def}
\end{equation}
Fix any $0<\delta<m_I/2$.  For every $y\in I$, continuity and strict
monotonicity in $a$ give
\[
qL_{\alpha_q(y)-\delta}(y)
<qL_{\alpha_q(y)}(y)=1
<qL_{\alpha_q(y)+\delta}(y).
\]
Both strict gaps are continuous functions of $y$.  Their minima over the
compact set $I$ are positive, so there is $\gamma>0$ such that, uniformly
on $I$,
\begin{equation}
qL_{\alpha_q(y)-\delta}(y)\le1-2\gamma,
\quad
qL_{\alpha_q(y)+\delta}(y)\ge1+2\gamma
\label{eq:null-only-margins}
\end{equation}
The choice $\delta<m_I/2$ also gives
\begin{equation}
\alpha_q(y)+\delta
\le \alpha_q(y)+\frac{m_I}{2}
{\le\frac{\alpha_q(y)+a_q}{2}<a_q<a(y)},
\label{eq:alphaq-delta-left}
\end{equation}
uniformly for $y\in I$.

Write $\psi_y(x)=\log L_x(y)$.  The proof of
Lemma~\ref{lem:envelope}, through \eqref{eq:gprime}, shows that
$\psi_y'(x)>0$ for $0<x<a(y)$.  By
\eqref{eq:alphaq-delta-left}, every
$x\in(0,\alpha_q(y)+\delta)$ lies in $(0,a(y))$.  Hence
$\partial_x L_x(y)={L_x(y)\psi_y'(x)}>0$ throughout
$(0,\alpha_q(y)+\delta)$, and therefore $x\mapsto L_x(y)$ is
increasing on $[0,\alpha_q(y)+\delta]$.

On $u_q\le s\le s_r(y)$,
\eqref{eq:primary-active-uniform} in
Lemma~\ref{lem:primary-active-range} gives
\[
 \sup_{y\in I}\sup_{u_q\le s\le s_r(y)}
 \bar R_{r,-y}(s)\le q+o(1).
\]

For $s>s_r(y)$, the term in \eqref{eq:population-curve} that corresponds to the primary non-nulls is zero, and \eqref{eq:null-only-signal-small} and \eqref{eq:null-tail} in Lemma~\ref{lem:null-tail} give, uniformly for $y\in I$ and $s_r(y)<s\le(\alpha_q(y)+\delta)/r$,
\[
\bar R_{r,-y}(s)
=
\bar R_{0,r,-y}(s)
+
\frac{q\tau_r}{p(s)}
\mu_r([\beta_{r,y,s},a_q])
=
qL_{rs}(y)+o(1).
\]
Combining these two ranges gives
\begin{equation}
\sup_{u_q\le s\le(\alpha_q(y)-\delta)/r}
\bar R_{r,-y}(s)
\le
qL_{\alpha_q(y)-\delta}(y)+o(1)
\le1-\gamma.
\label{eq:null-only-left-margin}
\end{equation}
Since $rs_r(y)\to0$ uniformly on $I$ and
$\alpha_q(y)+\delta$ is bounded away from zero, we have
\[
\frac{\alpha_q(y)+\delta}{r}>s_r(y)
\]
for all sufficiently small $r$, uniformly on $I$.  At the cutoff
$(\alpha_q(y)+\delta)/r$, the term in \eqref{eq:population-curve} that
corresponds to the primary non-nulls is zero, while the secondary term is
nonnegative.  Therefore \eqref{eq:null-tail} in Lemma~\ref{lem:null-tail}
gives, uniformly for $y\in I$,
\begin{equation}
\bar R_{r,-y}\left(\frac{\alpha_q(y)+\delta}{r}\right)
\ge
\bar R_{0,r,-y}\left(\frac{\alpha_q(y)+\delta}{r}\right)
=
qL_{\alpha_q(y)+\delta}(y)+o(1).
\label{eq:null-only-right-margin}
\end{equation}
By \eqref{eq:null-only-margins}, the right side is at least $1+\gamma$
for all sufficiently small $r$.  Hence
\[
\alpha_q(y)-\delta<rC_r(-y)\le\alpha_q(y)+\delta.
\]
Letting $\delta\downarrow0$ gives
\[
rC_r(-y)\longrightarrow\alpha_q(y),
\]
uniformly for $y\in I$.

\par\medskip\noindent
\textbf{Proof of \eqref{eq:null-only-fdp-result}.}
We now verify regularity for each fixed sufficiently small $r$.  Fix one
$0<\delta<m_I/2$ as in \eqref{eq:mI-def} and consider the interval
\[
\frac{\alpha_q(y)-\delta}{r}
\le s\le
\frac{\alpha_q(y)+\delta}{r}.
\]
The primary non-null has conditional value
\[
T_i
=
c_r+y(b_r)-y
=s_r(y).
\]
It is positive for sufficiently small $r$, uniformly in $y\in I$, and it
is rejected at cutoff $s$ exactly when $s\le s_r(y)$.  Since
$\delta<m_I/2$ and $m_I\le\alpha_q(y)$, we have
$\alpha_q(y)-\delta\ge m_I/2$.  On the other hand,
\[
r s_r(y)
=
r c_r+r\{y(b_r)-y\}
\longrightarrow0
\]
uniformly for $y\in I$, by \eqref{eq:cr-properties}, $y(b_r)\to1$, and
compactness of $I$.  Therefore, for all sufficiently small $r$, for every
$y\in I$ and every cutoff $s$ satisfying
\[
\frac{\alpha_q(y)-\delta}{r}
\le s\le
\frac{\alpha_q(y)+\delta}{r},
\]
we have
\[
s_r(y)
<
\frac{\alpha_q(y)-\delta}{r}
\le s.
\]
{
Thus every cutoff in this interval lies strictly to the right of the primary
value $s_r(y)$, so the primary non-null is not rejected there.  Moreover,
uniformly for these $y$ and $s$,
\[
\begin{aligned}
r(s+y)-g_r(b_r)
&\ge
\alpha_q(y)-\delta+r\{y-y(b_r)\}-b_r
\ge\frac{m_I}{2}-o(1)>0,\\
g_r(a_q)-r(s+y)
&\ge
a_q-\alpha_q(y)-\delta+r(y_q-y)
\ge\frac{m_I}{2}-o(1)>0.
\end{aligned}
\]
Strict monotonicity of $g_r$ therefore implies that
$\mathcal A_{r,y,s}\ne\varnothing$, that
$\beta_{r,y,s}\in(b_r,a_q)$.
{
By \eqref{eq:beta-rys-def}, continuity and strict monotonicity of
$g_r$ give
\[
 g_r(\beta_{r,y,s})=r(s+y).
\]
Using $g_r(a)=a+ry(a)$ from \eqref{eq:gr}, this equality yields
\[
\beta_{r,y,s}
=rs+r\{y-y(\beta_{r,y,s})\}
=rs+O(r)
\]
uniformly.}  Hence, for all sufficiently small $r$,
\[
\frac{m_I}{4}\le\beta_{r,y,s}\le a_q-\frac{m_I}{4},
\]
so the thresholds remain in a fixed compact subinterval of
$(0,a_q)$.
}

For $y\in I$ and cutoffs $s$ in the local crossing interval
\[
\frac{\alpha_q(y)-\delta}{r}
\le s\le
\frac{\alpha_q(y)+\delta}{r},
\]
write the secondary non-null contribution as
\[
\bar R_{2,r,-y}(s)
=\bigl[1-qM\{y(\beta_{r,y,s})\}\bigr]
\frac{p(\beta_{r,y,s}/r)}{p(s)}.
\]
Then
\[
 \bar R_{r,-y}(s)=\bar R_{0,r,-y}(s)+\bar R_{2,r,-y}(s).
\]
Abbreviate $\beta=\beta_{r,y,s}$.
Equation~\eqref{eq:two-secondary-lower-tail} holds uniformly because
$g_r(b_r)>0>r(y-s)$ for all sufficiently small $r$.
{
The term in
\eqref{eq:two-secondary-block-log-derivative} of
Lemma~\ref{lem:two-block-derivatives}, evaluated at
$c=s$ and $\beta_{r,y}(s)=\beta_{r,y,s}=\beta$, is
\[
 -\frac{qM'\{y(\beta)\}y'(\beta)}
 {1-qM\{y(\beta)\}}
 \frac{r}{1+r y'(\beta)}.
\]
Because
\[
 \frac{m_I}{4}\le\beta\le a_q-\frac{m_I}{4},
\]
\eqref{eq:deficit-properties} in Lemma~\ref{lem:envelope} and continuity
give a constant $C_I<\infty$ such that
\[
 \sup_{m_I/4\le a\le a_q-m_I/4}
 \left|
 \frac{qM'\{y(a)\}y'(a)}{1-qM\{y(a)\}}
 \right|\le C_I.
\]
Since $y'(\beta)>0$,
\[
 0<\frac{r}{1+r y'(\beta)}\le r,
\]
and therefore
\[
 \left|
 \frac{qM'\{y(\beta)\}y'(\beta)}
 {1-qM\{y(\beta)\}}
 \frac{r}{1+r y'(\beta)}
 \right|\le C_Ir.
\]

The expressions in
\eqref{eq:two-secondary-block-log-derivative} are
\[
 -\frac{\lambda(\beta/r)}{1+r y'(\beta)}
 \qquad\text{and}\qquad
 \lambda(s).
\]
To bound their sum, write
\[
 \lambda(t)=t+\varepsilon(t),
 \qquad 0<\varepsilon(t)<t^{-1},
\]
by \eqref{eq:mills-hazard} in Lemma~\ref{lem:mills}.  On the cutoff range,
\[
 s\ge\frac{m_I}{2r},
 \qquad
 \frac\beta r\ge\frac{m_I}{4r},
\]
and hence $\varepsilon(s)=O(r)$ and
$\varepsilon(\beta/r)=O(r)$ uniformly.  Using
$\beta/r=s+y-y(\beta)$, we obtain
\[
\begin{aligned}
 \lambda(s)-\frac{\lambda(\beta/r)}{1+r y'(\beta)}
 &=s-\frac{\beta/r}{1+r y'(\beta)}+O(r)\\
 &=\frac{rs y'(\beta)+y(\beta)-y}
 {1+r y'(\beta)}+O(r)
 =O(1),
\end{aligned}
\]
uniformly, because $rs$, $y$, $y(\beta)$, and $y'(\beta)$ are
uniformly bounded.  Hence
\begin{equation}
\sup_{y\in I}
\sup_{|rs-\alpha_q(y)|\le\delta}
\left|
\frac{\partial}{\partial s}\log\bar R_{2,r,-y}(s)
\right|
=O(1).
\label{eq:null-only-signal-log-bound}
\end{equation}}
Moreover, $\beta_{r,y,s}/r-s\ge\delta_I$ by
\eqref{eq:null-only-threshold-gap}, so
$\beta_{r,y,s}/r\ge s+\delta_I$.  Since $p$ is decreasing on
$[0,\infty)$, and since \eqref{eq:deficit-properties} in
Lemma~\ref{lem:envelope} gives
$0\le1-qM\{y(\beta_{r,y,s})\}\le1$, we have
\[
0\le\bar R_{2,r,-y}(s)
\le
\frac{p(s+\delta_I)}{p(s)}.
\]
Applying \eqref{eq:mills-shift-upper} in Lemma~\ref{lem:mills} with
$x=s$ and $\delta=\delta_I$ gives
\begin{equation}
0\le\bar R_{2,r,-y}(s)
\le
\frac{p(s+\delta_I)}{p(s)}
\le
(1+s^{-2})
\exp\left\{-s\delta_I-\frac{\delta_I^2}{2}\right\}
=o(r),
\label{eq:null-only-signal-size}
\end{equation}
uniformly on the interval.  The last $o(r)$ follows because
\eqref{eq:mI-def} and
$0<\delta<m_I/2$ imply
\[
s
\ge
\frac{\alpha_q(y)-\delta}{r}
\ge
\frac{m_I}{2r}
\]
whenever $y\in I$ and $|rs-\alpha_q(y)|\le\delta$, and therefore
\[
r^{-1}(1+s^{-2})
\exp\left\{-s\delta_I-\frac{\delta_I^2}{2}\right\}
\le
r^{-1}(1+o(1))
\exp\left\{-\frac{m_I\delta_I}{2r}-\frac{\delta_I^2}{2}\right\}
\longrightarrow0
\]
uniformly.  Since
$\partial_s\bar R_{2,r,-y}(s)
=\bar R_{2,r,-y}(s)\partial_s\log\bar R_{2,r,-y}(s)$,
\eqref{eq:null-only-signal-log-bound} and~\eqref{eq:null-only-signal-size} imply
\begin{equation}
\sup_{y\in I}
\sup_{|rs-\alpha_q(y)|\le\delta}
\left|\frac{\partial}{\partial s}\bar R_{2,r,-y}(s)\right|
=o(r).
\label{eq:null-only-signal-derivative}
\end{equation}

{
For the null contribution term, put
\[
 x_+=\frac{s+ry}{\sqrt{1-r^2}},
 \qquad
 x_-=\frac{s-ry}{\sqrt{1-r^2}},
 \qquad
 Q_\pm=\frac{\Phib(x_\pm)}{\Phib(s)}.
\]
Because $\alpha_q(y)-\delta\ge m_I/2$, the range in
\eqref{eq:null-only-signal-log-bound} implies
\[
 s\ge\frac{m_I}{2r}\longrightarrow\infty,
\]
while $rs$ and $y$ range over compact sets.  Therefore
\eqref{eq:null-tail-expansion} in Lemma~\ref{lem:null-tail}, with
$z=-y$, gives uniformly
\[
 Q_+
 =e^{-(rs)^2/2-rsy}+o(1),
 \qquad
 Q_-
 =e^{-(rs)^2/2+rsy}+o(1).
\]
Indeed, $s/x_\pm=1+o(1)$, the $O(r^2)$ term in
\eqref{eq:null-tail-expansion} is uniform, and its remainder tends to zero
because $s\to\infty$ uniformly.

Next, \eqref{eq:mills-hazard-expansion} in Lemma~\ref{lem:mills} gives
\begin{equation}
\begin{aligned}
 \frac1r\left\{\lambda(s)
 -\frac{\lambda(x_+)}{\sqrt{1-r^2}}\right\}
 &=-rs-y+o(1),\\
 \frac1r\left\{\lambda(s)
 -\frac{\lambda(x_-)}{\sqrt{1-r^2}}\right\}
 &=-rs+y+o(1).
\end{aligned}
\label{eq:null-only-scaled-hazards}
\end{equation}
More precisely, for either sign,
\[
 \lambda(s)-\frac{\lambda(x_\pm)}{\sqrt{1-r^2}}
 =s-\frac{s\pm ry}{1-r^2}
 +\frac1s-\frac1{s\pm ry}+O(r^3),
\]
where
\[
 \frac1s-\frac1{s\pm ry}=O(r^3),
 \qquad
 \frac1r\left\{s-\frac{s\pm ry}{1-r^2}\right\}
 =\frac{-rs\mp y}{1-r^2}.
\]
This proves \eqref{eq:null-only-scaled-hazards}.

Substitution of these four expansions into
\eqref{eq:two-null-block-derivative} in
Lemma~\ref{lem:two-block-derivatives}, together with
$\pi_{0,r}\to1$, yields
\[
\begin{aligned}
 \frac1r\frac{\partial}{\partial s}\bar R_{0,r,-y}(s)
 &=\frac q2 e^{-(rs)^2/2}
 \left[e^{-rsy}(-rs-y)+e^{rsy}(-rs+y)\right]+o(1)\\
 &=q e^{-(rs)^2/2}
 \{-rs\cosh(rsy)+y\sinh(rsy)\}+o(1).
\end{aligned}
\]
Finally, differentiating
$L_a(y)=e^{-a^2/2}\cosh(ay)$ gives
\[
 \frac{\partial}{\partial a}L_a(y)
 =e^{-a^2/2}\{-a\cosh(ay)+y\sinh(ay)\}.
\]
Evaluating at $a=rs$ proves
\begin{equation}
\frac1r\frac{\partial}{\partial s}\bar R_{0,r,-y}(s)
=
q e^{-(rs)^2/2}
\{-rs\cosh(rsy)+y\sinh(rsy)\}+o(1)
=q\frac{\partial}{\partial a}L_a(y)\bigg|_{a=rs}+o(1).
\label{eq:null-only-null-derivative}
\end{equation}}
{
By \eqref{eq:mI-def} and \eqref{eq:alphaq-delta-left}, the pairs
$(rs,y)$ under consideration range
over a compact subset of
$\{(a,y):y\in I,\ 0<a<a(y)\}$.  Continuity,
\eqref{eq:La-log-derivative}, and Lemma~\ref{lem:envelope} therefore show that
$\partial_aL_a(y)|_{a=rs}$ is bounded away from zero uniformly.}
Combining
\eqref{eq:null-only-signal-derivative} and~\eqref{eq:null-only-null-derivative} yields
\begin{equation}
\inf_{y\in I}
\inf_{|rs-\alpha_q(y)|\le\delta}
\frac{\partial}{\partial s}\bar R_{r,-y}(s)>0
\label{eq:null-only-strict-increase}
\end{equation}
for all sufficiently small $r$.  Equation~\eqref{eq:null-only-result}
gives, uniformly for $y\in I$,
\[
 \frac{\alpha_q(y)-\delta}{r}
 <C_r(-y)<
 \frac{\alpha_q(y)+\delta}{r}
\]
for all sufficiently small $r$.  Equation
\eqref{eq:null-only-strict-increase} and the smoothness of the null
contribution for
$|rs-\alpha_q(y)|\le\delta$
verify the hypotheses of Proposition~\ref{prop:local-regularity}; hence
$C_r(-y)$ is regular.

Equation~\eqref{eq:Dr} gives
\[
 D_r(-y)=\bar R_{0,r,-y}\{C_r(-y)\}.
\]
Because $rC_r(-y)\to\alpha_q(y)$ uniformly and
$\alpha_q(y)$ stays in a compact subset of $(0,a_q)$,
\eqref{eq:null-tail} in Lemma~\ref{lem:null-tail} applies at
$c=C_r(-y)$ and gives
\[
D_r(-y)
=qL_{rC_r(-y)}(y)+o(1)
\longrightarrow qL_{\alpha_q(y)}(y)=1,
\]
uniformly for $y\in I$, where the equality follows from
\eqref{eq:alpha}.  This proves \eqref{eq:null-only-fdp-result}.
\end{proof}

\subsubsection{\texorpdfstring{Conditional limit of $D_r$}{Conditional limit of Dr}}

The following lemma establishes the limit of $D_r(z)$ as $r \downarrow 0$, with $D_*$ defined in Section~\ref{sec:roadmap}.
\begin{lemma}
\label{lem:continuum-limit}
For every $z\notin\{-1,-y_q\}$,
\begin{equation}
D_r(z)\longrightarrow D_*(z).
\label{eq:pointwise-D}
\end{equation}
Consequently,
\begin{equation}
\mathbb E[D_r(Z)]\longrightarrow\mathbb E[D_*(Z)]=\ell_{=}(q).
\label{eq:ED-limit}
\end{equation}
\end{lemma}

\begin{proof}
If $z>-1$, Lemma~\ref{lem:primary} gives $D_r(z)\to q$.  If $-y_q<z<-1$, put $y=-z$ and apply Lemma~\ref{lem:rare-local}.  If $z<-y_q$, apply Lemma~\ref{lem:null-only}.  This proves \eqref{eq:pointwise-D}.

At every regular crossing, $0\le D_r\le1$; the value assigned at nonregular points also lies in $[0,1]$.  Dominated convergence gives the convergence in \eqref{eq:ED-limit}.  Finally,
\begin{align*}
\mathbb E[D_*(Z)]
&=
q\mathbb P(Z>-1)
+q\int_{-y_q}^{-1}M(-z)\phi(z)\dd z
+\mathbb P(Z<-y_q)\\
&=
q\Phi(1)
+q\int_1^{y_q}M(y)\phi(y)\dd y
+\Phib(y_q)\\
&=\ell_{=}(q).
\end{align*}
\end{proof}

\subsection{Proofs of Theorem~\ref{thm:main} and
Propositions~\ref{prop:ell-expansion} and~\ref{prop:ell-strict}}

\begin{proof}[Proof of Theorem~\ref{thm:main}]
Lemma~\ref{lem:continuum-limit} gives
$\mathbb E\{D_*(Z)\}=\ell_{=}(q)$.  Fix $\delta>0$.  Choose a compact set
\[
E\subset\R\setminus\{-1,-y_q\}
\]
such that
\[
\int_E D_*(z)\phi(z)\dd z>{\ell_{=}}(q)-\frac{\delta}{4}.
\]
Because $E$ is compact and separated from $-1$ and
$-y_q$, the sets
\[
E_+=E\cap(-1,\infty),\qquad
E_0=E\cap(-y_q,-1),\qquad
E_-=E\cap(-\infty,-y_q)
\]
are compact after discarding the empty ones.  Apply the compact-uniform
regime estimates on these pieces: \eqref{eq:primary-local}-\eqref{eq:primary-left-bound}
on $E_+$, \eqref{eq:rare-local}-\eqref{eq:no-earlier} on $E_0$, and
\eqref{eq:null-only-left-margin}-\eqref{eq:null-only-right-margin} on
$E_-$, together with the corresponding null limits
\eqref{eq:null-tail} and \eqref{eq:rare-null-local}.  Taking the minimum of
the positive crossing margins over the nonempty pieces and then taking $r$
sufficiently small, we obtain brackets
\[
A_r(z)<B_r(z),\quad z\in E,
\]
and a constant $\gamma>0$, such that
\[
\sup_{z\in E}\sup_{u_q\le c\le A_r(z)}\bar R_{r,z}(c)\le1-\gamma,
\quad
\inf_{z\in E}\bar R_{r,z}\{B_r(z)\}\ge1+\gamma,
\]
\[
\sup_{z\in E}\sup_{A_r(z)\le c\le B_r(z)}
|\bar R_{0,r,z}(c)-D_*(z)|
\le\frac{\delta}{8}.
\]
Fix such an $r$ and put
\[
 C_*=\sup_{z\in E}B_r(z)<\infty.
\]
For $z\in E$, define
\[
 \mathcal E_N(z)
 =
 \max\left\{
 \|\widehat R_N-R_{r,z}\|_{\infty,[0,1]},
 \|\widehat R_{0,N}-\pi_{0,r}R_{0,r,z}\|_{\infty,[0,1]}
 \right\}
\]
and
\[
 t_*=\frac{p(C_*)}{q}
 \min\left\{\frac{\gamma}{2},\frac{\delta}{16}\right\}>0.
\]
Following the argument in the proof of
Theorem~\ref{thm:population-transfer}, on the event
$\mathcal E_N(z)<t_*$, for $u_q\le c\le C_*$,
\[
\begin{aligned}
 \left|
 \frac{q\widehat R_N\{p(c)\}}{p(c)}-\bar R_{r,z}(c)
 \right|
 &\le\frac{q\mathcal E_N(z)}{p(C_*)},\\
 \left|
 \frac{q\widehat R_{0,N}\{p(c)\}}{p(c)}
 -\bar R_{0,r,z}(c)
 \right|
 &\le\frac{q\mathcal E_N(z)}{p(C_*)}.
\end{aligned}
\]
The first bound and the two crossing margins imply
\[
 A_r(z)<\widehat C_N\le B_r(z).
\]
The second bound, \eqref{eq:exact-FDP}, and the null bracket then give
\[
 \FDP
 \ge \bar R_{0,r,z}(\widehat C_N)-\frac{\delta}{16}
 \ge D_*(z)-\frac{3\delta}{16}.
\]
Lemma~\ref{lem:DKW} gives, uniformly for $z\in E$,
\[
 \mathbb P\{\mathcal E_N(z)\ge t_*\mid Z=z\}
 \le \eta_N:=4e^{-2Nt_*^2}\longrightarrow0.
\]
Since $\FDP\ge0$ and $0\le D_*\le1$, conditional expectation and
integration over $E$ yield
\[
\begin{aligned}
 \mathbb E_{\Omega_N,Z,\varepsilon}[\FDP]
 &\ge
 (1-\eta_N)
 \int_E\left(D_*(z)-\frac{3\delta}{16}\right)_+
 \phi(z)\dd z\\
 &\ge
 \int_E D_*(z)\phi(z)\dd z-\frac{3\delta}{16}-\eta_N.
\end{aligned}
\]
Choose $N$ large enough that $\eta_N<\delta/16$.  Using the choice of
$E$ and \eqref{eq:FDRomega}, we obtain
\[
\mathbb E_{\Omega_N}\!\left[
\FDR_{\Omega_N}(\BH_q)
\right]
>
\ell_{=}(q)-\frac{\delta}{2}.
\]
For such an $N$, some realization $\omega^\star$ satisfies
\[
\FDR_{\omega^\star}(\BH_q)>{\ell_{=}}(q)-\frac{\delta}{2}.
\]
By Lemma~\ref{lem:eta}, choose $\eta>0$ so small that
\begin{equation}
{
\FDR_{\omega^\star}^{(\eta)}(\BH_q)
>
\FDR_{\omega^\star}(\BH_q)-\frac{\delta}{2}
>
\ell_{=}(q)-\delta.}
\label{eq:perturbed-fdr-lower}
\end{equation}
{
The perturbed vector $T^{(\eta)}$ has finite dimension $N$, mean vector
$(\theta_i)_{i=1}^N\in\mathbb R^N$, and covariance matrix
$\Sigma_{\omega^\star,\eta}$, which has unit diagonal and is strictly
positive definite by \eqref{eq:eta-PD} in Lemma~\ref{lem:eta}.  Thus the
bound \eqref{eq:perturbed-fdr-lower} establishes \eqref{eq:main}.  Since
$\delta>0$ was
arbitrary, taking the supremum proves \eqref{eq:sup}.}
\end{proof}

\begin{proof}[Proof of Proposition~\ref{prop:ell-expansion}]
Recall from Lemma~\ref{lem:envelope} that, for $y>1$, $a(y)$ is the
unique positive maximizer of $a\mapsto L_a(y)$, equivalently the unique
positive solution of
\[
a(y)=y\tanh\{a(y)y\}.
\]
By the monotonicity of $a(\cdot)$ established in that lemma,
$a(y)\ge a(2)>0$ for $y\ge2$, so
$a(y)y\to\infty$.
Moreover,
\[
0\le y-a(y)
=y\left[1-\tanh\{a(y)y\}\right]
\le2y e^{-2a(y)y}
\longrightarrow0.
\]
Because $a(y)$ is the maximizer defining $M(y)$, we have
$M(y)=L_{a(y)}(y)$.  Substituting the definitions of $L_{a(y)}(y)$
and $\phi(y)$, expanding
$\cosh(t)=(e^t+e^{-t})/2$, and completing the two squares gives
{
\begin{equation}
\begin{aligned}
M(y)\phi(y)
&=
e^{-a(y)^2/2}\cosh\{a(y)y\}
\frac{e^{-y^2/2}}{\sqrt{2\pi}}\\
&=
\frac{1}{2\sqrt{2\pi}}
\left[
e^{-\{a(y)^2+y^2\}/2+a(y)y}
+e^{-\{a(y)^2+y^2\}/2-a(y)y}
\right]\\
&=
\frac{1}{2\sqrt{2\pi}}
\left[
e^{-\{y-a(y)\}^2/2}
+e^{-\{y+a(y)\}^2/2}
\right].
\end{aligned}
\label{eq:Mphi-completed-squares}
\end{equation}}
Consequently, $M(y)\phi(y)\to 1/(2\sqrt{2\pi})$.  The difference
$M(y)\phi(y)-1/(2\sqrt{2\pi})$ is integrable on $[1,\infty)$.  Indeed, for all
sufficiently large $y$, $a(y)\ge y/2$, so
\[
y-a(y)\le2y e^{-y^2},
\]
while
\[
\left|e^{-\{y-a(y)\}^2/2}-1\right|
\le\frac{\{y-a(y)\}^2}{2},
\quad
e^{-\{y+a(y)\}^2/2}\le e^{-y^2/2}.
\]
It follows from the definition of $c_\ell$ that, as $Y\to\infty$,
\begin{equation}
{
\int_1^Y M(y)\phi(y)\dd y
=\frac{Y}{2\sqrt{2\pi}}+c_\ell-\Phi(1)+o(1).}
\label{eq:integral-M-asymptotic}
\end{equation}

Next, \eqref{eq:Mphi-completed-squares} gives
\[
M(y)
=
\frac12e^{y^2/2}
\left[
e^{-\{y-a(y)\}^2/2}
+e^{-\{y+a(y)\}^2/2}
\right]
=\frac12e^{y^2/2}\{1+o(1)\}.
\]
{
Since $M(y_q)=1/q\to\infty$, \eqref{eq:envelope-limits} in
Lemma~\ref{lem:envelope} implies $y_q\to\infty$.  The identity
$qM(y_q)=1$ therefore gives}
\begin{equation}
{
y_q^2
=2\log(1/q)+2\log2+o(1),
\quad
y_q=\sqrt{2\log(1/q)}+o(1).}
\label{eq:yq-small-q}
\end{equation}
Equations~\eqref{eq:mills-upper} in Lemma~\ref{lem:mills} and
\eqref{eq:yq-small-q} give
\begin{equation}
{
\Phib(y_q)
\le\frac{\phi(y_q)}{y_q}
=O\left(\frac{q}{y_q}\right)
=o(q).}
\label{eq:yq-tail-small-q}
\end{equation}
Substituting \eqref{eq:integral-M-asymptotic},
\eqref{eq:yq-small-q}, and \eqref{eq:yq-tail-small-q} into
\eqref{eq:ell} yields
\[
\begin{aligned}
\ell_{=}(q)
&=q\Phi(1)+q\left\{\frac{y_q}{2\sqrt{2\pi}}+c_\ell-\Phi(1)+o(1)\right\}+o(q)\\
&=\frac{q\sqrt{\log(1/q)}}{2\sqrt\pi}+{c_\ell}q+o(q).
\end{aligned}
\]
Numerical quadrature of the absolutely convergent integral defining
$c_\ell$ gives $c_\ell=0.6492828\ldots$.
\end{proof}

\begin{proof}[Proof of Proposition~\ref{prop:ell-strict}]
Since $q\in(0,1)$, \eqref{eq:yq} gives $M(y_q)=1/q>1$.  By Lemma~\ref{lem:envelope}, $M=1$ on $[0,1]$ and $M$ is strictly increasing on $(1,\infty)$, so $y_q>1$ and $M(y)\ge1$ for $1\le y\le y_q$.  Using \eqref{eq:ell},
\[
\begin{aligned}
\ell_{=}(q)-q
&=
q\Phi(1)+q\int_1^{y_q}M(y)\phi(y)\dd y+\Phib(y_q)\\
&\quad
-q\left\{\Phi(1)+\int_1^{y_q}\phi(y)\dd y+\Phib(y_q)\right\}\\
&=
q\int_1^{y_q}\{M(y)-1\}\phi(y)\dd y
+(1-q)\Phib(y_q).
\end{aligned}
\]
The integral is nonnegative, and $(1-q)\Phib(y_q)>0$ because $q<1$ and $y_q<\infty$.  Hence ${\ell_{=}}(q)>q$.
\end{proof}

\begingroup

\section{Lower bound for one-sided tests}
\label{sec:one-sided-lower}

\subsection{Main result}

For $0<q<1$, define
\begin{equation}
 y_q^+=\sqrt{2\log(1/q)}
 \label{eq:one-yq}
\end{equation}
and
\begin{equation}
 \ell_{\le}(q)
 =\frac{q\sqrt{\log(1/q)}}{\sqrt\pi}+\frac q2+\Phib(y_q^+).
 \label{eq:one-ell}
\end{equation}

The quantities in \eqref{eq:one-yq} and \eqref{eq:one-ell}
determine the finite-dimensional one-sided Gaussian lower bound.
\begin{theorem}
\label{thm:one-lower}
For every $q\in(0,1)$ and $\delta>0$, there are $N<\infty$,
$\theta\in\R^N$, and a positive definite correlation matrix $\Sigma$
such that BH at level $q$, applied to the one-sided $p$-values
$P_i=\Phib(T_i)$ from $T\sim N(\theta,\Sigma)$, satisfies
\[
 \FDR_{\theta,\Sigma}(\BH_q)>\ell_{\le}(q)-\delta.
\]
Every true null in the construction has $\theta_i=0$, and every
non-null
has $\theta_i>0$.
\end{theorem}

The first proposition gives the small-$q$ expansion of the one-sided lower
bound.
\begin{proposition}
\label{prop:one-expansion}
As $q\downarrow0$,
\[
 \ell_{\le}(q)
 =\frac{q\sqrt{\log(1/q)}}{\sqrt\pi}+\frac q2+o(q).
\]
In particular, $\ell_{\le}(q)/q\to\infty$.
\end{proposition}

The next proposition shows strict anticonservativeness at every nontrivial
nominal level.
\begin{proposition}
\label{prop:one-strict}
For every $q\in(0,1)$, $\ell_{\le}(q)>q$.
\end{proposition}

Figure~\ref{fig:one-ellq} in the Introduction plots $\ell_{\le}(q)-q$ and
$\ell_{\le}(q)/q$ over the grid $q=0.005k$, $k=1,\ldots,199$.
Table~\ref{tab:one-ell} gives the absolute inflation
$\ell_{\le}(q)-q$ and relative inflation $\ell_{\le}(q)/q$ for commonly
used nominal levels.

\subsection{Roadmap}

Fix $q$.  For a small $r>0$, let
$(\Theta_i,B_i)_{i=1}^N$ be i.i.d. from a law $\nu_{q,r}^+$, independently
of $Z,\varepsilon_1,\ldots,\varepsilon_N$, which are independent standard
normal variables, and set
\begin{equation}
 T_i=\Theta_i+B_iZ+\sqrt{1-B_i^2}\,\varepsilon_i.
 \label{eq:one-population-model}
\end{equation}
In the model \eqref{eq:one-population-model}, the law
$\nu_{q,r}^+$ has one null component and two non-null
components.  We
first analyze the population BH crossing conditional on $Z=z$.  For
every $z\notin\{0,y_q^+\}$, its conditional FDP converges, as
$r\downarrow0$, to
\begin{equation}
{
 D_+^*(z)=
 \begin{cases}
 q,&z\leq0,\\
 qe^{z^2/2},&0<z<y_q^+,\\
 1,&z\geq y_q^+.
 \end{cases}
}
 \label{eq:one-Dstar}
\end{equation}
The definition~\eqref{eq:one-Dstar} specifies the values at
$\{0,y_q^+\}$; these
values do not affect the expectation because the boundary points have zero
probability.  Therefore
\[
 \mathbb E\{D_+^*(Z)\}
 =\frac q2+q\int_0^{y_q^+}e^{z^2/2}\phi(z)\dd z+\Phib(y_q^+)
 =\ell_{\le}(q).
\]

We reuse the machinery and use similar notation in Subsections~\ref{sec:generic}
and~\ref{sec:finite-transfer}.  Throughout this section, put $p_+(x)=\Phib(x)$ and
$u_s^+=p_+^{-1}(s)$.  For $x\geq0$ and
$0<s\leq1/2$,
\begin{equation}
 p_+(x)=\frac{p(x)}2,
 \qquad
 p_+^{-1}(s)=p^{-1}(2s).
 \label{eq:one-two-tail-bridge}
\end{equation}
Therefore, all bounds for $p$ in Lemma~\ref{lem:mills} remain to apply
to $p_+$ up to a constant.  Use the symbols
\[
 \pi_0^+,\ R_z^+,\ R_{0,z}^+,\ \bar R_z^+,\ \bar R_{0,z}^+,
 \ C^+(z),\ D^+(z)
\]
as the one-sided analogues of the quantities in
\eqref{eq:pi0-generic} and \eqref{eq:Rz-generic}-\eqref{eq:D-generic}.
In particular, the one-sided population cutoff is
\[
 C^+(z)=\inf\{c\geq u_q^+:\bar R_z^+(c)\geq1\},
\]
and every one-sided use of Definition~\ref{def:regular} has lower endpoint
$u_q^+$ in place of $u_q$.
The only change to \eqref{eq:Ftheta-formula}-\eqref{eq:Ftheta-degenerate}
is
\begin{equation}
 F^+_{\theta,b,z}(s)
 =\Phib\!\left(\frac{u_s^+-\theta-bz}{\sqrt{1-b^2}}\right),
 \qquad |b|<1,
 \label{eq:one-conditional-cdf}
\end{equation}
while
$F^+_{\theta,b,z}(s)=\1\{\theta+bz\geq u_s^+\}$ for $|b|=1$.
We retain the empty-crossing convention $C^+(z)=\infty$, $D^+(z)=0$,
and the regularity definition in Definition~\ref{def:regular}.

The proofs of Lemmas~\ref{lem:DKW} and~\ref{lem:exact-BH} and
Theorem~\ref{thm:population-transfer} are p-value agnostic and apply after
\eqref{eq:one-conditional-cdf}; the one-sided rejection set in
Lemma~\ref{lem:exact-BH} is $\{i:T_i\geq\widehat C_N\}$.  We also use the
localized DKW bracketing in the proof of
Theorem~\ref{thm:population-transfer}.  Finally, the perturbation
\eqref{eq:eta-model}-\eqref{eq:eta-cov} and Lemma~\ref{lem:eta} apply to the
one-sided p-values as well, because their continuity argument does not use
the two-sided form of $p$.

{\subsection{Construction of
\texorpdfstring{$\nu_{q,r}^+$}{nu+(q,r)}}}

This is the one-sided, sign-reversed specialization of the three-component
construction in Subsection~\ref{sec:continuum-design}: the signal offset
$y(a)$ is replaced by $a$, the envelope term $M\{y(a)\}$ by
$e^{a^2/2}$, and the non-null loading $+1$ by $-1$.

Take $r>0$, and set
\begin{equation}
 b_r^+=r\,p_+^{-1}(r^2).
 \label{eq:one-br}
\end{equation}
We henceforth assume $r$ is sufficiently small that $r^2<1/2$ and
$b_r^+<y_q^+$.  For $b_r^+\leq a\leq y_q^+$, define
\begin{equation}
 W_r^+(a)=\frac{\{1-qe^{a^2/2}\}p_+(a/r)}{q}.
 \label{eq:one-Wr}
\end{equation}
Both factors in the numerator are nonnegative and decreasing in $a$, and the first
vanishes at $a=y_q^+$.  Hence $W_r^+$ is continuous and decreasing.  Put
\begin{equation}
 \tau_r^+=W_r^+(b_r^+),
 \qquad
 \mu_r^+([a,y_q^+])=\frac{W_r^+(a)}{\tau_r^+},
 \quad b_r^+\leq a\leq y_q^+.
 \label{eq:one-tau-mu}
\end{equation}
Thus $\mu_r^+$ is a probability measure on $[b_r^+,y_q^+]$.  Since
$p_+(b_r^+/r)=r^2$, \eqref{eq:one-two-tail-bridge} and the
quantile expansion
\eqref{eq:mills-quantile} in Lemma~\ref{lem:mills} gives $b_r^+\to0$, and hence
\begin{equation}
 \tau_r^+\sim\frac{1-q}{q}r^2.
 \label{eq:one-tau-asymptotic}
\end{equation}
Finally, set
\begin{equation}
 \pi_{0,r}^+=1-r-\tau_r^+,
 \qquad
 p_+(c_r^+)=\frac{q(r+\tau_r^+)}{1-q\pi_{0,r}^+}.
 \label{eq:one-pi-cr}
\end{equation}
For all sufficiently small $r$, these quantities are well-defined and
$\pi_{0,r}^+>0$.

The law $\nu_{q,r}^+$ is displayed in Table~\ref{tab:one-construction}.
Equivalently, for bounded measurable $f$,
\begin{equation}
 \int f\,\dd\nu_{q,r}^+
 =\pi_{0,r}^+f(0,r)
 +rf(c_r^++b_r^+,-1)
 +\tau_r^+\int_{b_r^+}^{y_q^+}f(a/r+a,-1)\,\mu_r^+(\dd a).
 \label{eq:one-nu}
\end{equation}

\begin{table}[ht]

\centering
\caption{One-sided population construction.  The opposite loading signs are essential.}
\label{tab:one-construction}
\begin{tabular}{lccc}
\toprule
Type & Mass & $\Theta$ & $B$\\
\midrule
Nulls & $\pi_{0,r}^+$ & $0$ & $r$\\
Primary non-nulls & $r$ & $c_r^++b_r^+$ & $-1$\\
Secondary non-nulls & $\tau_r^+\mu_r^+(\dd a)$ & $a/r+a$ &
$-1$, $a\in[b_r^+,y_q^+]$\\
\bottomrule
\end{tabular}
\end{table}

Under \eqref{eq:one-nu}, the nulls become more significant
when $Z$ is positive.  Every
non-null
statistic is deterministic conditional on $Z$ and decreases with $Z$.
The construction is not PRDS.  To see this directly, pair a null $i$ with
a non-null $j$, and fix $v\in(0,1)$.  Conditional on
$P_i=s$, the event
$\{P_j\geq v\}$, which is increasing in the p-values, has probability
\[
 \Phi\!\left(
 \frac{p_+^{-1}(v)-\theta_j+r p_+^{-1}(s)}{\sqrt{1-r^2}}
 \right),
\]
strictly decreasing in $s$.  This also exhibits the negative
null-to-non-null covariance, opposite to the nonnegative
covariance assumption
that guarantees Gaussian one-sided PRDS
\citep{BenjaminiYekutieli2001,chi2025multiple}.  On the other hand, all
null-null correlations equal $r^2$, so the null p-values are positively
regression dependent and the full vector is PRDN in the terminology of
\citet{Su2018}.

\vspace{1em}

\subsection{Population analysis}

\subsubsection{Preliminaries}

The next lemma gives the one-sided counterparts of
\eqref{eq:br-properties}-\eqref{eq:cr-properties}.
\begin{lemma}
\label{lem:one-scale-properties}
As $r\downarrow0$,
\begin{equation}
 b_r^+\to0,\qquad b_r^+/r\to\infty,\qquad c_r^+\to\infty,
 \qquad rc_r^+\to0,\qquad b_r^+/r-c_r^+\to\infty.
 \label{eq:one-scale-properties}
\end{equation}
\end{lemma}

\begin{proof}
By \eqref{eq:one-br} and the quantile expansion
\eqref{eq:mills-quantile} in Lemma~\ref{lem:mills},
\[
 \frac{b_r^+}{r}=p_+^{-1}(r^2)
 \sim2\sqrt{\log(1/r)}\to\infty,
 \qquad
 b_r^+\sim2r\sqrt{\log(1/r)}\to0.
\]
Moreover, \eqref{eq:one-tau-asymptotic} gives
$\tau_r^+=O(r^2)$ and hence $\pi_{0,r}^+\to1$.  Therefore
\eqref{eq:one-pi-cr} yields
\[
 p_+(c_r^+)\sim\frac{qr}{1-q}.
\]
Another application of \eqref{eq:mills-quantile} in
Lemma~\ref{lem:mills} gives
\[
 c_r^+\sim\sqrt{2\log(1/r)}.
\]
Thus $c_r^+\to\infty$ and $rc_r^+\to0$, while
\[
 \frac{b_r^+}{r}-c_r^+
 \sim(2-\sqrt2)\sqrt{\log(1/r)}
 \longrightarrow\infty.
\]
\end{proof}

Apply the reused population notation to $\nu_{q,r}^+$, and write the
resulting quantities as
$R_{r,z}^+$, $R_{0,r,z}^+$, $\bar R_{r,z}^+$,
$\bar R_{0,r,z}^+$, $\bar R_{1,r,z}^+$,
$\bar R_{2,r,z}^+$, $C_r^+(z)$, and $D_r^+(z)$.
Then
\begin{align}
 \bar R_{r,z}^+(c)
 &=\bar R_{0,r,z}^+(c)
 +\bar R_{1,r,z}^+(c)
 +\bar R_{2,r,z}^+(c),
 \label{eq:one-population-curve}\\
 \bar R_{0,r,z}^+(c)
 &=q\pi_{0,r}^+
 \frac{p_+\{(c-rz)/\sqrt{1-r^2}\}}{p_+(c)}.
 \label{eq:one-null-curve}\\
 \bar R_{1,r,z}^+(c)
 &{=\frac{qr}{p_+(c)}
 \1\{c\leq c_r^++b_r^+-z\},}
 \label{eq:one-primary-curve}\\
 \bar R_{2,r,z}^+(c)
 &=\frac{q\tau_r^+}{p_+(c)}
 \mu_r^+\!\left(\left[\frac{r(c+z)}{1+r},y_q^+\right]\right).
 \label{eq:one-secondary-curve}
\end{align}

The next lemma derives the derivatives of the null and secondary non-null
contributions.
\begin{lemma}
\label{lem:one-block-derivatives}
For $c>0$, put
\[
 u=\frac{c-rz}{\sqrt{1-r^2}}.
\]
The null contribution satisfies
\begin{align}
 \frac{\dd}{\dd c}\log\bar R_{0,r,z}^+(c)
 &=\lambda(c)-\frac{\lambda(u)}{\sqrt{1-r^2}},
 \label{eq:one-null-block-log-derivative}\\
 \frac{\dd}{\dd c}\bar R_{0,r,z}^+(c)
 &=\bar R_{0,r,z}^+(c)
 \left\{\lambda(c)-\frac{\lambda(u)}{\sqrt{1-r^2}}\right\}.
 \label{eq:one-null-block-derivative}
\end{align}
Define
\[
 \beta_{r,z}(c)=\frac{r(c+z)}{1+r}.
\]
If $\beta_{r,z}(c)\in(b_r^+,y_q^+)$, then
\begin{equation}
 \bar R_{2,r,z}^+(c)
 =\left[1-qe^{\{\beta_{r,z}(c)\}^2/2}\right]
 \frac{p_+\{\beta_{r,z}(c)/r\}}{p_+(c)}
 \label{eq:one-secondary-block-formula}
\end{equation}
and
\begin{equation}
 \frac{\dd}{\dd c}\log\bar R_{2,r,z}^+(c)
 =-\frac{qr\beta_{r,z}(c)e^{\{\beta_{r,z}(c)\}^2/2}}
 {(1+r)\left[1-qe^{\{\beta_{r,z}(c)\}^2/2}\right]}
 +\lambda(c)-\frac{\lambda\{\beta_{r,z}(c)/r\}}{1+r}.
 \label{eq:one-secondary-block-log-derivative}
\end{equation}
\end{lemma}

\begin{proof}
Since
\[
 \frac{\dd u}{\dd c}=\frac1{\sqrt{1-r^2}},
 \qquad
 \frac{\dd}{\dd t}\log p_+(t)
 =\frac{p_+'(t)}{p_+(t)}=-\lambda(t),
\]
Equation~\eqref{eq:one-null-curve} gives
\begin{align*}
 \frac{\dd}{\dd c}\log\bar R_{0,r,z}^+(c)
 &=\frac{\dd}{\dd c}
 \left[
   \log(q\pi_{0,r}^+)+\log p_+(u)-\log p_+(c)
 \right]\\
 &=-\frac{\lambda(u)}{\sqrt{1-r^2}}+\lambda(c),
\end{align*}
which is \eqref{eq:one-null-block-log-derivative}.  Therefore
\begin{align*}
 \frac{\dd}{\dd c}\bar R_{0,r,z}^+(c)
 &=\bar R_{0,r,z}^+(c)
 \frac{\dd}{\dd c}\log\bar R_{0,r,z}^+(c)\\
 &=\bar R_{0,r,z}^+(c)
 \left\{\lambda(c)-\frac{\lambda(u)}{\sqrt{1-r^2}}\right\},
\end{align*}
which proves \eqref{eq:one-null-block-derivative}.

Now suppose that
$\beta_{r,z}(c)\in(b_r^+,y_q^+)$.  By
\eqref{eq:one-secondary-curve}, \eqref{eq:one-tau-mu}, and
\eqref{eq:one-Wr},
\begin{align*}
 \bar R_{2,r,z}^+(c)
 &=\frac{q\tau_r^+}{p_+(c)}
 \mu_r^+([\beta_{r,z}(c),y_q^+])\\
 &=\frac{qW_r^+\{\beta_{r,z}(c)\}}{p_+(c)}\\
 &=\left[1-qe^{\{\beta_{r,z}(c)\}^2/2}\right]
 \frac{p_+\{\beta_{r,z}(c)/r\}}{p_+(c)}.
\end{align*}
This proves \eqref{eq:one-secondary-block-formula}.  Moreover,
\[
 \frac{\partial\beta_{r,z}(c)}{\partial c}=\frac{r}{1+r}.
\]
Since $\beta_{r,z}(c)<y_q^+$, we have
$1-qe^{\{\beta_{r,z}(c)\}^2/2}>0$, and hence
\begin{align*}
 \frac{\dd}{\dd c}\log\bar R_{2,r,z}^+(c)
 &=-\frac{q\beta_{r,z}(c)e^{\{\beta_{r,z}(c)\}^2/2}}
 {1-qe^{\{\beta_{r,z}(c)\}^2/2}}
 \frac{\partial\beta_{r,z}(c)}{\partial c}\\
 &\quad
 -\lambda\{\beta_{r,z}(c)/r\}\frac1r
 \frac{\partial\beta_{r,z}(c)}{\partial c}
 +\lambda(c)\\
 &=-\frac{qr\beta_{r,z}(c)e^{\{\beta_{r,z}(c)\}^2/2}}
 {(1+r)[1-qe^{\{\beta_{r,z}(c)\}^2/2}]}
 +\lambda(c)-\frac{\lambda\{\beta_{r,z}(c)/r\}}{1+r}.
\end{align*}
This proves \eqref{eq:one-secondary-block-log-derivative}.
\end{proof}

The next lemma gives the limit of the null contribution.
\begin{lemma}
\label{lem:one-null-sequential}
For every compact
$X\subset[0,\infty)$ and $K\subset\mathbb R$, as $r\downarrow0$
and $c\to\infty$, uniformly over $z\in K$ subject to $rc\in X$,
\begin{equation}
 \bar R_{0,r,z}^+(c)
 =q\exp\left\{(rc)z-\frac{(rc)^2}{2}\right\}+o(1).
 \label{eq:one-null-sequential}
\end{equation}
In particular, if $c=c(r)\to\infty$ and $rc\to x\geq0$, then the
left-hand side converges to
$q\exp(xz-x^2/2)$, uniformly for $z$ in compact sets.
\end{lemma}

\begin{proof}

Set
\[
 \Delta_r=\frac{c-rz}{\sqrt{1-r^2}}-c.
\]
Because $rc$ and $z$ range over compact sets while $c\to\infty$,
uniformly over the stated range,
$\Delta_r=O(r)=o(1)$ and
\begin{align*}
 c\Delta_r
 &=c^2\left\{\frac{1}{\sqrt{1-r^2}}-1\right\}
   -\frac{rzc}{\sqrt{1-r^2}}\\
 &=\frac12(rc)^2-z(rc)+o(1),
\end{align*}
where the remainder is uniform.  The hazard expansion
\eqref{eq:mills-hazard-expansion} in Lemma~\ref{lem:mills}, together with
$c\to\infty$ and $\Delta_r=o(1)$, gives
\[
 \log\frac{p_+(c+\Delta_r)}{p_+(c)}
 =-\int_c^{c+\Delta_r}\lambda(t)\,\dd t
 =-c\Delta_r+o(1)
\]
uniformly on the same range.  Exponentiating, using
$(c-rz)/\sqrt{1-r^2}=c+\Delta_r$, and then applying
\eqref{eq:one-null-curve} together with $\pi_{0,r}^+\to1$ proves
\eqref{eq:one-null-sequential}.
\end{proof}

The following lemma is the one-sided analogue of
Lemma~\ref{lem:localized-crossing}.
\begin{lemma}
\label{lem:one-localized-crossing}
Let $K$ be compact.  For $z\in K$, let
$m_r(z)>u_q^+$ and $w_r(z)>0$.
Suppose that, for every fixed $H>0$,
\begin{equation}
 \inf_{z\in K}\{m_r(z)-Hw_r(z)\}>u_q^+
 \label{eq:one-localized-crossing-domain}
\end{equation}
for all sufficiently small $r$, and
\begin{equation}
 \sup_{z\in K}\sup_{|h|\le H}
 \left|
 \bar R_{r,z}^+\{m_r(z)+w_r(z)h\}
 -d(z)-\{1-d(z)\}e^{\xi(z)h}
 \right|\longrightarrow0,
 \label{eq:one-localized-crossing-full}
\end{equation}
where $d:K\to[0,1)$ and $\xi:K\to(0,\infty)$ are continuous.
Suppose also that, for every $h_0>0$,
\begin{equation}
 \limsup_{r\downarrow0}\sup_{z\in K}
 \sup_{u_q^+\le c\le m_r(z)-w_r(z)h_0}
 \bar R_{r,z}^+(c)<1.
 \label{eq:one-localized-crossing-left}
\end{equation}
Then
\begin{equation}
 \sup_{z\in K}
 \left|\frac{C_r^+(z)-m_r(z)}{w_r(z)}\right|
 \longrightarrow0.
 \label{eq:one-localized-crossing-cutoff}
\end{equation}
\end{lemma}

\begin{proof}
The proof of Lemma~\ref{lem:localized-crossing} applies with
$u_q$, $\bar R_{r,z}$, and $C_r(z)$ replaced by
$u_q^+$, $\bar R_{r,z}^+$, and $C_r^+(z)$, respectively.
\end{proof}

The following lemma is the one-sided analogue of
Lemma~\ref{lem:localized-crossing-regularity}.
\begin{lemma}
\label{lem:one-localized-crossing-regularity}
Under the assumptions and notation of
Lemma~\ref{lem:one-localized-crossing}, assume in addition that, for every
fixed $H>0$,
\begin{equation}
 \sup_{z\in K}\sup_{|h|\le H}
 \left|
 \bar R_{0,r,z}^+\{m_r(z)+w_r(z)h\}-d(z)
 \right|\longrightarrow0.
 \label{eq:one-localized-crossing-null}
\end{equation}
If, for some $H_*>0$ and all sufficiently small $r$, the map
\[
 h\longmapsto\bar R_{r,z}^+\{m_r(z)+w_r(z)h\}
\]
is differentiable with positive derivative on $(-H_*,H_*)$, and the map
\[
 h\longmapsto\bar R_{0,r,z}^+\{m_r(z)+w_r(z)h\}
\]
is continuous there for every $z\in K$, then $C_r^+(z)$ is regular and
\begin{equation}
 \sup_{z\in K}|D_r^+(z)-d(z)|\longrightarrow0.
 \label{eq:one-localized-crossing-fdp}
\end{equation}
\end{lemma}

\begin{proof}
The proof of Lemma~\ref{lem:localized-crossing-regularity} applies with
$u_q$, $\bar R_{r,z}$, $\bar R_{0,r,z}$, $C_r(z)$, and $D_r(z)$
replaced by $u_q^+$, $\bar R_{r,z}^+$, $\bar R_{0,r,z}^+$,
$C_r^+(z)$, and $D_r^+(z)$, respectively.
\end{proof}

The next lemma shows that the primary non-null contribution is negligible uniformly
away from $z=0$.
\begin{lemma}
\label{lem:one-primary-bound}
For every compact $K\subset(0,\infty)$, uniformly for $z\in K$, the
primary non-null contribution in \eqref{eq:one-population-curve}, whenever nonzero,
is bounded by
\begin{equation}
 \frac{qr}{p_+(c_r^++b_r^+-z)}=o(1).
 \label{eq:one-primary-bound}
\end{equation}
\end{lemma}

\begin{proof}
Whenever the primary non-null contribution is nonzero, the indicator in
\eqref{eq:one-primary-curve} equals one, so
$\bar R_{1,r,z}^+(c)=qr/p_+(c)$.  Equation~\eqref{eq:one-primary-curve}
also gives
$c\leq c_r^++b_r^+-z$.  Since $p_+$ is decreasing,
\[
 \frac{qr}{p_+(c)}
 \leq \frac{qr}{p_+(c_r^++b_r^+-z)}.
\]
The right-hand side is the expression in \eqref{eq:one-primary-bound}.
By \eqref{eq:one-pi-cr}, it can be written as
\begin{equation*}
 \frac{qr}{p_+(c_r^++b_r^+-z)}
 =\frac{r(1-q\pi_{0,r}^+)}{r+\tau_r^+}
 \frac{p_+(c_r^+)}{p_+(c_r^++b_r^+-z)}.
\end{equation*}
Fix a compact $K\subset(0,\infty)$, and put
$\delta_{r,z}=z-b_r^+$.  By \eqref{eq:one-scale-properties} in
Lemma~\ref{lem:one-scale-properties}, $\delta_{r,z}$ remains in a fixed
compact subset of $(0,\infty)$, uniformly for $z\in K$, and
$c_r^+-\delta_{r,z}\to\infty$ uniformly.  Therefore
\eqref{eq:fixed-shift} in Lemma~\ref{lem:mills}, applied with
$x=c_r^+-\delta_{r,z}$, gives
\[
 \frac{p_+(c_r^+)}{p_+(c_r^++b_r^+-z)}
 =\frac{p_+(x+\delta_{r,z})}{p_+(x)}\longrightarrow0
\]
uniformly for $z\in K$.  Since
$0<r(1-q\pi_{0,r}^+)/(r+\tau_r^+)\leq1$, the expression in
\eqref{eq:one-primary-bound} is $o(1)$ uniformly.
\end{proof}

The next lemma bounds the secondary non-null contribution in the two positive-
$z$ regimes.
\begin{lemma}
\label{lem:one-secondary-bound}
For all $c,z\in\mathbb R$,
\begin{equation}
 0\leq \frac{q\tau_r^+}{p_+(c)}
 \mu_r^+\!\left(\left[\frac{r(c+z)}{1+r},y_q^+\right]\right)
 \leq\frac{p_+\{(c+z)/(1+r)\}}{p_+(c)}.
 \label{eq:one-secondary-bound}
\end{equation}
\end{lemma}

\begin{proof}
If $r(c+z)/(1+r)\leq b_r^+$, then
\[
 q\tau_r^+\leq p_+(b_r^+/r)
 \leq p_+\{(c+z)/(1+r)\}.
\]
If $b_r^+<r(c+z)/(1+r)<y_q^+$, then
\begin{align*}
 q\tau_r^+\mu_r^+\!\left(\left[\frac{r(c+z)}{1+r},y_q^+\right]\right)
 &=\left[1-q\exp\left\{\frac{r^2(c+z)^2}{2(1+r)^2}\right\}\right]
 p_+\!\left(\frac{c+z}{1+r}\right)\\
 &\leq p_+\!\left(\frac{c+z}{1+r}\right).
\end{align*}
If $r(c+z)/(1+r)\geq y_q^+$, the contribution is zero.  Dividing these
bounds by $p_+(c)$ proves \eqref{eq:one-secondary-bound}.
\end{proof}

\subsubsection{\texorpdfstring{The regime $z<0$}{The regime z < 0}}

\begin{lemma}
\label{lem:one-negative-regime}
For every compact $K\subset(-\infty,0)$, uniformly for $z\in K$,
{\begin{equation}
 C_r^+(z)=c_r^++o\{(c_r^+)^{-1}\}
 \label{eq:one-negative-cutoff}
\end{equation}
and
\begin{equation}
 D_r^+(z)\longrightarrow q.
 \label{eq:one-negative-fdp}
\end{equation}}
Moreover, for all sufficiently small $r$, $C_r^+(z)$ is regular for
every $z\in K$.
\end{lemma}

\begin{proof}
\textbf{Proof of \eqref{eq:one-negative-cutoff}.}
Fix a compact $K\subset(-\infty,0)$.  Uniformly for $z\in K$ and bounded
$h$, all primary and secondary non-nulls are rejected at
the cutoff $c_r^++h/c_r^+$.  Indeed, the primary statistic equals
$c_r^++b_r^+-z$, so its distance above the cutoff is
$b_r^+-z-h/c_r^+$.  This converges to $-z>0$ because
$b_r^+\to0$, $c_r^+\to\infty$, and $h$ is bounded.

Among the
secondary non-nulls, the smallest statistic is attained at
$a=b_r^+$.
Its distance above the cutoff is
\[
 \frac{b_r^+}{r}+b_r^+-z-c_r^+-\frac{h}{c_r^+}
 =\left(\frac{b_r^+}{r}-c_r^+\right)
  +b_r^+-z-\frac{h}{c_r^+},
\]
which tends to infinity by \eqref{eq:one-scale-properties} in
Lemma~\ref{lem:one-scale-properties}.  Thus both gaps are positive for all
sufficiently small $r$, uniformly when $z$ ranges over a compact subset
of $(-\infty,0)$ and $h$ remains bounded.

Since every non-null is rejected,
\[
 \bar R_{r,z}^+\!\left(c_r^++\frac h{c_r^+}\right)
 =\bar R_{0,r,z}^+\!\left(c_r^++\frac h{c_r^+}\right)
  +\frac{q(r+\tau_r^+)}{p_+(c_r^++h/c_r^+)}.
\]
For every fixed $H>0$,
\eqref{eq:one-scale-properties} in
Lemma~\ref{lem:one-scale-properties} gives, uniformly for $z\in K$ and
$|h|\leq H$,
\[
 c_r^++\frac h{c_r^+}\longrightarrow\infty,
 \qquad
 r\left(c_r^++\frac h{c_r^+}\right)\longrightarrow0.
\]
Therefore \eqref{eq:one-null-sequential} in
Lemma~\ref{lem:one-null-sequential} directly yields
\[
 \bar R_{0,r,z}^+\!\left(c_r^++\frac h{c_r^+}\right)\longrightarrow q
\]
uniformly for $z\in K$ and $|h|\leq H$.
For the combined non-null contribution, the
calibration \eqref{eq:one-pi-cr} gives the exact identity
\begin{equation}
 \frac{q(r+\tau_r^+)}{p_+(c_r^++h/c_r^+)}
 =(1-q\pi_{0,r}^+)
 \frac{p_+(c_r^+)}{p_+(c_r^++h/c_r^+)}.
 \label{eq:one-negative-nonnull-calibration}
\end{equation}
The tail ratio in
\eqref{eq:one-negative-nonnull-calibration} tends to $e^h$ by
\eqref{eq:h-over-x} in
Lemma~\ref{lem:mills}, whereas $1-q\pi_{0,r}^+\to1-q$.  Combining the
two contributions gives
\begin{equation}
 \bar R_{r,z}^+(c_r^++h/c_r^+)
 \longrightarrow q+(1-q)e^h,
 \qquad
 \bar R_{0,r,z}^+(c_r^++h/c_r^+)\longrightarrow q.
 \label{eq:one-negative-local}
\end{equation}
The first limit is below one for $h<0$, equals one at zero, and is above
one for $h>0$.

We next exclude earlier crossings.  By
\eqref{eq:one-scale-properties} in
Lemma~\ref{lem:one-scale-properties}, uniformly for $z\in K$, all
non-nulls are rejected when $u_q^+\leq c\leq c_r^+$, and
$(c-rz)/\sqrt{1-r^2}>c$ for small $r$.  Hence
\begin{equation}
 \bar R_{r,z}^+(c)
 <q\pi_{0,r}^++\frac{q(r+\tau_r^+)}{p_+(c)}
 \leq q\pi_{0,r}^++\frac{q(r+\tau_r^+)}{p_+(c_r^+)}=1.
 \label{eq:one-negative-left-pointwise}
\end{equation}
For every fixed $h_0>0$,
\eqref{eq:one-negative-left-pointwise} and the
calibration \eqref{eq:one-pi-cr} give, uniformly for $z\in K$,
\begin{equation}
\begin{aligned}
 \sup_{u_q^+\leq c\leq c_r^+-h_0/c_r^+}\bar R_{r,z}^+(c)
 &\leq q\pi_{0,r}^+
 +(1-q\pi_{0,r}^+)
 \frac{p_+(c_r^+)}{p_+(c_r^+-h_0/c_r^+)}\\
 &\longrightarrow q+(1-q)e^{-h_0}<1.
\end{aligned}
 \label{eq:one-negative-left}
\end{equation}
Thus Lemma~\ref{lem:one-localized-crossing} applies
with
\[
 m_r(z)=c_r^+,
 \qquad w_r(z)=\frac1{c_r^+},
 \qquad d(z)=q,
 \qquad \xi(z)=1.
\]
Equations~\eqref{eq:one-negative-local} and
\eqref{eq:one-negative-left} verify
\eqref{eq:one-localized-crossing-full} and
\eqref{eq:one-localized-crossing-left}, respectively.  Equation
\eqref{eq:one-localized-crossing-cutoff} therefore gives
\eqref{eq:one-negative-cutoff}.

\par\medskip\noindent
\textbf{Proof of \eqref{eq:one-negative-fdp}.}
To apply
Lemma~\ref{lem:one-localized-crossing-regularity}, it remains to
prove that the derivative of
$\bar R_{r,z}^+(c_r^++h/c_r^+)$ with respect to $h$ is positive
there.
{
Fix $H<\infty$.  Uniformly for $z\in K$ and $|h|\le H$,
{
put
\[
 c=c_r^++\frac{h}{c_r^+},
 \qquad
 u=\frac{c-rz}{\sqrt{1-r^2}}.
\]
Equation~\eqref{eq:one-scale-properties} in
Lemma~\ref{lem:one-scale-properties} gives
$c/c_r^+\to1$ and $u/c_r^+\to1$ uniformly.  Hence
\eqref{eq:mills-hazard} in Lemma~\ref{lem:mills} gives
\[
\begin{aligned}
 \lambda(c)-\frac{\lambda(u)}{\sqrt{1-r^2}}
 &=c-\frac{u}{\sqrt{1-r^2}}+O\{(c_r^+)^{-1}\}\\
 &=\frac{rz-r^2c}{1-r^2}+O\{(c_r^+)^{-1}\}\\
 &=O_{K,H}\!\left(r+r^2c_r^++\frac1{c_r^+}\right)
 =o(c_r^+).
\end{aligned}
\]
Thus the hazard difference in
\eqref{eq:one-null-block-derivative} is $o(c_r^+)$ uniformly.}
Applying
\eqref{eq:one-null-block-derivative} in
Lemma~\ref{lem:one-block-derivatives}, differentiating the fully rejected
non-null contribution directly using
$p_+'(c)=-\lambda(c)p_+(c)$, and then using
\eqref{eq:one-negative-local}, the calibration \eqref{eq:one-pi-cr},
and \eqref{eq:h-over-x} in Lemma~\ref{lem:mills}, gives
\[
\begin{aligned}
 \frac{\partial}{\partial h}
 \bar R_{0,r,z}^+\!\left(c_r^++\frac h{c_r^+}\right)
 &\longrightarrow0,
 &
 \frac{\partial}{\partial h}
 \frac{q(r+\tau_r^+)}{p_+(c_r^++h/c_r^+)}
 &\longrightarrow(1-q)e^h,\\
 \frac{\partial}{\partial h}
 \bar R_{r,z}^+\!\left(c_r^++\frac h{c_r^+}\right)
 &\longrightarrow(1-q)e^h.
\end{aligned}
\]}
Since $\inf_{|h|\le H}(1-q)e^h=(1-q)e^{-H}>0$, for all sufficiently
small $r$,
\begin{equation}
 \inf_{z\in K}\inf_{|h|\le H}
 \frac{\partial}{\partial h}
 \bar R_{r,z}^+\!\left(c_r^++\frac h{c_r^+}\right)>0.
 \label{eq:one-negative-positive-derivative}
\end{equation}
The second limit in \eqref{eq:one-negative-local} verifies
\eqref{eq:one-localized-crossing-null}.  The derivative
bound \eqref{eq:one-negative-positive-derivative} and
continuity of the null contribution verify the remaining hypotheses of
Lemma~\ref{lem:one-localized-crossing-regularity}.  That lemma
therefore gives regularity, and
\eqref{eq:one-localized-crossing-fdp} gives
\eqref{eq:one-negative-fdp}, uniformly on $K$.
\end{proof}

\subsubsection{\texorpdfstring{The regime $z\in(0,y_q^+)$}{The regime 0 < z < yq+}}

\begin{lemma}
\label{lem:one-middle-regime}
For every compact $K\subset(0,y_q^+)$, uniformly for $z\in K$,
{\begin{equation}
 C_r^+(z)=\frac zr+o(1)
 \label{eq:one-middle-cutoff}
\end{equation}
and
\begin{equation}
 D_r^+(z)\longrightarrow qe^{z^2/2}.
 \label{eq:one-middle-fdp}
\end{equation}}
Moreover, for all sufficiently small $r$, $C_r^+(z)$ is regular for
every $z\in K$.
\end{lemma}

\begin{proof}
\textbf{Proof of \eqref{eq:one-middle-cutoff}.}
Fix a compact $K\subset(0,y_q^+)$ and $H>0$, and put $c=z/r+h$.
On $z\in K$ and $|h|\leq H$, the primary non-null contribution in
\eqref{eq:one-population-curve} is nonzero only if
\[
 c_r^++b_r^+-z\geq\frac zr+h.
\]
Let $\delta_K=\inf_{z\in K}z>0$.  Uniformly for $z\in K$ and
$|h|\leq H$,
\[
\begin{aligned}
 r\left\{\frac zr+h-(c_r^++b_r^+-z)\right\}
 &=z+rh-rc_r^+-rb_r^++rz\\
 &\geq\delta_K-rH-rc_r^+-rb_r^++r\delta_K
 \longrightarrow\delta_K>0
\end{aligned}
\]
by \eqref{eq:one-scale-properties} in
Lemma~\ref{lem:one-scale-properties}.  Hence
\begin{equation}
 c_r^++b_r^+-z<\frac zr+h
 \label{eq:one-middle-primary-absent}
\end{equation}
{
for all sufficiently small $r$, uniformly for $z\in K$ and
$|h|\le H$, so the primary non-null contribution is absent.}

{A secondary non-null is rejected if and only if
\[
 a\geq\beta_{r,z}(c)=z+\frac{rh}{1+r}.
\]
Compactness of $K\subset(0,y_q^+)$ implies that
$\beta_{r,z}(c)\in(b_r^+,y_q^+)$ for all sufficiently small $r$,
uniformly for $z\in K$ and $|h|\le H$.  In the formula for
$\bar R_{2,r,z}^+(c)$ in Lemma~\ref{lem:one-block-derivatives},
\[
 \beta_{r,z}(c)\longrightarrow z,
 \qquad
 \frac{\beta_{r,z}(c)}r=\frac zr+\frac{h}{1+r},
\]
uniformly on $K\times[-H,H]$, so its first factor converges to
$1-qe^{z^2/2}$.  With $x=z/r$,
\eqref{eq:mills-local-ratio} in Lemma~\ref{lem:mills} and the identity
$p_+(x+u)/p_+(x+v)=p(x+u)/p(x+v)$ give
\[
\begin{aligned}
 \log
 \frac{p_+\{x+h/(1+r)\}}{p_+(x+h)}
 &=
 \log\frac{x+h}{x+h/(1+r)}
 -x\left\{\frac{h}{1+r}-h\right\}
 -\frac12\left\{\frac{h^2}{(1+r)^2}-h^2\right\}
 +o(1)\\
 &=
 \log\left\{1+\frac{hr}{(1+r)x+h}\right\}
 +\frac{zh}{1+r}
 +\frac{h^2r(2+r)}{2(1+r)^2}
 +o(1)\\
 &=zh+o(1)
 \qquad\text{uniformly on }K\times[-H,H].
\end{aligned}
\]
Consequently,
\begin{equation}
 \bar R_{2,r,z}^+(z/r+h)
 \longrightarrow
 \{1-qe^{z^2/2}\}e^{zh}.
 \label{eq:one-middle-secondary-local}
\end{equation}}
{For $c=z/r+h$, uniformly for $z\in K$ and $|h|\leq H$,
we have $c\to\infty$ and $rc=z+rh$.  By uniformity of
\eqref{eq:one-null-sequential} in Lemma~\ref{lem:one-null-sequential},
\[
 \sup_{z\in K}\sup_{|h|\leq H}
 \left|
 \bar R_{0,r,z}^+(z/r+h)-qe^{z^2/2}
 \right|\longrightarrow0.
\]}
Combining this with \eqref{eq:one-middle-secondary-local} and the absence
of the primary non-null contribution for $z\in K$ and
$|h|\le H$ gives
\begin{equation}
 \bar R_{r,z}^+(z/r+h)
 \longrightarrow
 qe^{z^2/2}+\{1-qe^{z^2/2}\}e^{zh},
 \qquad
 \bar R_{0,r,z}^+(z/r+h)\longrightarrow qe^{z^2/2}.
 \label{eq:one-middle-local}
\end{equation}
Because $qe^{z^2/2}<1$ uniformly away from one on $K$, the first limit
crosses one strictly at $h=0$.

We now verify \eqref{eq:one-localized-crossing-left} in
Lemma~\ref{lem:one-localized-crossing}.  For every fixed $\eta>0$, we show
that
\begin{equation}
 \limsup_{r\downarrow0}\sup_{z\in K}
 \sup_{u_q^+\leq c\leq z/r-\eta}\bar R_{r,z}^+(c)<1.
 \label{eq:one-middle-left}
\end{equation}
Suppose instead that this assertion failed.  Then there would be sequences
satisfying
\[
 r_n\downarrow0,
 \qquad z_n\in K,
 \qquad u_q^+\leq c_n\leq\frac {z_n}{r_n}-\eta,
 \qquad \bar R_{r_n,z_n}^+(c_n)\geq1-o(1).
\]
After passing to subsequences,
\[
 z_n\longrightarrow z_0,
 \qquad r_nc_n\longrightarrow x\in[0,z_0].
\]
We consider three cases to get a contradiction.

~\\
\noindent\textbf{Case (a). } If $c_n=O(1)$, then, since $z_n\in K$,
\[
 \Delta_n
 :=\frac{c_n-r_nz_n}{\sqrt{1-r_n^2}}-c_n
 =c_n\left\{\frac{1}{\sqrt{1-r_n^2}}-1\right\}
 -\frac{r_nz_n}{\sqrt{1-r_n^2}}
 =o(1).
\]
Moreover, $\inf_n p_+(c_n)>0$.  Hence continuity of $p_+$,
$\pi_{0,r_n}^+\to1$, and \eqref{eq:one-null-curve} give
\[
 \frac{p_+(c_n+\Delta_n)}{p_+(c_n)}=1+o(1),
 \qquad
 \bar R_{0,r_n,z_n}^+(c_n)
 =q\pi_{0,r_n}^+
 \frac{p_+(c_n+\Delta_n)}{p_+(c_n)}
 =q+o(1).
\]
By \eqref{eq:one-population-curve},
\[
\begin{aligned}
 0
 &\leq \bar R_{r_n,z_n}^+(c_n)-\bar R_{0,r_n,z_n}^+(c_n)\\
 &=\frac{qr_n}{p_+(c_n)}
 \1\{c_{r_n}^++b_{r_n}^+-z_n\geq c_n\}
 +\frac{q\tau_{r_n}^+}{p_+(c_n)}
 \mu_{r_n}^+\!\left(
 \left[\frac{r_n(c_n+z_n)}{1+r_n},y_q^+\right]
 \right)\\
 &\leq\frac{q(r_n+\tau_{r_n}^+)}{p_+(c_n)}
 =O(r_n)=o(1),
\end{aligned}
\]
The equality
$q(r_n+\tau_{r_n}^+)/p_+(c_n)=O(r_n)$ uses
\eqref{eq:one-tau-asymptotic}.  Thus
$\bar R_{r_n,z_n}^+(c_n)=q+o(1)<1$, a contradiction.

~\\
\noindent\textbf{Case (b). } $c_n\to\infty$ and $x<z_0$.  Put
$d_n=(c_n+z_n)/(1+r_n)$.  Then
\[
 d_n-c_n=\frac{z_n-r_nc_n}{1+r_n}\longrightarrow z_0-x>0,
 \qquad c_n(d_n-c_n)\longrightarrow\infty.
\]
By \eqref{eq:one-secondary-bound} in
Lemma~\ref{lem:one-secondary-bound} and \eqref{eq:mills-hazard} in
Lemma~\ref{lem:mills}, the secondary non-null contribution satisfies
\[
 0\leq
 \frac{q\tau_{r_n}^+}{p_+(c_n)}
 \mu_{r_n}^+\!\left(\left[\frac{r_n(c_n+z_n)}{1+r_n},y_q^+\right]\right)
 \leq\frac{p_+(d_n)}{p_+(c_n)}
 =\exp\!\left\{-\int_{c_n}^{d_n}\lambda(t)\,dt\right\}
 \leq e^{-c_n(d_n-c_n)}=o(1).
\]
Equations \eqref{eq:one-primary-bound} and
\eqref{eq:one-null-sequential} in Lemmas~\ref{lem:one-primary-bound}
and~\ref{lem:one-null-sequential} therefore yield
\[
 \bar R_{r_n,z_n}^+(c_n)
 =q\exp\left(xz_0-\frac{x^2}{2}\right)+o(1)
 \leq qe^{z_0^2/2}+o(1)<1,
\]
again a contradiction.

~\\
\noindent\textbf{Case (c). } $c_n\to\infty$ and $x=z_0$.  Write
$h_n=c_n-z_n/r_n\leq-\eta$.  If $h_n$ is bounded,
\eqref{eq:one-middle-local} gives a limit strictly below one.
If $h_n\to-\infty$, then
\[
 d_n-c_n=-\frac{r_nh_n}{1+r_n}>0,
 \qquad c_n(d_n-c_n)\sim-z_0h_n\longrightarrow\infty.
\]
The hazard bound \eqref{eq:mills-hazard} in Lemma~\ref{lem:mills} gives
\[
 \frac{p_+(d_n)}{p_+(c_n)}
 =\exp\!\left\{-\int_{c_n}^{d_n}\lambda(t)\,dt\right\}
 \leq\exp\{-c_n(d_n-c_n)\}
 \longrightarrow0.
\]
Hence \eqref{eq:one-secondary-bound} makes the secondary non-null contribution
$o(1)$, while \eqref{eq:one-primary-bound} makes the primary
non-null contribution $o(1)$.  Since $r_nc_n\to x=z_0$ and $z_n\to z_0$,
\eqref{eq:one-null-sequential} gives
\[
 \bar R_{0,r_n,z_n}^+(c_n)
 \longrightarrow
 q\exp\left(z_0^2-\frac{z_0^2}{2}\right)
 =qe^{z_0^2/2}<1.
\]
Therefore
\[
 \bar R_{r_n,z_n}^+(c_n)
 =\bar R_{0,r_n,z_n}^+(c_n)+o(1)
 \longrightarrow qe^{z_0^2/2}<1,
\]
again a contradiction.

Since all three cases lead to a contradiction, \eqref{eq:one-middle-left}
holds.  Apply Lemma~\ref{lem:one-localized-crossing} with
\[
 m_r(z)=\frac zr,\qquad w_r(z)=1,\qquad
 d(z)=qe^{z^2/2},\qquad \xi(z)=z.
\]
Equations~\eqref{eq:one-middle-local} and
\eqref{eq:one-middle-left} verify
\eqref{eq:one-localized-crossing-full} and
\eqref{eq:one-localized-crossing-left}, respectively.  Therefore
\eqref{eq:one-localized-crossing-cutoff} yields
\eqref{eq:one-middle-cutoff}.

\par\medskip\noindent
\textbf{Proof of \eqref{eq:one-middle-fdp}.}
We apply
Lemma~\ref{lem:one-localized-crossing-regularity}.  Fix $H>0$, put
\[
 c=\frac zr+h,
 \qquad
 \beta_{r,z}(c)=z+\frac{rh}{1+r},
 \qquad |h|\leq H.
\]
Equation~\eqref{eq:one-middle-primary-absent} shows that the primary
non-null contribution is absent uniformly on this set.  Put
$u=(c-rz)/\sqrt{1-r^2}$.  The hazard expansion
\eqref{eq:mills-hazard-expansion} in Lemma~\ref{lem:mills} gives
\[
 \lambda(c)-\frac{\lambda(u)}{\sqrt{1-r^2}}=o(1)
\]
uniformly for $z\in K$ and $|h|\leq H$.  In the secondary log
derivative \eqref{eq:one-secondary-block-log-derivative}, the first term is
$O(r)$, while the same hazard expansion gives
\[
 \lambda(c)-\frac{\lambda\{\beta_{r,z}(c)/r\}}{1+r}
 =z+o(1)
\]
uniformly.  Combining these two estimates with
\eqref{eq:one-middle-secondary-local} yields
\[
 \sup_{z\in K}\sup_{|h|\leq H}
 \left|
 \frac{\partial}{\partial h}\bar R_{r,z}^+(z/r+h)
 -z\{1-qe^{z^2/2}\}e^{zh}
 \right|\longrightarrow0.
\]
Since
\[
 \inf_{z\in K}\inf_{|h|\leq H}
 z\{1-qe^{z^2/2}\}e^{zh}>0,
\]
the derivative is positive on every fixed local window for all sufficiently
small $r$.  Equation~\eqref{eq:one-middle-cutoff} verifies
\eqref{eq:one-localized-crossing-cutoff}, and the second limit in
\eqref{eq:one-middle-local} verifies
\eqref{eq:one-localized-crossing-null}.  Continuity follows from
\eqref{eq:one-null-curve}.
Lemma~\ref{lem:one-localized-crossing-regularity} therefore gives
regularity, and \eqref{eq:one-localized-crossing-fdp} gives
\eqref{eq:one-middle-fdp}, uniformly on $K$.
\end{proof}

\subsubsection{\texorpdfstring{The regime $z>y_q^+$}{The regime z > yq+}}

For $z>y_q^+$, define
\[
\alpha_q^+(z)=z-\sqrt{z^2-(y_q^+)^2}\in(0,z).
\]
\begin{lemma}
\label{lem:one-positive-regime}
For every compact $K\subset(y_q^+,\infty)$, uniformly for $z\in K$,
{\begin{equation}
 rC_r^+(z)\longrightarrow\alpha_q^+(z)
 \label{eq:one-positive-cutoff}
\end{equation}
and
\begin{equation}
 D_r^+(z)\longrightarrow1.
 \label{eq:one-positive-fdp}
\end{equation}}
Moreover, for all sufficiently small $r$, $C_r^+(z)$ is regular for
every $z\in K$.
\end{lemma}

\begin{proof}
\textbf{Proof of \eqref{eq:one-positive-cutoff}.}
Fix a compact $K\subset(y_q^+,\infty)$.  Fix
$z\in K$ and $x\in(0,z)$.  The limits below also hold uniformly when
$(x,z)$ ranges over a compact subset of
$\{(x,z):z\in K,\ 0<x<z\}$.  Then
\eqref{eq:one-scale-properties} in
Lemma~\ref{lem:one-scale-properties} implies that the primary
non-null is not rejected at $c=x/r$.  {For the
secondary non-null contribution, put $d=(x/r+z)/(1+r)$.  Then
\[
 d-\frac xr=\frac{z-x}{1+r}>0,
\]
and \eqref{eq:one-secondary-bound} together with
\eqref{eq:mills-hazard} in Lemma~\ref{lem:mills} gives
\[
 0\leq\bar R_{2,r,z}^+(x/r)
 \leq\frac{p_+(d)}{p_+(x/r)}
 \leq\exp\{- (x/r)(d-x/r)\}\longrightarrow0.
\]}
Equation
\eqref{eq:one-null-sequential} in Lemma~\ref{lem:one-null-sequential},
applied with $c=x/r$, gives
\begin{equation}
 \bar R_{r,z}^+(x/r)\longrightarrow
 q\exp\left(xz-\frac{x^2}{2}\right),
 \label{eq:one-large-local}
\end{equation}
with the same limit for the null contribution.  The equation
\[
 q\exp\left(xz-\frac{x^2}{2}\right)=1
\]
has smaller positive solution $\alpha_q^+(z)$.  The limit in
\eqref{eq:one-large-local} is below one for $x<\alpha_q^+(z)$ and above
one for $\alpha_q^+(z)<x<z$.

We now establish the uniform gap needed to exclude earlier
crossings.  Fix $\eta>0$ uniformly small enough that
\[
 0<\alpha_q^+(z)-\eta<\alpha_q^+(z)+\eta<z,
 \qquad z\in K.
\]
We claim that
\begin{equation}
 \limsup_{r\downarrow0}\sup_{z\in K}
 \sup_{u_q^+\leq c\leq\{\alpha_q^+(z)-\eta\}/r}
 \bar R_{r,z}^+(c)<1.
 \label{eq:one-positive-left}
\end{equation}
If \eqref{eq:one-positive-left} failed, there would be sequences
\[
 r_n\downarrow0,\qquad z_n\in K,\qquad
 u_q^+\leq c_n\leq\frac{\alpha_q^+(z_n)-\eta}{r_n},
 \qquad
 \bar R_{r_n,z_n}^+(c_n)\geq1-o(1).
\]
After taking a subsequence,
\[
 z_n\longrightarrow z_0\in K,
 \qquad
 r_nc_n\longrightarrow x\in[0,\alpha_q^+(z_0)-\eta].
\]
If $(c_n)$ is bounded, then $p_+(c_n)$ is bounded away from zero,
\eqref{eq:one-null-curve} gives
$\bar R_{0,r_n,z_n}^+(c_n)=q+o(1)$, and
\eqref{eq:one-population-curve} and \eqref{eq:one-tau-asymptotic} give
\[
 0\leq
 \bar R_{r_n,z_n}^+(c_n)-\bar R_{0,r_n,z_n}^+(c_n)
 \leq\frac{q(r_n+\tau_{r_n}^+)}{p_+(c_n)}=o(1).
\]
Thus $\bar R_{r_n,z_n}^+(c_n)\to q<1$, a contradiction.

Suppose instead that $c_n\to\infty$, and define
$d_n=(c_n+z_n)/(1+r_n)$.  Since
\[
 d_n-c_n=\frac{z_n-r_nc_n}{1+r_n}
 \longrightarrow z_0-x
 \geq z_0-\alpha_q^+(z_0)+\eta>0,
\]
\eqref{eq:one-secondary-bound} in
Lemma~\ref{lem:one-secondary-bound} and \eqref{eq:mills-hazard} in
Lemma~\ref{lem:mills} give
\[
 0\leq\bar R_{2,r_n,z_n}^+(c_n)
 \leq\exp\{-c_n(d_n-c_n)\}=o(1).
\]
Equation~\eqref{eq:one-primary-bound} in
Lemma~\ref{lem:one-primary-bound} gives
$\bar R_{1,r_n,z_n}^+(c_n)=o(1)$, and
\eqref{eq:one-null-sequential} in
Lemma~\ref{lem:one-null-sequential} gives
\[
 \bar R_{r_n,z_n}^+(c_n)
 \longrightarrow q\exp\left(xz_0-\frac{x^2}{2}\right)
 \leq q\exp\left(
   \{\alpha_q^+(z_0)-\eta\}z_0
   -\frac{\{\alpha_q^+(z_0)-\eta\}^2}{2}
 \right)<1,
\]
where the inequality uses
$x\leq\alpha_q^+(z_0)-\eta<z_0$.  This is again a contradiction, so
\eqref{eq:one-positive-left} holds.  In particular, no crossing occurs at
or before $\{\alpha_q^+(z)-\eta\}/r$.  At
$c=\{\alpha_q^+(z)+\eta\}/r$, the null contribution alone is larger
than one for all sufficiently small $r$, uniformly on $K$.  Therefore
\[
 \alpha_q^+(z)-\eta<rC_r^+(z)<\alpha_q^+(z)+\eta.
\]
Letting $\eta\downarrow0$ proves
\[
 rC_r^+(z)\longrightarrow\alpha_q^+(z)
\]
uniformly on $K$.

\par\medskip\noindent
\textbf{Proof of \eqref{eq:one-positive-fdp}.}
It remains to prove regularity.  Choose $\eta>0$ such that, for every
$z\in K$,
\[
 0<\alpha_q^+(z)-2\eta<\alpha_q^+(z)+2\eta
 <\min\{z,y_q^+\}.
\]
For any
$x\in[\alpha_q^+(z)-\eta,\alpha_q^+(z)+\eta]$, uniformly over $z\in K$,
\[
 r\{x/r-(c_r^++b_r^+-z)\}
 =x-rc_r^+-rb_r^++rz>0
\]
by \eqref{eq:one-scale-properties} in
Lemma~\ref{lem:one-scale-properties}, so the primary non-null
contribution is zero.
{
Applying \eqref{eq:mills-hazard-expansion} in Lemma~\ref{lem:mills}
directly to \eqref{eq:one-null-block-log-derivative} in
Lemma~\ref{lem:one-block-derivatives} gives, uniformly for $z\in K$ and
$x\in[\alpha_q^+(z)-\eta,\alpha_q^+(z)+\eta]$,
\begin{equation}
 \left.\frac{\dd}{\dd c}\log\bar R_{0,r,z}^+(c)\right|_{c=x/r}
 =r(z-x)+o(r).
 \label{eq:one-positive-null-log-derivative}
\end{equation}
Set
\[
 c_1=\inf_{z\in K}\alpha_q^+(z)-\eta>0,
 \qquad
 \delta=\frac12\left\{\inf_{z\in K}
 [z-\alpha_q^+(z)]-\eta\right\}>0.
\]
The null limit \eqref{eq:one-null-sequential} in
Lemma~\ref{lem:one-null-sequential} and
\eqref{eq:one-positive-null-log-derivative} imply that, for
some $c_0>0$,
\begin{equation}
\begin{aligned}
 \inf_{z\in K}\inf_{x\in[\alpha_q^+(z)-\eta,\alpha_q^+(z)+\eta]}
 \bar R_{0,r,z}^+(x/r)&\geq c_0,\\
 \inf_{z\in K}\inf_{x\in[\alpha_q^+(z)-\eta,\alpha_q^+(z)+\eta]}
 \left.\frac{\dd}{\dd c}\bar R_{0,r,z}^+(c)\right|_{c=x/r}
 &\geq c_0\delta r.
\end{aligned}
 \label{eq:one-positive-null-derivative}
\end{equation}
for all sufficiently small $r$.
{
For the secondary non-null contribution, put
$d=(x/r+z)/(1+r)$.  Here $rd\in(b_r^+,y_q^+)$,
$x/r\geq c_1/r$, and $d-x/r\geq\delta$.
Thus \eqref{eq:one-secondary-block-log-derivative} in
Lemma~\ref{lem:one-block-derivatives} and \eqref{eq:mills-hazard} in
Lemma~\ref{lem:mills} yield
\[
\begin{aligned}
 0\leq\bar R_{2,r,z}^+(x/r)
 &\leq\frac{p_+(d)}{p_+(x/r)}
 \leq e^{-(x/r)(d-x/r)}\leq e^{-c_1\delta/r},\\
 \left|\left.\frac{\dd}{\dd c}\bar R_{2,r,z}^+(c)\right|_{c=x/r}\right|
 &=O(r^{-1}e^{-c_1\delta/r})=o(r).
\end{aligned}
\]}
Combining this bound with \eqref{eq:one-positive-null-derivative},
\[
 \inf_{z\in K}\inf_{x\in[\alpha_q^+(z)-\eta,\alpha_q^+(z)+\eta]}
 \left.\frac{\dd}{\dd c}\bar R_{r,z}^+(c)\right|_{c=x/r}>0.
\]
{
Equation~\eqref{eq:one-positive-cutoff} gives
\[
 \frac{\alpha_q^+(z)-\eta}{r}
 <C_r^+(z)<
 \frac{\alpha_q^+(z)+\eta}{r}
\]
for all sufficiently small $r$, uniformly for $z\in K$.}  Hence
Proposition~\ref{prop:local-regularity} proves regularity uniformly on
$K$.

At the population cutoff, regularity gives
$\bar R_{r,z}^+\{C_r^+(z)\}=1$.  By
\eqref{eq:one-positive-cutoff}, uniformly for $z\in K$,
\[
 rC_r^+(z)\longrightarrow\alpha_q^+(z),
 \qquad
 \inf_{z\in K}\alpha_q^+(z)>0.
\]
In contrast, \eqref{eq:one-scale-properties} in
Lemma~\ref{lem:one-scale-properties} gives
\[
 r(c_r^++b_r^+-z)=rc_r^++rb_r^+-rz\longrightarrow0
\]
uniformly on $K$.  Hence
$C_r^+(z)>c_r^++b_r^+-z$ for all sufficiently small $r$, and the
indicator in \eqref{eq:one-primary-curve} is zero.  Thus the primary
non-null contribution vanishes at $C_r^+(z)$.  With
$d=\{C_r^+(z)+z\}/(1+r)$,
\eqref{eq:one-positive-cutoff} gives
\[
 C_r^+(z)\longrightarrow\infty,
 \qquad
 d-C_r^+(z)=\frac{z-rC_r^+(z)}{1+r}
 \longrightarrow z-\alpha_q^+(z)>0
\]
uniformly on $K$.  Hence Lemma~\ref{lem:one-secondary-bound} and
\eqref{eq:mills-hazard} in Lemma~\ref{lem:mills} give
\[
 0\leq\bar R_{2,r,z}^+\{C_r^+(z)\}
 \leq \exp[-C_r^+(z)\{d-C_r^+(z)\}]\longrightarrow0.
\]
It follows that
$\bar R_{0,r,z}^+\{C_r^+(z)\}\to1$.  By the definition of
$D_r^+(z)$, uniformly on $K$,
\[
 D_r^+(z)=\bar R_{0,r,z}^+\{C_r^+(z)\}\longrightarrow1.
\]
}
\end{proof}

\subsection{Proofs of Theorem~\ref{thm:one-lower} and
Propositions~\ref{prop:one-expansion} and~\ref{prop:one-strict}}

\begin{proof}[Proof of Theorem~\ref{thm:one-lower}]
Choose a compact set
\[
 E\subset\mathbb R\setminus\{0,y_q^+\}
\]
such that
\[
 \mathbb E\{D_+^*(Z)\1\{Z\in E\}\}>\ell_{\le}(q)-\frac\delta4.
\]
The proofs of Lemmas~\ref{lem:one-negative-regime},
\ref{lem:one-middle-regime}, and \ref{lem:one-positive-regime} give, for all
sufficiently small $r$, brackets
\[
 A_r(z)<B_r(z),\qquad z\in E,
\]
and a constant $\gamma>0$ such that
\[
 \sup_{z\in E}\sup_{u_q^+\leq c\leq A_r(z)}
 \bar R_{r,z}^+(c)\leq1-\gamma,
 \qquad
 \inf_{z\in E}\bar R_{r,z}^+\{B_r(z)\}\geq1+\gamma,
\]
\[
 \sup_{z\in E}\sup_{A_r(z)\leq c\leq B_r(z)}
 \left|\bar R_{0,r,z}^+(c)-D_+^*(z)\right|
 \leq\frac{\delta}{8}.
\]
Fix such an $r$, and put $C_*=\sup_{z\in E}B_r(z)<\infty$.  For
$z\in E$, define
\[
 \mathcal E_N(z)=
 \max\left\{
 \|\widehat R_N-R_{r,z}^+\|_{\infty,[0,1]},\,
 \|\widehat R_{0,N}-\pi_{0,r}^+R_{0,r,z}^+
 \|_{\infty,[0,1]}
 \right\}
\]
and
\[
 t_*=\frac{p_+(C_*)}{2q}
 \min\left\{\gamma,\frac{\delta}{8}\right\}.
\]
Following the same argument in the proof of
Theorem~\ref{thm:population-transfer}, Lemma~\ref{lem:DKW} gives
\[
 \sup_{z\in E}
 \mathbb P\{\mathcal E_N(z)>t_*\mid Z=z\}
 \leq4e^{-2Nt_*^2}.
\]
On the complementary event, the empirical and population curves differ
by at most
$q\mathcal E_N(z)/p_+(C_*)\leq
\frac12\min\{\gamma,\delta/8\}$ throughout
$[u_q^+,C_*]$.  The two population brackets and
Lemma~\ref{lem:exact-BH} therefore imply
\[
 A_r(z)<\widehat C_N\leq B_r(z),
 \qquad
 \FDP\geq D_+^*(z)-\frac{\delta}{4}.
\]
Consequently,
\begin{align*}
 \mathbb E_{\Omega_N}\{\FDR_{\Omega_N}(\BH_q)\}
 &=\mathbb E_{\Omega_N,Z,\varepsilon}(\FDP)\\
 &\geq
 \int_E\left\{D_+^*(z)-\frac{\delta}{4}\right\}\phi(z)\,\dd z
 -4e^{-2Nt_*^2}\\
 &>\ell_{\le}(q)-\frac{\delta}{2}-4e^{-2Nt_*^2}.
\end{align*}
For all sufficiently large $N$, this gives
\[
 \mathbb E_{\Omega_N}\{\FDR_{\Omega_N}(\BH_q)\}
 >\ell_{\le}(q)-\frac{3\delta}{4}.
\]
Therefore some deterministic realization $\omega$ satisfies the same
inequality.  The one-sided analogue of Lemma~\ref{lem:eta} permits an
$\eta>0$ for which the perturbed model has a positive definite correlation
matrix and
\[
 \FDR_\omega^{(\eta)}(\BH_q)>\ell_{\le}(q)-\delta.
\]
All means and null labels are unchanged by the perturbation.
Table~\ref{tab:one-construction} gives the pre-perturbation sign
pattern, and \eqref{eq:eta-cov} shows that every off-diagonal covariance is
multiplied by the positive factor $1-\eta^2$.  Thus the perturbation
preserves that sign pattern.
\end{proof}

\begin{proof}[Proof of Proposition~\ref{prop:one-expansion}]
Since $e^{-(y_q^+)^2/2}=q$, the Mills expansion in
Lemma~\ref{lem:mills} gives
\[
 \Phib(y_q^+)\sim\frac{q}{\sqrt{2\pi}\,y_q^+}=o(q).
\]
Substitute $y_q^+=\sqrt{2\log(1/q)}$ into \eqref{eq:one-ell}.
\end{proof}

\begin{proof}[Proof of Proposition~\ref{prop:one-strict}]
By \eqref{eq:one-Dstar}, $D_+^*(z)=q$ for $z\leq0$ and
$D_+^*(z)>q$ for every $z>0$.  Hence
$\ell_{\le}(q)=\mathbb E\{D_+^*(Z)\}>q$.
\end{proof}
\endgroup

\section{\texorpdfstring{Upper bounds}{Upper bounds}}
\label{sec:upper-bound}

\subsection{Preliminaries}

The upper bounds concern the common-factor Gaussian model
\begin{equation}
 T_i=\theta_i+a_iZ+\sqrt{1-a_i^2}\,\varepsilon_i,
 \quad 1\leq i\leq m,
 \label{eq:common-factor-model}
\end{equation}
where $Z,\varepsilon_1,\ldots,\varepsilon_m$ are i.i.d. standard
normal random variables and $|a_i|\leq1$.  Its covariance matrix is
diagonal plus rank one.  Conditional on $Z$, the coordinates and their
corresponding $p$-values are independent.

The following conditional leave-one-out bound will be used for
both types of tests.  Let $P_1,\ldots,P_m$ be conditionally independent
given $Z$, let $\mathcal H_0$ denote the true-null indices, and apply BH
at level $q$.  For $i\in\mathcal H_0$, write
$F_{i,z}(t)=\mathbb P(P_i\leq t\mid Z=z)$.
\begin{lemma}
\label{lem:conditional-fdp-bound}
For almost every $z$,
\begin{equation}
 \mathbb E(\FDP\mid Z=z)
 \leq\frac qm\sum_{i\in\mathcal H_0}
       \sup_{0<t\leq q}\frac{F_{i,z}(t)}t
 \label{eq:conditional-fdp-bound}
\end{equation}
\end{lemma}

\begin{proof}
Let $R$ be the BH rejection count.  For
$i\in\mathcal H_0$, let $R_i$ be the rejection count after replacing
$P_i$ by zero; in particular, $R_i\geq1$.  The standard leave-one-out
argument \citep{BenjaminiYekutieli2001} starts from the pointwise identity
\[
 \frac{\1\{i\text{ is rejected}\}}{R\vee1}
 =\frac{\1\{P_i\leq qR_i/m\}}{R_i}.
\]
In fact, if $i$ is rejected, replacing $P_i$ by any value below, including
zero, does not change other rejections because of the step-up rejection rule
for BH.  Taking conditional expectations given $Z=z$
and summing gives, for almost every $z$,

\begin{equation}
 \mathbb E(\FDP\mid Z=z)
 =\sum_{i\in\mathcal H_0}
   \mathbb E\left[
   \left.
   \frac{\1\{P_i\leq qR_i/m\}}{R_i}
   \right|Z=z\right].
 \label{eq:loo-identity}
\end{equation}

Conditional independence implies that $P_i$ and $R_i$ are independent
given $Z=z$.  Since $0<qR_i/m\leq q$, conditioning gives, for every
$i\in\mathcal H_0$ and almost every $z$,
\begin{equation}
\begin{aligned}
 \mathbb E\left[
   \left.
   \frac{\1\{P_i\leq qR_i/m\}}{R_i}
   \right| Z=z\right]
 &=\frac qm\,
   \mathbb E\left[
   \left.
   \frac{F_{i,z}(qR_i/m)}{qR_i/m}
   \right| Z=z\right]\\
 &\leq\frac qm\sup_{0<t\leq q}\frac{F_{i,z}(t)}t .
\end{aligned}
 \label{eq:conditional-loo-term-bound}
\end{equation}
Summing \eqref{eq:conditional-loo-term-bound} over
$i\in\mathcal H_0$ proves
\eqref{eq:conditional-fdp-bound}.
\end{proof}

\subsection{Upper bound for two-sided tests}

The common-factor class \eqref{eq:common-factor-model} contains the
lower-bound construction in Subsection~\ref{sec:continuum-design}.  The
following theorem shows that this construction attains the sharp leading
constant as $q\downarrow0$.  All constants below are universal and all
asymptotic statements are uniform over the common-factor model.

\begin{theorem}\label{thm:common-factor-upper}
For ordinary BH applied to the two-sided $p$-values $P_i=p(|T_i|)$,
let $\mathcal H_0=\{i:\theta_i=0\}$, $m_0=|\mathcal H_0|$, and
$\pi_0=m_0/m$.
Throughout $0<q\leq2\Phib(1)$, let $u_q$ be defined by
$p(u_q)=2\Phib(u_q)=q$.  If $m_0=0$, then $\FDR(\BH_q)=0$.  If
$m_0>0$, then the following bounds hold uniformly over the common-factor
model.
\begin{enumerate}[label=(\alph*)]
\item
\[
 \FDR(\BH_q)
 \leq q\left\{1+\pi_0\left(\frac{1-q}{2}
                    +\frac{u_q}{\sqrt{2\pi}}\right)\right\}.
\]
\item There is a universal constant $C<\infty$ such that
\[
 \FDR(\BH_q)
 \leq q\left[1+\frac{\pi_0}{2\sqrt\pi}\sqrt{\log(1/q)}
       +C\pi_0\{\log(1/q)\}^{1/4}\right].
\]
\end{enumerate}
\end{theorem}

\subsubsection{Proof of Theorem~\ref{thm:common-factor-upper}~(a)}

\begin{proof}
If $m_0=0$, then $V=0$ identically and hence $\FDR=0$.  Assume
henceforth that $m_0\geq1$.
Conditional on $Z=z$, each $P_i$ is a function of $\varepsilon_i$
alone, or is deterministic when $|a_i|=1$.  Thus the $p$-values are
conditionally independent, and Lemma~\ref{lem:conditional-fdp-bound} gives
\eqref{eq:conditional-fdp-bound}.

To average \eqref{eq:conditional-fdp-bound} over $Z$, it remains to control
{\[
 \int_{-u_q}^{u_q}
 \phi(z)\sup_{0<t\leq q}\frac{F_{i,z}(t)}t\dd z
\]}
for each true null $i$.  We do this on $|z|\leq u_q$; the event
$|Z|>u_q$ is handled at the end by $\FDP\leq1$.
Fix a true null index $i$.  Put $r=|a_i|$ and
$s=\sqrt{1-r^2}$, and first suppose $s>0$.  Conditional on $Z=z$,
$T_i\sim N(a_i z,s^2)$.  The inequality
$t=p(u)=2\Phib(u)\leq q$ holds if and only if $u\geq u_q$, and in this
case
{\[
 F_{i,z}(p(u))
 =\int_u^\infty s^{-1}\phi\{(x-a_i z)/s\}\dd x
  +\int_u^\infty s^{-1}\phi\{(x+a_i z)/s\}\dd x .
\]}
Since $s^2=1-a_i^2$,
{\[
\begin{aligned}
 \frac{s^{-1}\phi\{(x-a_i z)/s\}}{\phi(x)}
 &=\frac1s
   \exp\left\{-\frac{(x-a_i z)^2}{2s^2}+\frac{x^2}{2}\right\}\\
 &=\frac1s
   \exp\left\{\frac{z^2}{2}
        -\frac{(a_i x-z)^2}{2s^2}\right\}.
\end{aligned}
\]
The second tail is the same calculation with $z$ replaced by $-z$.  Choose
$\eta_+,\eta_-\in\{-1,1\}$ with $\{\eta_+,\eta_-\}=\{-1,1\}$
so that $\eta_+a_i z\geq0$ and
$\eta_-a_i z\leq0$.  For $x\geq u$,
\[
 |a_i x-\eta_+z|\geq (ru-|z|)_+,\qquad
 |a_i x-\eta_-z|\geq ru+|z|.
\]
Both right-hand sides are nondecreasing in $u$, so for $u\geq u_q$,
\[
\begin{aligned}
 \phi(z)\frac{F_{i,z}(p(u))}{p(u)}
 &\leq
 \frac{\phi(z)}{p(u)}\frac{e^{z^2/2}}s
 \left[
 e^{-(r u_q-|z|)_+^2/(2s^2)}
 +e^{-(r u_q+|z|)^2/(2s^2)}
 \right]\int_u^\infty\phi(x)\dd x\\
 &=\frac1{2s\sqrt{2\pi}}
 \left[
 e^{-(r u_q-|z|)_+^2/(2s^2)}
 +e^{-(r u_q+|z|)^2/(2s^2)}
 \right],
\end{aligned}
\]}
because $\int_u^\infty\phi(x)\dd x=p(u)/2$.  Taking the supremum gives
\begin{equation}
 \phi(z)\sup_{0<t\leq q}\frac{F_{i,z}(t)}t
 \leq\frac1{2s\sqrt{2\pi}}
 \left[
 e^{-(r u_q-|z|)_+^2/(2s^2)}+e^{-(r u_q+|z|)^2/(2s^2)}
 \right].                                                    \label{eq:conditional-tail-ratio-bound}
\end{equation}
Integrating \eqref{eq:conditional-tail-ratio-bound} over $|z|\leq u_q$ and using symmetry gives
{
\begin{equation}
\begin{aligned}
 \int_{-u_q}^{u_q}\phi(z)
     \sup_{0<t\leq q}\frac{F_{i,z}(t)}t\dd z
 &\leq\frac1{s\sqrt{2\pi}}
 \left\{\int_0^{u_q}e^{-(r u_q-z)_+^2/(2s^2)}\dd z
       +\int_0^{u_q}e^{-(r u_q+z)^2/(2s^2)}\dd z\right\}.
\end{aligned}
\label{eq:integrated-tail-ratio-decomposition}
\end{equation}}
{The two integrals on the right-hand side of
\eqref{eq:integrated-tail-ratio-decomposition} are
\[
\begin{aligned}
 \int_0^{u_q}e^{-(r u_q-z)_+^2/(2s^2)}\dd z
 &=\int_0^{r u_q}e^{-y^2/(2s^2)}\dd y+(1-r)u_q,\\
 \int_0^{u_q}e^{-(r u_q+z)^2/(2s^2)}\dd z
 &=\int_{r u_q}^{(1+r)u_q}e^{-y^2/(2s^2)}\dd y.
\end{aligned}
\]}
Consequently, writing
$\gamma_r=\sqrt{(1-r)/(1+r)}\in(0,1]$,
{\begin{equation}
\begin{aligned}
 \int_{-u_q}^{u_q}\phi(z)
     \sup_{0<t\leq q}\frac{F_{i,z}(t)}t\dd z
 &\leq\frac1{s\sqrt{2\pi}}
 \left\{\int_0^{(1+r)u_q}e^{-y^2/(2s^2)}\dd y+(1-r)u_q\right\}\\
 &=\left\{\Phi\left(\frac{(1+r)u_q}{s}\right)-\frac12\right\}
   +\frac{(1-r)u_q}{s\sqrt{2\pi}}\\
 &=\Phi(u_q/\gamma_r)-\frac12+\frac{u_q\gamma_r}{\sqrt{2\pi}},
\end{aligned}
\label{eq:integrated-tail-ratio-bound}
\end{equation}}
where the final equality in
\eqref{eq:integrated-tail-ratio-bound} uses
$s^2=1-r^2=(1-r)(1+r)$, so
$(1+r)/s=1/\gamma_r$ and $(1-r)/s=\gamma_r$.
The right-hand side of
\eqref{eq:integrated-tail-ratio-bound} is increasing in
$\gamma_r\in(0,1]$.  Its derivative with respect to $\gamma_r$ is
\[
 \frac{u_q}{\sqrt{2\pi}}
 \left[1-\gamma_r^{-2}\exp\{-u_q^2/(2\gamma_r^2)\}\right]\geq0,
\]
because $u_q\geq1$ and, on putting $w=\gamma_r^{-2}\geq1$,
\[
 w e^{-u_q^2w/2}\leq w e^{-w/2}\leq 2/e<1.
\]
Hence \eqref{eq:integrated-tail-ratio-bound} is at most
\begin{equation}
 \Phi(u_q)-\frac12+\frac{u_q}{\sqrt{2\pi}}
 =\frac{1-q}{2}+\frac{u_q}{\sqrt{2\pi}}.                       \label{eq:integrated-tail-ratio-uniform}
\end{equation}
When $r=1$, the conditional $p$-value is deterministic.  For
$|z|<u_q$, it exceeds $q$, and the boundary $|z|=u_q$ has Lebesgue
measure zero; hence \eqref{eq:integrated-tail-ratio-uniform} remains valid.
When $r=0$, the conditional null $p$-value is uniform and independent of
$Z$; \eqref{eq:integrated-tail-ratio-uniform} remains valid, although it is
not sharp.

Let $A=\{|Z|\leq u_q\}$.  Decomposing over $A$ and $A^c$,
{\[
 \FDR=\mathbb E(\FDP\1_A)
      +\mathbb E(\FDP\1_{A^c}).
\]}
For the first term, integrate \eqref{eq:conditional-fdp-bound} over $z\in[-u_q,u_q]$ and
use \eqref{eq:integrated-tail-ratio-uniform}.  For the second term, no conditional tail bound is
needed: since $0\leq\FDP\leq1$,
{
\begin{equation}
 \mathbb E(\FDP\1_{A^c})
 \leq\mathbb P(A^c)
 =\mathbb P(|Z|>u_q)
 =2\Phib(u_q)
 =q.
\label{eq:outer-factor-fdp-bound}
\end{equation}}
Combining \eqref{eq:conditional-fdp-bound},
\eqref{eq:integrated-tail-ratio-uniform}, and
\eqref{eq:outer-factor-fdp-bound} yields
\[
 \FDR
 \leq q+\frac qm\sum_{i\in\mathcal H_0}
       \left(\frac{1-q}{2}+\frac{u_q}{\sqrt{2\pi}}\right)
 =q\left\{1+\pi_0\left(\frac{1-q}{2}
       +\frac{u_q}{\sqrt{2\pi}}\right)\right\},
\]
as claimed.
\end{proof}

\subsubsection{Proof of Theorem~\ref{thm:common-factor-upper}~(b)}

The proof of part (b) has two parts.  First we prove a deterministic statement
for the rank-weighted contribution that will arise from the finite BH threshold.
We then verify its hypotheses from the finite threshold by a uniform
concentration argument for Poisson-binomial counts, that is, sums of
independent Bernoulli random variables with possibly different success
probabilities.
Throughout the proof, $C$ and $c$ denote positive universal constants
whose values may change from line to line.

For $0\leq r<1$, define
\begin{equation}
 \rho_r(u,z)=
 \frac{\Phib\{(u-rz)/\sqrt{1-r^2}\}
       +\Phib\{(u+rz)/\sqrt{1-r^2}\}}
      {2\Phib(u)}.                                   \label{eq:upper-rho}
\end{equation}
For $r=1$, set $\rho_1(u,z)=\mathbf1\{|z|\geq u\}/p(u)$; the convention
on the measure-zero boundary $|z|=u$ is immaterial.
For every true null $i$, this is its conditional two-sided tail ratio:
\begin{equation}
 \frac{F_{i,z}\{p(u)\}}{p(u)}=\rho_{|a_i|}(u,z),       \label{eq:upper-null-tail-ratio}
\end{equation}
because the two Gaussian tails are unchanged when $a_i$ is replaced by
$|a_i|$.
Integrating \eqref{eq:upper-null-tail-ratio} over the common factor gives
the unconditional normalization: for every true null $i$ and every
$u>0$,
\begin{equation}
 \int_{\mathbb R}\rho_{|a_i|}(u,z)\phi(z)\dd z
 =\frac{1}{p(u)}
   \int_{\mathbb R}F_{i,z}\{p(u)\}\phi(z)\dd z
 =\frac{\mathbb P(P_i\leq p(u))}{p(u)}
 =1.                                                \label{eq:upper-rho-normalization}
\end{equation}
The middle equality follows because
\[
 F_{i,z}\{p(u)\}
 =\mathbb P(P_i\leq p(u)\mid Z=z),
\]
and $Z$ has density $\phi$.  Hence
\[
 \int_{\mathbb R}F_{i,z}\{p(u)\}\phi(z)\dd z
 =\mathbb E\{\mathbb P(P_i\leq p(u)\mid Z)\}
 =\mathbb P(P_i\leq p(u)),
\]
which is the tower property.  The last equality in
\eqref{eq:upper-rho-normalization} uses that $P_i$ is a two-sided null
$p$-value.

We first give a useful bound on $\rho_r(u,z)$ without restrictions on $r$.
\begin{lemma}\label{lem:upper-null-ratio}
If $u_q\geq1$ and $0\leq r<1$, then
\begin{equation*}
 \int_{-u_q}^{u_q}\phi(z)
       \sup_{u\geq u_q}\rho_r(u,z)\dd z
 \leq
 \Phi\left(\frac{(1+r)u_q}{\sqrt{1-r^2}}\right)-\frac12
       +\frac{(1-r)u_q}{\sqrt{1-r^2}\sqrt{2\pi}}.
\end{equation*}
For $r=1$, the integral on the left is zero.
\end{lemma}

\begin{proof}
Assume first that $0\leq r<1$.  For $\xi\in\{-1,1\}$, using
$\phi(t)=(2\pi)^{-1/2}\exp(-t^2/2)$,
\begin{align*}
 \frac{\phi\{(x-\xi rz)/\sqrt{1-r^2}\}}
      {\sqrt{1-r^2}\,\phi(x)}
 &=\frac1{\sqrt{1-r^2}}
   \exp\left\{-\frac{(x-\xi rz)^2}{2(1-r^2)}
              +\frac{x^2}{2}\right\}\\
 &=\frac1{\sqrt{1-r^2}}\exp\left\{\frac{z^2}{2}
            -\frac{(rx-\xi z)^2}{2(1-r^2)}\right\}.
\end{align*}
The second equality uses the identity
\[
 (1-r^2)z^2-(rx-\xi z)^2
 =(1-r^2)x^2-(x-\xi rz)^2
 =-r^2x^2+2\xi rxz-r^2z^2.
\]
If $\xi z=|z|$, then
\[
 |rx-\xi z|=|rx-|z||\geq (rx-|z|)_+
       \geq (r u_q-|z|)_+,
\]
while if $\xi z=-|z|$, then
\[
 |rx-\xi z|=rx+|z|\geq r u_q+|z|.
\]
For fixed $u\geq u_q$, each conditional tail in the numerator of
$\rho_r(u,z)$ can be written as
\[
 \Phib\{(u-\xi rz)/\sqrt{1-r^2}\}
 =\int_u^\infty
   \frac{\phi\{(x-\xi rz)/\sqrt{1-r^2}\}}{\sqrt{1-r^2}}\dd x .
\]
Multiplying the density-ratio identity above by $\phi(z)$ and using the two
lower bounds on $|rx-\xi z|$, we get, for $|z|\leq u_q$,
\begin{align*}
 \phi(z)\rho_r(u,z)
&=\frac1{p(u)}
  \sum_{\xi\in\{-1,1\}}\int_u^\infty
   \phi(x)\phi(z)
   \frac{\phi\{(x-\xi rz)/\sqrt{1-r^2}\}}
        {\sqrt{1-r^2}\,\phi(x)}\dd x\\
&\leq
 \frac{\int_u^\infty\phi(x)\dd x}{p(u)}
 \frac1{\sqrt{1-r^2}\sqrt{2\pi}}
 \left[
 e^{-(r u_q-|z|)_+^2/\{2(1-r^2)\}}
 +e^{-(r u_q+|z|)^2/\{2(1-r^2)\}}
 \right].
\end{align*}
Since $p(u)=2\int_u^\infty\phi(x)\dd x$, and the right side no longer
depends on $u$, taking the supremum over $u\geq u_q$ gives
\[
 \phi(z)\sup_{u\geq u_q}\rho_r(u,z)
 \leq\frac1{2\sqrt{1-r^2}\sqrt{2\pi}}
 \left[
 e^{-(r u_q-|z|)_+^2/\{2(1-r^2)\}}
 +e^{-(r u_q+|z|)^2/\{2(1-r^2)\}}
 \right].
\]
Integrating this bound over $z\in[-u_q,u_q]$ and using symmetry reduces the
right side to the integral over $0\leq z\leq u_q$ without the factor $1/2$.
The first exponential contributes
\begin{align*}
 \int_0^{u_q}e^{-(r u_q-z)_+^2/\{2(1-r^2)\}}\dd z
 &=\int_0^{r u_q}e^{-(r u_q-z)^2/\{2(1-r^2)\}}\dd z
   +\int_{r u_q}^{u_q}1\,\dd z\\
 &=\int_0^{r u_q}e^{-y^2/\{2(1-r^2)\}}\dd y+(1-r)u_q,
\end{align*}
where the last line uses $y=r u_q-z$ in the first integral.  The second exponential
contributes
\[
 \int_0^{u_q}e^{-(r u_q+z)^2/\{2(1-r^2)\}}\dd z
 =\int_{r u_q}^{(1+r)u_q}e^{-y^2/\{2(1-r^2)\}}\dd y,
\]
using $y=r u_q+z$.  Therefore
\begin{align*}
&\int_{-u_q}^{u_q}\phi(z)
       \sup_{u\geq u_q}\rho_r(u,z)\dd z
\\
&\leq \frac1{\sqrt{1-r^2}\sqrt{2\pi}}
 \bigg\{\int_0^{r u_q}e^{-y^2/\{2(1-r^2)\}}\dd y
       +\int_{r u_q}^{(1+r)u_q}e^{-y^2/\{2(1-r^2)\}}\dd y+(1-r)u_q\bigg\}\\
&=\Phi\{(1+r)u_q/\sqrt{1-r^2}\}-\frac12
       +\frac{(1-r)u_q}{\sqrt{1-r^2}\sqrt{2\pi}},
\end{align*}
which is the asserted bound.

When $r=1$, the conditional residual variance $1-r^2$ is zero.  Thus a
null statistic with $|a_i|=1$ is $T_i=a_iZ$, so conditional on $Z=z$
it equals $\pm z$ and has absolute value $|z|$.  For $|z|\leq u_q$,
\begin{equation}
 \{u\geq u_q:|z|\geq u\}
 =\begin{cases}
 \varnothing,&|z|<u_q,\\
 \{u_q\},&|z|=u_q.
 \end{cases}
 \label{eq:upper-degenerate-threshold-set}
\end{equation}
By \eqref{eq:upper-degenerate-threshold-set} and $p(u_q)=q$,
\begin{align}
 \sup_{u\geq u_q}\rho_1(u,z)
 &=\frac{\mathbf1\{|z|=u_q\}}q,
 \qquad |z|\leq u_q,\label{eq:upper-degenerate-ratio-sup}\\
 \int_{-u_q}^{u_q}\phi(z)\sup_{u\geq u_q}\rho_1(u,z)\dd z
 &=\frac1q\int_{\{-u_q,u_q\}}\phi(z)\dd z=0.
 \label{eq:upper-degenerate-ratio-integral}
\end{align}
Changing the convention at $|z|=u$ affects
\eqref{eq:upper-degenerate-ratio-sup} only at $z=\pm u_q$, so
\eqref{eq:upper-degenerate-ratio-integral} is unchanged.

\end{proof}

Define
\begin{equation}
 K_q=\sqrt{8\log u_q},\qquad L_q=K_q+\sqrt{u_q}.
\label{eq:upper-KLq-def}
\end{equation}
The following lemma shows a refined bound on $\rho_r(u, z)$ for small $r$.
\begin{lemma}\label{lem:upper-localization}
With $K_q$ and $L_q$ as defined in \eqref{eq:upper-KLq-def}, uniformly for
$0\leq r\leq3/5$, $u\geq u_q$, and $|z|\leq u_q$,
\begin{equation}
 \phi(z)\rho_r(u,z)
 \leq C\left[
 e^{-(z-ru)^2/\{2(1-r^2)\}}
 +e^{-(z+ru)^2/\{2(1-r^2)\}}\right].   \label{eq:upper-localization}
\end{equation}
Moreover, there are universal constants $C<\infty$ and $q_0>0$ such that,
whenever $q\leq q_0$, if
$|z|\geq L_q$ and
$\big||z|-ru\big|\leq K_q$, then
\begin{equation}
 \phi(z)\rho_r(u,z)
 \leq \frac{1+C(K_q+1)/u_q}{2\sqrt{2\pi}\sqrt{1-r^2}}
       +C\exp\{-u_q\}.                              \label{eq:upper-localization-quantitative}
\end{equation}
\end{lemma}

\begin{proof}
Since $r\leq3/5$, $u\geq u_q$, and
$|z|\leq u_q$,
\[
 u-rz\geq u-r|z|\geq u-ru_q\geq (1-r)u\geq \frac25u\geq\frac25u_q.
\]
Since $\sqrt{1-r^2}\leq1$, this gives
\[
 \frac{u-rz}{\sqrt{1-r^2}}\geq\frac25u_q.
\]
Hence the upper Mills bound
\eqref{eq:mills-upper} in Appendix~A applies to
$(u-rz)/\sqrt{1-r^2}$, while the lower Mills bound
\eqref{eq:mills-lower} applies to $u$.  These two inequalities give
\begin{align}
 \frac{\phi(z)\Phib\{(u-rz)/\sqrt{1-r^2}\}}{2\Phib(u)}
 &\leq
 \frac{\phi(z)\phi\{(u-rz)/\sqrt{1-r^2}\}}{2\phi(u)}
 \frac{\sqrt{1-r^2}(u^2+1)}{u(u-rz)} \nonumber\\
 &= \frac1{2\sqrt{2\pi}}
 \frac{\sqrt{1-r^2}(u^2+1)}{u(u-rz)}
 \exp\left\{
 -\frac{z^2}{2}
 -\frac{(u-rz)^2}{2(1-r^2)}
 +\frac{u^2}{2}
 \right\} \nonumber\\
 &= \frac1{2\sqrt{2\pi}}
 \frac{\sqrt{1-r^2}(u^2+1)}{u(u-rz)}
 \exp\left\{
 -\frac{(1-r^2)z^2+(u-rz)^2-(1-r^2)u^2}
 {2(1-r^2)}
 \right\} \nonumber\\
 &= \frac1{2\sqrt{2\pi}}
 \frac{\sqrt{1-r^2}(u^2+1)}{u(u-rz)}
 \exp\left\{-\frac{(z-ru)^2}{2(1-r^2)}\right\}.  \label{eq:upper-mills-localization}
\end{align}
Using a similar argument, we can prove
\begin{align}
 \frac{\phi(z)\Phib\{(u+rz)/\sqrt{1-r^2}\}}{2\Phib(u)}
 &\leq \frac1{2\sqrt{2\pi}}
\frac{\sqrt{1-r^2}(u^2+1)}{u(u+rz)}
 \exp\left\{-\frac{(z+ru)^2}{2(1-r^2)}\right\}.  \label{eq:upper-mills-localization-minus}
\end{align}
For both signs, $|z|\leq u_q\leq u$ implies
\[
 u\mp rz\geq u-r|z|\geq (1-r)u\geq \frac25 u.
\]
Also $\sqrt{1-r^2}\leq1$ and $u\geq u_q\geq1$.  Therefore
\[
 \max\left\{
 \frac{\sqrt{1-r^2}(u^2+1)}{u(u-rz)},
 \frac{\sqrt{1-r^2}(u^2+1)}{u(u+rz)}
 \right\}
 \leq
 \frac{u^2+1}{(2/5)u^2}
 \leq 5.
\]
This bound is uniform over $0\leq r\leq3/5$, $u\geq u_q$, and
$|z|\leq u_q$.  Substituting it into
\eqref{eq:upper-mills-localization} and
\eqref{eq:upper-mills-localization-minus}, and increasing $C$, proves
\eqref{eq:upper-localization}.  If $|z|=ru+d$ with
$|d|\leq K_q$, then the first term in
\eqref{eq:upper-mills-localization}-\eqref{eq:upper-mills-localization-minus}
has prefactor
\begin{equation}
 \frac{\sqrt{1-r^2}(u^2+1)}{u\{(1-r^2)u-rd\}}
 =\frac1{\sqrt{1-r^2}}\,
   \frac{1+u^{-2}}{1-rd/\{(1-r^2)u\}}.               \label{eq:upper-localization-prefactor}
\end{equation}
Because $r\leq3/5$, one has $1-r^2\geq16/25$.  Hence, when
$|d|\leq K_q\leq u_q/4$,
\[
 \left|\frac{rd}{(1-r^2)u}\right|
 \leq
 \frac{(3/5)K_q}{(16/25)u_q}
 \leq \frac{15}{64}.
\]
Since $(1-x)^{-1}\leq1+C|x|$ for $|x|\leq15/64$, and since
$u\geq u_q\geq1$,
\[
 \frac{1+u^{-2}}{1-rd/\{(1-r^2)u\}}
 \leq
 (1+u_q^{-1})\left(1+C\frac{K_q}{u_q}\right)
 \leq 1+C\frac{K_q+1}{u_q}.
\]
Combining this estimate with \eqref{eq:upper-localization-prefactor} gives
\[
\begin{aligned}
 \frac{\sqrt{1-r^2}(u^2+1)}{u\{(1-r^2)u-rd\}}
 &=\frac1{\sqrt{1-r^2}}\,
   \frac{1+u^{-2}}{1-rd/\{(1-r^2)u\}}\\
 &\leq
 \frac1{\sqrt{1-r^2}}
 \left\{1+C\frac{K_q+1}{u_q}\right\}\\
 &=
 \frac{1+C(K_q+1)/u_q}{\sqrt{1-r^2}}.
\end{aligned}
\]
Note that
\[
\begin{aligned}
 z\geq0:\quad
 |z+ru|
 &=|z|+ru=2|z|-d
 \geq 2L_q-K_q,\\
 z<0:\quad
 |z-ru|
 &=|z|+ru=2|z|-d
 \geq 2L_q-K_q.
\end{aligned}
\]
Therefore, with the $+$ sign when $z\geq0$ and the $-$ sign when $z<0$,
\begin{equation}
\exp\left\{-\frac{(z\pm ru)^2}{2(1-r^2)}\right\}
\leq
\exp\left\{-\frac{(2L_q-K_q)^2}{2(1-r^2)}\right\}
\leq \exp\{-u_q\},
\label{eq:upper-localization-opposite-exponential}
\end{equation}
because $1-r^2\leq1$ and $(2L_q-K_q)^2/2\geq u_q$.
The proof is then completed by
\eqref{eq:upper-mills-localization} - \eqref{eq:upper-localization-opposite-exponential}.
\end{proof}

For the next three lemmas, fix $0<r\leq3/5$ and use $K_q,L_q$ from
\eqref{eq:upper-KLq-def}.  Let $\Pi$ be a positive measure of mass at most
one on $(\theta,\alpha)\in\mathbb R\times[-1,1]$, let
$\varepsilon\sim N(0,1)$, and define
\begin{equation}
 G_z(u)=\int \mathbb P\{|\theta+\alpha z+\sqrt{1-\alpha^2}\,\varepsilon|
                    \geq u\}\,\dd\Pi(\theta,\alpha). \label{eq:upper-mixture}
\end{equation}
For $z$ and $u$, write
\[
\begin{aligned}
G_z^{\mathrm{large}}(u)
&=\int_{\{|\alpha|\ge r/u_q\}}
 \mathbb P\{|\theta+\alpha z+\sqrt{1-\alpha^2}\,\varepsilon|\ge u\}
 \,\dd\Pi(\theta,\alpha),\\
G_z^{\mathrm{small}}(u)
&=\int_{\{|\alpha|< r/u_q\}}
 \mathbb P\{|\theta+\alpha z+\sqrt{1-\alpha^2}\,\varepsilon|\ge u\}
 \,\dd\Pi(\theta,\alpha).
\end{aligned}
\]
Then
\[
G_z(u)=G_z^{\mathrm{large}}(u)+G_z^{\mathrm{small}}(u).
\]
For each $z$, let $\mathcal C_z$ be a finite set satisfying
\begin{equation}
 \mathcal C_z\subseteq\left\{u\geq u_q:\ G_z(u)\geq\frac{p(u)}{2q},\
 G_z(v)\leq3\frac{p(v)}q\ \forall\,u_q\leq v<u\right\}. \label{eq:upper-threshold-class}
\end{equation}
Define
\begin{equation}
\begin{aligned}
E_\pm^{\mathrm{large}}(r)
=\Big\{z:\;&\pm z\in[L_q,u_q],\ \exists u\text{ such that }
u\in\mathcal C_z,\ \big||z|-ru\big|\leq K_q,\\
&\text{and }G_z^{\mathrm{large}}(u)\geq p(u)/(4q)\Big\},
\end{aligned}                                      \label{eq:upper-large-radii-def}
\end{equation}
and
\begin{equation}
\begin{aligned}
E_\pm^{\mathrm{small}}(r)
=\Big\{z:\;&\pm z\in[L_q,u_q],\ \exists u\text{ such that }
u\in\mathcal C_z,\ \big||z|-ru\big|\leq K_q,\\
&\text{and }G_z^{\mathrm{small}}(u)\geq p(u)/(4q)\Big\}.
\end{aligned}                                      \label{eq:upper-small-radii-def}
\end{equation}
The next lemma bounds $|E_\pm^{\mathrm{large}}(r)|$.
\begin{lemma}\label{lem:upper-large-coefficients}
There are universal constants $C<\infty$ and $q_0>0$ such that, whenever
$q\leq q_0$, the following bound holds uniformly over all positive measures
$\Pi$ on $\mathbb R\times[-1,1]$ with
$\Pi(\mathbb R\times[-1,1])\leq1$, and over all families
$(\mathcal C_z)$ satisfying \eqref{eq:upper-threshold-class}:
\begin{equation}
 |E_+^{\mathrm{large}}(r)|+|E_-^{\mathrm{large}}(r)|
 \leq(1-r)u_q+C\sqrt{u_q}.                             \label{eq:upper-large-coefficient-range}
\end{equation}
\end{lemma}

\begin{proof}
Set
\begin{equation}
\begin{aligned}
Y_z&=Y_z(\theta,\alpha,e)=\theta+\alpha z+\sqrt{1-\alpha^2}\,e,\\
\dd\nu(\theta,\alpha,e)&=\frac13\,\dd\Pi(\theta,\alpha)\phi(e)\dd e .
\end{aligned}                                      \label{eq:upper-Yz-def}
\end{equation}
Then
\begin{align}
\nu\{|Y_z|\geq v\}&=\frac13G_z(v), \notag\\
\nu\{|\alpha|\geq r/u_q,\ |Y_z|\geq v\}
   &=\frac13G_z^{\mathrm{large}}(v).                  \label{eq:upper-large-path-identity}
\end{align}
If $u\in\mathcal C_z$ and $u_q\leq v<u$, then
\eqref{eq:upper-threshold-class} gives
\begin{equation}
\nu\{|Y_z|\geq v\}\leq\frac{p(v)}q.                 \label{eq:upper-admissible-tail-bound}
\end{equation}
If $z\in E_+^{\mathrm{large}}(r)$ and
$u\in\mathcal C_z$ is as in \eqref{eq:upper-large-radii-def}, then by
\eqref{eq:upper-large-path-identity},
\begin{equation}
 \nu\{|\alpha|\geq r/u_q,\ |Y_z|\geq u\}
 \geq \frac{p(u)}{12q}.                             \label{eq:upper-high-mass}
\end{equation}
The same result holds with $z\in E_-^{\mathrm{large}}(r)$ and
$u\in\mathcal C_z$.

Fix any $z_1\in E_+^{\mathrm{large}}(r)$.  Consider
$z_2\in E_+^{\mathrm{large}}(r)$ and
$-z_3\in E_-^{\mathrm{large}}(r)$ with $z_2>z_1$.  Choose
$v_1\in\mathcal C_{z_1}$, $v_2\in\mathcal C_{z_2}$, and
$v_3\in\mathcal C_{-z_3}$ as in \eqref{eq:upper-large-radii-def}, and put
\begin{equation}
 d_2=\frac r{u_q}(z_2-z_1),\qquad
 d_3=\frac r{u_q}(z_1+z_3).                         \label{eq:upper-three-increments}
\end{equation}
Split
$\{|\alpha|\ge r/u_q,\ |Y_{z_1}|\ge v_1\}=A_+\dot\cup A_-$, where
\[
\begin{aligned}
A_+&=\{|\alpha|\ge r/u_q,\ |Y_{z_1}|\ge v_1,\ \alpha Y_{z_1}\ge0\},\\
A_-&=\{|\alpha|\ge r/u_q,\ |Y_{z_1}|\ge v_1,\ \alpha Y_{z_1}<0\}.
\end{aligned}
\]
Then \eqref{eq:upper-high-mass} gives
\[
\nu(A_+)+\nu(A_-)
=\nu\{|\alpha|\ge r/u_q,\ |Y_{z_1}|\ge v_1\}
\geq\frac{p(v_1)}{12q}.
\]
Moreover,
\[
A_+\subseteq\{|Y_{z_2}|\geq v_1+d_2\},\qquad
A_-\subseteq\{|Y_{-z_3}|\geq v_1+d_3\}.
\]
Indeed, on $A_+$,
\[
|Y_{z_2}|=|Y_{z_1}+\alpha(z_2-z_1)|
\geq |Y_{z_1}|+|\alpha|(z_2-z_1)\geq v_1+d_2,
\]
while on $A_-$,
\[
|Y_{-z_3}|=|Y_{z_1}-\alpha(z_1+z_3)|
\geq |Y_{z_1}|+|\alpha|(z_1+z_3)\geq v_1+d_3.
\]

Let
\[
 \widetilde E_-^{\mathrm{large}}(r)
 =\{-z:\ z\in E_-^{\mathrm{large}}(r)\}.
\]
It is enough to bound
$|E_+^{\mathrm{large}}(r)|+|\widetilde E_-^{\mathrm{large}}(r)|$.
If $E_+^{\mathrm{large}}(r)\cup\widetilde E_-^{\mathrm{large}}(r)$ is empty,
there is nothing to prove.  Otherwise choose
$z_1\in E_+^{\mathrm{large}}(r)\cup\widetilde E_-^{\mathrm{large}}(r)$
arbitrarily close to
\[
\inf\{z:\ z\in E_+^{\mathrm{large}}(r)\cup
\widetilde E_-^{\mathrm{large}}(r)\}.
\]
Choose
$v_1\in\mathcal C_{z_1}$ if $z_1\in E_+^{\mathrm{large}}(r)$, and
$v_1\in\mathcal C_{-z_1}$ if
$z_1\in\widetilde E_-^{\mathrm{large}}(r)$.  In the first case,
\eqref{eq:upper-large-radii-def} gives
\[
 z_1\in[L_q,u_q],\qquad |z_1-rv_1|\leq K_q.
\]
In the second case, $-z_1\in E_-^{\mathrm{large}}(r)$, so the same
definition gives
\[
 z_1=|-z_1|\in[L_q,u_q],\qquad |z_1-rv_1|\leq K_q.
\]
Thus, in either case,
\begin{equation}
 z_1\geq L_q,\quad |z_1-rv_1|\leq K_q
 \quad\Longrightarrow\quad
 \frac{r v_1}{u_q}\geq\frac{z_1-K_q}{u_q}
 \geq\frac{L_q-K_q}{u_q}.                            \label{eq:upper-ru-lower}
\end{equation}
Set
\begin{equation}
 B_q=\frac{2K_q+2r^2}{1-r^2/u_q}
       +\frac{4u_q}{L_q-K_q}.                            \label{eq:upper-buffer}
\end{equation}
We claim that, to the right of $z_1+B_q$, the two reflected sets cannot
overlap:
\begin{equation}
\bigl(E_+^{\mathrm{large}}(r)\cap(z_1+B_q,u_q]\bigr)
\cap
\bigl(\widetilde E_-^{\mathrm{large}}(r)\cap(z_1+B_q,u_q]\bigr)=\emptyset .
                                                        \label{eq:upper-reflected-disjoint-tail}
\end{equation}
We prove \eqref{eq:upper-reflected-disjoint-tail} by contradiction.  Assume
\eqref{eq:upper-reflected-disjoint-tail} fails.  Then there exist
\begin{equation}
 z_2\in E_+^{\mathrm{large}}(r),\qquad
 z_3\in \widetilde E_-^{\mathrm{large}}(r),\qquad
 z_2,z_3>z_1+B_q .                                  \label{eq:upper-no-reflected-tail-pair}
\end{equation}
If $z_1\in E_+^{\mathrm{large}}(r)$, assume
\eqref{eq:upper-no-reflected-tail-pair} and choose
$v_2\in\mathcal C_{z_2}$ and $v_3\in\mathcal C_{-z_3}$.  By
\eqref{eq:upper-large-radii-def},
\[
 |z_1-rv_1|\leq K_q,\qquad |z_2-rv_2|\leq K_q,\qquad
 |z_3-rv_3|\leq K_q,
\]
where the last inequality uses $-z_3\in E_-^{\mathrm{large}}(r)$, so that
$|-z_3|=z_3$.  Hence
\[
\begin{aligned}
 r(v_2-v_1)
 &= (rv_2-z_2)+(z_2-z_1)+(z_1-rv_1)
    \geq z_2-z_1-2K_q,\\
 r(v_3-v_1)
 &= (rv_3-z_3)+(z_3-z_1)+(z_1-rv_1)
    \geq z_3-z_1-2K_q.
\end{aligned}
\]
Therefore
\begin{equation}
 v_2-v_1\geq\frac{z_2-z_1-2K_q}{r},\qquad
 v_3-v_1\geq\frac{z_3-z_1-2K_q}{r}.              \label{eq:upper-alignment-threshold-lower}
\end{equation}
Since $z_2,z_3>z_1+B_q$,
\[
 z_2-z_1,\ z_3-z_1>
 \frac{2K_q+2r^2}{1-r^2/u_q}.
\]
Therefore
\[
 (z_2-z_1)\left(1-\frac{r^2}{u_q}\right)>2K_q,
\]
and, using $z_1\leq u_q$,
\[
 (z_3-z_1)\left(1-\frac{r^2}{u_q}\right)>2K_q+2r^2
 \geq 2K_q+\frac{2r^2z_1}{u_q}.
\]
Equivalently,
\[
 \frac{z_2-z_1-2K_q}{r}>\frac r{u_q}(z_2-z_1),\qquad
 \frac{z_3-z_1-2K_q}{r}>\frac r{u_q}(z_1+z_3).
\]
Combining these inequalities with
\eqref{eq:upper-alignment-threshold-lower} gives
\[
 v_2-v_1>\frac r{u_q}(z_2-z_1),\qquad
 v_3-v_1>\frac r{u_q}(z_1+z_3),
\]
that is,
\begin{equation}
 v_1+d_2<v_2,\qquad v_1+d_3<v_3.
\label{eq:upper-three-ordering}
\end{equation}
Also, \eqref{eq:upper-ru-lower} gives
\[
 \frac{r v_1}{u_q}\geq\frac{L_q-K_q}{u_q}.
\]
Since $z_2,z_3>z_1+B_q$ and the second term in $B_q$ is
$4u_q/(L_q-K_q)$,
\[
 z_2-z_1>\frac{4u_q}{L_q-K_q},\qquad
 z_1+z_3>z_3-z_1>\frac{4u_q}{L_q-K_q}.
\]
Consequently
\[
 \frac{r v_1}{u_q}\min\{z_2-z_1,z_1+z_3\}
 \geq
 \frac{L_q-K_q}{u_q}\min\{z_2-z_1,z_1+z_3\}>4
\]
and hence, by \eqref{eq:upper-three-increments},
\begin{equation}
 \min\{d_2,d_3\}>\frac4{v_1}.                         \label{eq:upper-min-increment}
\end{equation}

\indent The ordering in \eqref{eq:upper-three-ordering}, together with
$v_1\geq u_q$ and $d_2,d_3\geq0$, gives
\begin{equation}
 u_q\leq v_1+d_2<v_2,\qquad u_q\leq v_1+d_3<v_3.
 \label{eq:upper-admissible-threshold-ordering}
\end{equation}
Thus $v_2\in\mathcal C_{z_2}$, $v_3\in\mathcal C_{-z_3}$, so
\eqref{eq:upper-admissible-tail-bound} applies with
$(z,u,v)=(z_2,v_2,v_1+d_2)$ and
$(z,u,v)=(-z_3,v_3,v_1+d_3)$.  Hence
\[
\begin{aligned}
\frac{p(v_1)}{12q}
&\leq \nu(A_+)+\nu(A_-)\\
&\leq \nu\{|Y_{z_2}|\geq v_1+d_2\}
     +\nu\{|Y_{-z_3}|\geq v_1+d_3\}\\
&\leq \frac{p(v_1+d_2)+p(v_1+d_3)}q,
   \qquad\text{by \eqref{eq:upper-admissible-tail-bound}.}
\end{aligned}
\]
Multiplying by $q$ and substituting the definitions of $d_2$ and $d_3$
in \eqref{eq:upper-three-increments} imply
\begin{equation}
 \frac{p(v_1)}{12}
 \leq p\{v_1+(r/u_q)(z_2-z_1)\}
     +p\{v_1+(r/u_q)(z_1+z_3)\}.                    \label{eq:upper-three-crossing}
\end{equation}
For $u>0$ and $d\geq0$, $p=2\Phib$ and the normal hazard
$\lambda(t)=\phi(t)/\Phib(t)$ give
\begin{equation}
\begin{aligned}
\frac{p(u+d)}{p(u)}
&=\frac{\Phib(u+d)}{\Phib(u)}
 =\exp\left\{-\int_u^{u+d}\lambda(t)\,\dd t\right\}  \\
&\leq \exp\left\{-\int_u^{u+d}t\,\dd t\right\}
 =\exp\{-ud-d^2/2\}\leq e^{-ud},
\end{aligned}                                      \label{eq:upper-tail-transport}
\end{equation}
where the inequality uses $\lambda(t)>t$ for $t>0$, as stated in
\eqref{eq:mills-hazard} of Appendix Lemma~\ref{lem:mills}.
Thus the right side of \eqref{eq:upper-three-crossing} is at most
\[
 \frac{p(v_1)}{48}+\frac{p(v_1)}{48}<\frac{p(v_1)}{12},
\]
which is impossible.  Equivalently, the tail calculation gives
\begin{equation}
 \frac r{u_q}\min\{z_2-z_1,z_1+z_3\}
 =\min\{d_2,d_3\}\leq \frac4{v_1}.                 \label{eq:upper-three-obstruction}
\end{equation}
This contradicts \eqref{eq:upper-min-increment}.  By the same argument with signs
reversed, if $-z_1,-z_3\in E_-^{\mathrm{large}}(r)$,
$z_2\in E_+^{\mathrm{large}}(r)$, and $z_3>z_1$, then, for
$v_1\in\mathcal C_{-z_1}$ chosen as in
\eqref{eq:upper-large-radii-def},
\begin{equation}
 \frac r{u_q}\min\{z_3-z_1,z_1+z_2\}\leq \frac4{v_1}.
 \label{eq:upper-three-obstruction-reversed}
\end{equation}
If $z_1\in\widetilde E_-^{\mathrm{large}}(r)$, the same contradiction uses
\eqref{eq:upper-three-obstruction-reversed}.
Together with the preceding case $z_1\in E_+^{\mathrm{large}}(r)$, this proves
\eqref{eq:upper-reflected-disjoint-tail}.

Put
\[
 a=\inf\{z:\ z\in E_+^{\mathrm{large}}(r)\cup
\widetilde E_-^{\mathrm{large}}(r)\}.
\]
Since $z_1$ can be chosen arbitrarily close to $a$,
\eqref{eq:upper-reflected-disjoint-tail} gives, after letting $z_1\downarrow a$,
\[
 \bigl(E_+^{\mathrm{large}}(r)\cap(a+B_q,u_q]\bigr)
 \cap
 \bigl(\widetilde E_-^{\mathrm{large}}(r)\cap(a+B_q,u_q]\bigr)=\emptyset .
\]
By \eqref{eq:upper-large-radii-def} and \eqref{eq:upper-threshold-class}, both
sets are contained in
\[
 I_q=[\max\{L_q,r u_q-K_q\},u_q].
\]
Indeed, for any signed point $s\in E_+^{\mathrm{large}}(r)\cup
E_-^{\mathrm{large}}(r)$, and for its associated $u\in\mathcal C_s$,
\[
 |s|\in[L_q,u_q],\qquad |s|\geq ru-K_q\geq r u_q-K_q.
\]
Recall the definition of $B_q$ in \eqref{eq:upper-buffer}.
Set
\[
 J_0=I_q\cap[a,a+B_q],\qquad
 J_1=I_q\cap(a+B_q,u_q].
\]
Since both reflected sets are contained in $I_q$ and have no points below
$a$, their total length decomposes as
\[
\begin{aligned}
 |E_+^{\mathrm{large}}(r)|+|\widetilde E_-^{\mathrm{large}}(r)|
&=|E_+^{\mathrm{large}}(r)\cap J_0|
  +|\widetilde E_-^{\mathrm{large}}(r)\cap J_0|\\
&\quad
  +|E_+^{\mathrm{large}}(r)\cap J_1|
  +|\widetilde E_-^{\mathrm{large}}(r)\cap J_1|.
\end{aligned}
\]
On the set $J_0$, each set contributes at most
$|J_0|$ with
\[
 |J_0|\leq\min\{B_q,|I_q|\}.
\]
On $J_1$, the disjointness above gives
\[
\bigl(E_+^{\mathrm{large}}(r)\cap J_1\bigr)
\cap
\bigl(\widetilde E_-^{\mathrm{large}}(r)\cap J_1\bigr)=\emptyset ,
\]
so
\[
\begin{aligned}
 |E_+^{\mathrm{large}}(r)\cap J_1|
 +|\widetilde E_-^{\mathrm{large}}(r)\cap J_1|
 &=
 |\bigl(E_+^{\mathrm{large}}(r)\cup\widetilde E_-^{\mathrm{large}}(r)\bigr)
 \cap J_1|\leq |J_1|.
\end{aligned}
\]
Since $I_q=[\max\{L_q,r u_q-K_q\},u_q]$, the containment of
$E_+^{\mathrm{large}}(r)$ and $\widetilde E_-^{\mathrm{large}}(r)$ in
$I_q$ gives $a\geq\max\{L_q,r u_q-K_q\}$.  Hence
\[
\begin{aligned}
 |J_1|
 &=|I_q\cap(a+B_q,u_q]|\\
 &\leq (u_q-a-B_q)_+\\
 &\leq (u_q-\max\{L_q,r u_q-K_q\}-B_q)_+
 =(|I_q|-B_q)_+ .
\end{aligned}
\]
Therefore
\[
\begin{aligned}
 |E_+^{\mathrm{large}}(r)|+|E_-^{\mathrm{large}}(r)|
 &=|E_+^{\mathrm{large}}(r)|+|\widetilde E_-^{\mathrm{large}}(r)|\\
 &\leq 2|J_0|+
 |\bigl(E_+^{\mathrm{large}}(r)\cup\widetilde E_-^{\mathrm{large}}(r)\bigr)
 \cap J_1|\\
 &\leq 2\min\{B_q,|I_q|\}+(|I_q|-B_q)_+\\
 &\leq B_q+|I_q|\\
 &= B_q+u_q-\max\{L_q,r u_q-K_q\} \\
 &\leq (1-r)u_q+K_q+B_q .
\end{aligned}
\]
After decreasing $q_0$ if necessary, $1-r^2/u_q\geq1/2$.  Thus
\eqref{eq:upper-buffer} gives
\[
 B_q
 \leq 2(2K_q+2r^2)+4\sqrt{u_q}
 \leq C\{K_q+\sqrt{u_q}+1\}.
\]
Consequently, since $K_q=\sqrt{8\log u_q}$,
\[
 K_q+B_q\leq C\{K_q+\sqrt{u_q}+1\}\leq C\sqrt{u_q},
\]
which proves \eqref{eq:upper-large-coefficient-range}.
\end{proof}

The next lemma bounds $|E_\pm^{\mathrm{small}}(r)|$.

\begin{lemma}\label{lem:upper-small-coefficients}
With $E_\pm^{\mathrm{small}}(r)$ defined in
\eqref{eq:upper-small-radii-def}, there are universal constants $C<\infty$
and $q_0>0$ such that, whenever $q\leq q_0$,
\begin{equation}
 |E_+^{\mathrm{small}}(r)|+|E_-^{\mathrm{small}}(r)|
 \leq C(K_q+1)\leq C\sqrt{u_q}                         \label{eq:upper-small-coefficient-range}
\end{equation}
uniformly over all positive measures $\Pi$ on $\mathbb R\times[-1,1]$
such that $\Pi(\mathbb R\times[-1,1])\leq1$, and over all families
$(\mathcal C_z)$ satisfying
\eqref{eq:upper-threshold-class}.
\end{lemma}

\begin{proof}
Put
\[
 b=\frac r{u_q},\qquad
 \eta=\Pi\{(\theta,\alpha):|\alpha|<b\}.
\]
For each $z$ and $u$, put
\begin{equation}
 Q_z(u)=\frac{G_z^{\mathrm{small}}(u)}{p(u)}
 =\frac1{p(u)}\int_{\{|\alpha|<b\}}
 \mathbb P\{|\theta+\alpha z+\sqrt{1-\alpha^2}\,\varepsilon|\geq u\}
 \,\dd\Pi(\theta,\alpha),
 \label{eq:upper-Qz-def}
\end{equation}
where $\varepsilon\sim N(0,1)$.  If
$z\in E_+^{\mathrm{small}}(r)\cup E_-^{\mathrm{small}}(r)$ and
$u\in\mathcal C_z$ is one of the values appearing in
\eqref{eq:upper-small-radii-def}, then
\begin{equation}
 Q_z(u)\geq\frac1{4q}.                              \label{eq:upper-small-coefficient-mass}
\end{equation}

If $E_+^{\mathrm{small}}(r)\cup E_-^{\mathrm{small}}(r)$ is empty, there is
nothing to prove.  Otherwise choose
$z_1\in E_+^{\mathrm{small}}(r)\cup E_-^{\mathrm{small}}(r)$ so that
$|z_1|$ is arbitrarily close to
\[
\inf\{|z|:\ z\in E_+^{\mathrm{small}}(r)\cup E_-^{\mathrm{small}}(r)\}.
\]
Choose $v_1\in\mathcal C_{z_1}$
as in \eqref{eq:upper-small-radii-def}.  Since
$Q_{z_1}(v_1)\geq1/(4q)>0$, necessarily $\eta>0$; if
$\eta=0$, then $Q_z\equiv0$ for every $z$.  By
\eqref{eq:upper-small-radii-def}, $|z_1|\leq u_q$ and
$\big||z_1|-rv_1\big|\leq K_q$.  Therefore
\begin{equation}
 rv_1\leq |z_1|+K_q\leq u_q+K_q,
 \qquad\text{and hence}\qquad
 bv_1\leq1+K_q/u_q.                                  \label{eq:upper-ru-upper}
\end{equation}
Define
\[
 H_1=\log\left\{\frac{Q_{z_1}(v_1)\sqrt{1-b^2}}{2\eta}\right\}.
\]
Using \eqref{eq:upper-small-coefficient-mass} and $\eta\leq1$,
\[
 H_1
 \geq\log\left\{\frac{\sqrt{1-b^2}}{8q}\right\}.
\]
For $q_0$ small enough, $u_q\geq2$.  Since $0\leq r\leq1$, this gives
\[
 b=\frac r{u_q}\leq\frac12,
 \qquad
 \sqrt{1-b^2}\geq\frac12.
\]
Therefore, using $q=p(u_q)\leq e^{-u_q^2/2}$ from
\eqref{eq:mills-two-sided-exponential-tail},
\begin{equation}
 H_1\geq\log\frac1{16q}
 \geq\frac{u_q^2}{2}-\log16
 \geq\frac{u_q^2}{3}\geq4,
 \label{eq:upper-H1-lower}
\end{equation}
after decreasing $q_0$, if necessary.  Moreover, by
\eqref{eq:upper-ru-upper},
\[
 \frac{9b^2v_1^2}{1-b^2}
 \leq
 \frac{9(1+K_q/u_q)^2}{1-u_q^{-2}}
 \leq C.
\]
Since $H_1\geq u_q^2/3$, decreasing $q_0$ once more makes this last
constant at most $H_1$.  Hence
\begin{equation}
 H_1=\log\left\{\frac{Q_{z_1}(v_1)\sqrt{1-b^2}}{2\eta}\right\}\geq4,
 \qquad
 \frac{9b^2v_1^2}{1-b^2}\leq H_1.                  \label{eq:upper-ratio-growth-conditions}
\end{equation}

We will show the following result at the end of the proof:
\begin{equation}
 \exists\,v\geq v_1:\qquad
 Q_{z_1}(v)\geq \frac{3e^{10u_q}}q,\qquad
 0\leq v-v_1\leq C\frac{v_1}{u_q}.
 \label{eq:upper-Q-large}
\end{equation}

\noindent Now we prove \eqref{eq:upper-small-coefficient-range} under
\eqref{eq:upper-Q-large}.  For any
$|z'|\leq u_q$, put
\begin{equation}
 \Delta=b|z'-z_1|\leq2r.
 \label{eq:upper-delta-def}
\end{equation}
Fix a component $(\theta,\alpha)$ with $|\alpha|<b$, and use the same
residual $\varepsilon$ to form the two conditional statistics at $z_1$ and
at $z'$:
\[
 Y_{z_1}=\theta+\alpha z_1+\sqrt{1-\alpha^2}\,\varepsilon,\qquad
 Y_{z'}=\theta+\alpha z'+\sqrt{1-\alpha^2}\,\varepsilon .
\]
Then
\[
 |Y_{z'}-Y_{z_1}|=|\alpha|\,|z'-z_1|\leq b|z'-z_1|=\Delta,
\]
and hence
\[
 \{|Y_{z_1}|\geq v\}\subseteq \{|Y_{z'}|\geq v-\Delta\}.
\]
For this fixed $(\theta,\alpha)$, the inclusion gives
\[
 \mathbb P_\varepsilon\{|Y_{z'}|\geq v-\Delta\}
 \geq
 \mathbb P_\varepsilon\{|Y_{z_1}|\geq v\}.
\]
Integrating the last inequality with respect to $\Pi$ over the
small-coefficient set $\{|\alpha|<b\}$ gives, by the definition of
$G_z^{\mathrm{small}}$,
\[
\begin{aligned}
 G_{z'}^{\mathrm{small}}(v-\Delta)
 &=
 \int_{\{|\alpha|<b\}}
 \mathbb P_\varepsilon\{|Y_{z'}|\geq v-\Delta\}\,\dd\Pi(\theta,\alpha)
 \\
 &\geq
 \int_{\{|\alpha|<b\}}
 \mathbb P_\varepsilon\{|Y_{z_1}|\geq v\}\,\dd\Pi(\theta,\alpha)
 =
 G_{z_1}^{\mathrm{small}}(v).
\end{aligned}
\]
Dividing by $p(v-\Delta)$ and using
$G_{z_1}^{\mathrm{small}}(v)=Q_{z_1}(v)p(v)$ gives
\begin{equation}
 Q_{z'}(v-\Delta)
 \geq Q_{z_1}(v)\frac{p(v)}{p(v-\Delta)}.            \label{eq:upper-factor-coupling}
\end{equation}
Since $p=2\Phib$, $p'(t)/p(t)=-\lambda(t)$, where
$\lambda(t)=\phi(t)/\Phib(t)$.  Also, by
\eqref{eq:upper-delta-def}, $r\leq3/5$, and $v\geq v_1\geq u_q$,
\[
 v-\Delta\geq v_1-\Delta\geq u_q-2r\geq u_q-\frac65>0
\]
after decreasing $q_0$, if necessary.  Thus the interval
$[v-\Delta,v]$ lies in $(0,\infty)$, and
\[
 \frac{p(v)}{p(v-\Delta)}
 =
 \exp\left\{-\int_{v-\Delta}^v \lambda(t)\,\dd t\right\}.
\]
By \eqref{eq:mills-hazard} of Lemma~\ref{lem:mills}, for
$v-\Delta\leq t\leq v$,
\[
 \lambda(t)\leq t+\frac1t\leq v+\frac1{v-\Delta}.
\]
Therefore
\begin{equation}
 \frac{p(v)}{p(v-\Delta)}
 \geq\exp\left[-\Delta\left\{v+\frac1{v-\Delta}\right\}\right]. \label{eq:upper-tail-ratio-lower}
\end{equation}
By \eqref{eq:upper-small-radii-def},
\[
 rv_1\leq |z_1|+K_q\leq u_q+K_q.
\]
Also, \eqref{eq:upper-Q-large} and \eqref{eq:upper-ru-upper} give
\[
 r(v-v_1)\leq C\frac{rv_1}{u_q}
 \leq C\left(1+\frac{K_q}{u_q}\right)=O(1).
\]
Therefore
\[
 rv=rv_1+r(v-v_1)\leq u_q+K_q+O(1).
\]
Since $\Delta\leq2r$ and $v-\Delta\geq u_q-2r$, the exponent in
\eqref{eq:upper-tail-ratio-lower} satisfies
\[
 \Delta\left\{v+\frac1{v-\Delta}\right\}
 \leq 2rv+\frac{2r}{u_q-2r}
 \leq 2u_q+2K_q+O(1)\leq3u_q,
\]
after decreasing $q_0$, if necessary, where the last step uses
$K_q=o(u_q)$ from \eqref{eq:upper-KLq-def}.  Combining
\eqref{eq:upper-factor-coupling}, \eqref{eq:upper-Q-large}, and
\eqref{eq:upper-tail-ratio-lower},
\[
\begin{aligned}
 Q_{z'}(v-\Delta)
 &\geq Q_{z_1}(v)\frac{p(v)}{p(v-\Delta)}\\
 &\geq \frac{3e^{10u_q}}q
 \exp\left[-\Delta\left\{v+\frac1{v-\Delta}\right\}\right]\\
 &\geq \frac{3e^{10u_q}}q e^{-3u_q}
 =\frac{3e^{7u_q}}q>\frac3q .
\end{aligned}
\]
Hence
\begin{equation}
 Q_{z'}(v-\Delta)>\frac3q.                          \label{eq:upper-global-obstruction}
\end{equation}
If $v-\Delta<u_q$, then
\begin{equation}
 v_1\leq v<u_q+\Delta\leq u_q+2r .
 \label{eq:upper-v-below-uq-plus-delta}
\end{equation}
Since $r\leq3/5$, \eqref{eq:mills-local-ratio} in
Lemma~\ref{lem:mills} gives, uniformly over $0\leq h\leq2r$,
\[
\begin{aligned}
\log\frac{p(u_q+h)}{p(u_q)}
&=-u_qh-\frac{h^2}{2}
  +\log\frac{u_q}{u_q+h}+o(1)\\
&\geq -2ru_q-2r^2-\log\left(1+\frac{2r}{u_q}\right)+o(1)\\
&\geq -\frac65u_q-\frac{18}{25}
      -\log\left(1+\frac{6}{5u_q}\right)+o(1)\\
&\geq -2u_q,
\end{aligned}
\]
after decreasing $q_0$, if necessary.  The last step uses $u_q\to\infty$
as $q\downarrow0$, and the $o(1)$ term is uniform over
$0\leq h\leq2r\leq6/5$.  Taking $h=2r$ and using
\eqref{eq:upper-v-below-uq-plus-delta} together with the monotonicity of $p$,
\[
 p(v)\geq p(u_q+2r)\geq p(u_q)e^{-2u_q}=q e^{-2u_q}.
\]
Because $G_{z_1}^{\mathrm{small}}(v)\leq\Pi(\mathbb R\times[-1,1])\leq1$,
\[
 Q_{z_1}(v)
 =\frac{G_{z_1}^{\mathrm{small}}(v)}{p(v)}
 \leq \frac1{p(v)}
 \leq \frac{e^{2u_q}}q.
\]
This contradicts \eqref{eq:upper-Q-large}, which gives
$Q_{z_1}(v)\geq3e^{10u_q}/q$.  Hence
\begin{equation}
 v-\Delta\geq u_q.                                  \label{eq:upper-v-minus-delta-lower}
\end{equation}

We next show
\begin{equation}
 \mathcal C_{z'}\cap(v-\Delta,\infty)=\emptyset.
 \label{eq:upper-cutoff-exclusion}
\end{equation}
Suppose, to the contrary, that $u'\in\mathcal C_{z'}$ and
$u'>v-\Delta$.  Since \eqref{eq:upper-v-minus-delta-lower} holds, the defining upper-tail
condition in \eqref{eq:upper-threshold-class}, applied with
$w=v-\Delta<u'$, gives
\[
 G_{z'}(v-\Delta)\leq 3\frac{p(v-\Delta)}q.
\]
On the other hand, \eqref{eq:upper-global-obstruction} gives
\[
 G_{z'}^{\mathrm{small}}(v-\Delta)
 =p(v-\Delta)Q_{z'}(v-\Delta)
 >3\frac{p(v-\Delta)}q,
\]
and $G_{z'}(v-\Delta)\geq G_{z'}^{\mathrm{small}}(v-\Delta)$.  This is a
contradiction.  Therefore every $u'\in\mathcal C_{z'}$ satisfies
\begin{equation}
 u'\leq v-\Delta\leq v .
\label{eq:upper-cutoff-below-v}
\end{equation}
Therefore, \eqref{eq:upper-cutoff-below-v} is proved.

Now let $z'\in E_+^{\mathrm{small}}(r)\cup E_-^{\mathrm{small}}(r)$.  By
\eqref{eq:upper-small-radii-def}, there exists $u'\in\mathcal C_{z'}$ such
that
\[
 \big||z'|-ru'\big|\leq K_q.
\]
Together with \eqref{eq:upper-cutoff-below-v}, this gives
\begin{equation}
 |z'|\leq ru'+K_q
 \leq rv+K_q
 \leq |z_1|+2K_q+C.                                \label{eq:upper-crossing-radius}
\end{equation}
Here the last inequality follows from
\[
 rv=rv_1+r(v-v_1)\leq |z_1|+K_q+C,
\]
using $\big||z_1|-rv_1\big|\leq K_q$ from
\eqref{eq:upper-small-radii-def} and the bound
$r(v-v_1)\leq C$ proved above.

Let
\[
 a=\inf\{|z|:\ z\in E_+^{\mathrm{small}}(r)\cup E_-^{\mathrm{small}}(r)\}.
\]
The point $z_1$ was chosen with $|z_1|$ arbitrarily close to $a$.
Thus \eqref{eq:upper-crossing-radius} shows that
\[
 \{|z|:\ z\in E_+^{\mathrm{small}}(r)\cup E_-^{\mathrm{small}}(r)\}
 \subseteq [a,a+2K_q+C].
\]
By \eqref{eq:upper-small-radii-def},
\[
 E_+^{\mathrm{small}}(r)\subseteq [L_q,u_q],
 \qquad
 E_-^{\mathrm{small}}(r)\subseteq [-u_q,-L_q].
\]
On $E_+^{\mathrm{small}}(r)$, the map $z\mapsto |z|$ is the identity.
On $E_-^{\mathrm{small}}(r)$, it is the reflection $z\mapsto -z$.  Thus,
\[
 |E_\pm^{\mathrm{small}}(r)|
 =
 \big|\{|z|:\ z\in E_\pm^{\mathrm{small}}(r)\}\big|.
\]
As a result,
\[
 |E_+^{\mathrm{small}}(r)|+|E_-^{\mathrm{small}}(r)|
 \leq 2(2K_q+C)\leq C(K_q+1).
\]
This proves \eqref{eq:upper-small-coefficient-range}.

~\\
\noindent \textbf{Proof of \eqref{eq:upper-Q-large}.}
Put
\[
 R=12e^{10u_q},\qquad
 H_1=\log\left\{\frac{Q_{z_1}(v_1)\sqrt{1-b^2}}{2\eta}\right\}.
\]
The choice of $v_1$ in \eqref{eq:upper-small-radii-def}, together with
\eqref{eq:upper-small-coefficient-mass}, gives
\[
 Q_{z_1}(v_1)\geq\frac1{4q}.
\]
Also, \eqref{eq:upper-H1-lower} gives $H_1\geq u_q^2/3$.  Thus it is
enough to find $v\geq v_1$ such that
\[
 Q_{z_1}(v)\geq RQ_{z_1}(v_1),\qquad
 v-v_1\leq C\frac{v_1}{u_q}.
\]

For $z\in\mathbb R$ and $x\geq0$, define
\[
 \kappa_z(x)=
 \int_{\{|\alpha|<b\}}
 \frac{\phi\left(\frac{x-(\theta+\alpha z)}{\sqrt{1-\alpha^2}}\right)+
 \phi\left(\frac{x+(\theta+\alpha z)}{\sqrt{1-\alpha^2}}\right)}
 {2\phi(x)\sqrt{1-\alpha^2}}\,\dd\Pi(\theta,\alpha).
\]
By definition of $Q_z$ in \eqref{eq:upper-Qz-def},
\[
 Q_{z_1}(u)
 =
 \frac{\int_u^\infty 2\phi(x)\kappa_{z_1}(x)\,\dd x}
      {\int_u^\infty 2\phi(x)\,\dd x}.
\]
Using straightforward algebra,
\[
 Q_{z_1}'(x)=\lambda(x)\{Q_{z_1}(x)-\kappa_{z_1}(x)\},
 \qquad \lambda(x)=\phi(x)/\Phib(x).
\]
Define
\begin{equation}
 \tilde v=
 \inf\left\{x\geq v_1:\ 
 Q_{z_1}(x)\geq RQ_{z_1}(v_1)
 \ \text{or}\ 
 \kappa_{z_1}(x)\geq \frac12 Q_{z_1}(x)\right\}.
 \label{eq:upper-tilde-v-def}
\end{equation}
Before $\tilde v$, we have $\kappa_{z_1}(x)<Q_{z_1}(x)/2$, and therefore
\[
 Q_{z_1}'(x)
 \geq \frac12\lambda(x)Q_{z_1}(x)
 \geq \frac{x}{2}Q_{z_1}(x)
 \geq \frac{v_1}{2}Q_{z_1}(x).
\]
Thus
\[
 \bigl(\log Q_{z_1}\bigr)'(x)
 =\frac{Q_{z_1}'(x)}{Q_{z_1}(x)}
 \geq \frac{v_1}{2}.
\]
Let $T=v_1+2\log R/v_1$. If $\tilde v>T$, then
$(\log Q_{z_1})'(x)\geq v_1/2$ on $[v_1,T]$, so
\[
 \log\frac{Q_{z_1}(T)}{Q_{z_1}(v_1)}
 =\int_{v_1}^{T}(\log Q_{z_1})'(x)\,\dd x
 \geq\frac{v_1}{2}(T-v_1)=\log R.
\]
Thus $Q_{z_1}(T)\geq RQ_{z_1}(v_1)$, contradicting $T<\tilde v$
in \eqref{eq:upper-tilde-v-def}. Hence
\begin{equation}
 \tilde v\leq v_1+\frac{2\log R}{v_1}.              \label{eq:upper-tilde-v-upper}
\end{equation}
Continuity of $Q_{z_1}$ and $\kappa_{z_1}$ ensures that at least one
condition in \eqref{eq:upper-tilde-v-def} holds at $\tilde v$.

\indent Now we consider two cases depending on which inequality is binding in
\eqref{eq:upper-tilde-v-def}.
If
$Q_{z_1}(\tilde v)\geq RQ_{z_1}(v_1)$, take $v=\tilde v$.  Since
$v_1\geq u_q$ by \eqref{eq:upper-threshold-class},
\[
 \frac{\log R}{v_1}
 \leq C\frac{u_q}{v_1}
 \leq C\frac{v_1}{u_q},
\]
Together with \eqref{eq:upper-tilde-v-upper}, this gives
\begin{equation}
 \tilde v\leq v_1+C\frac{u_q}{v_1}
 \leq v_1\left(1+\frac{C}{u_q}\right).
 \label{eq:upper-tilde-v-relative}
\end{equation}
Then \eqref{eq:upper-Q-large} holds due to
\eqref{eq:upper-tilde-v-relative}.

It remains to consider the case
$\kappa_{z_1}(\tilde v)\geq Q_{z_1}(\tilde v)/2$.  Since $Q_{z_1}$ is
increasing on $[v_1,\tilde v]$, this gives
$\kappa_{z_1}(\tilde v)\geq Q_{z_1}(v_1)/2$.  We will show that
\eqref{eq:upper-Q-large} is satisfied with
\begin{equation}
 v=\tilde v+72\frac{\tilde v}{u_q}.
 \label{eq:upper-v-explicit-choice}
\end{equation}

For $z\in\mathbb R$ and $x\geq0$, define
\begin{equation}
 \widetilde\kappa_z(x)
 =\exp\left\{\frac{b^2x^2}{2(1-b^2)}\right\}\kappa_z(x).
 \label{eq:upper-tilde-kappa-def}
\end{equation}
For a fixed pair $(\theta,\alpha)$ with $|\alpha|<b$, the corresponding
integrand in $\kappa_{z_1}(x)$ is
\[
\begin{aligned}
&\frac{1}{2\phi(x)\sqrt{1-\alpha^2}}
\left\{\phi\left(\frac{x-(\theta+\alpha z_1)}{\sqrt{1-\alpha^2}}\right)
      +\phi\left(\frac{x+(\theta+\alpha z_1)}{\sqrt{1-\alpha^2}}\right)\right\}\\
&\quad
=\frac1{2\sqrt{1-\alpha^2}}\left[
\exp\left\{-\frac{\{x-(\theta+\alpha z_1)\}^2}{2(1-\alpha^2)}
             +\frac{x^2}{2}\right\}\right.\\
&\hspace{31mm}\left.
+\exp\left\{-\frac{\{x+(\theta+\alpha z_1)\}^2}{2(1-\alpha^2)}
             +\frac{x^2}{2}\right\}
\right]                                             \\
&\quad
=\frac1{\sqrt{1-\alpha^2}}
\exp\left\{-\frac{\alpha^2x^2}{2(1-\alpha^2)}
            -\frac{(\theta+\alpha z_1)^2}{2(1-\alpha^2)}\right\}
\cosh\left(\frac{(\theta+\alpha z_1)x}{1-\alpha^2}\right).
\end{aligned}
\]
Multiplying by the factor in \eqref{eq:upper-tilde-kappa-def}, the
corresponding integrand in $\widetilde\kappa_{z_1}(x)$ is denoted by
$f_{\theta,\alpha}(x)$, where
\begin{equation}
\begin{aligned}
f_{\theta,\alpha}(x)
&=
\frac1{\sqrt{1-\alpha^2}}
 \exp\left\{-\frac{(\theta+\alpha z_1)^2}{2(1-\alpha^2)}\right\}
 \exp\left\{
 \left(\frac{b^2}{2(1-b^2)}
       -\frac{\alpha^2}{2(1-\alpha^2)}\right)x^2
 \right\} \\
&\quad\times
 \cosh\left(\frac{(\theta+\alpha z_1)x}{1-\alpha^2}\right).
\end{aligned}
\label{eq:upper-tilde-kappa-integrand}
\end{equation}
The first two factors are positive and independent of $x$.  The coefficient of
$x^2$ in the second term is
\[
 a=\frac{b^2}{2(1-b^2)}-\frac{\alpha^2}{2(1-\alpha^2)}
 =\frac{b^2-\alpha^2}{2(1-b^2)(1-\alpha^2)}\geq0,
\]
because $|\alpha|<b$.  Thus, taking logarithms in
\eqref{eq:upper-tilde-kappa-integrand} and dropping the additive constant
independent of $x$, the logarithm of the integrand has the form
\[
 a x^2+\log\cosh(cx),
 \qquad c=\frac{\theta+\alpha z_1}{1-\alpha^2}.
\]
On
$[0,\infty)$,
\[
\begin{aligned}
 \frac{d}{dx}\{a x^2+\log\cosh(cx)\}
 &=2ax+c\tanh(cx)\\
 &=2ax+|c|\tanh(|c|x)\geq0,\\
\frac{d^2}{dx^2}\{a x^2+\log\cosh(cx)\}
&=2a+c^2\sech^2(cx)\geq0.
\end{aligned}
\]
Thus $f_{\theta,\alpha}$ is log-convex.

To pass from each integrand to the mixture, fix $0\leq t\leq1$ and
$x,y\geq0$.  Log-convexity gives
\[
 f_{\theta,\alpha}\{t x+(1-t)y\}
 \leq f_{\theta,\alpha}(x)^t f_{\theta,\alpha}(y)^{1-t}.
\]
After integrating over $\{|\alpha|<b\}$, H\"older's inequality gives
\[
\begin{aligned}
 \widetilde\kappa_{z_1}\{t x+(1-t)y\}
 &\leq
 \int_{\{|\alpha|<b\}}
 f_{\theta,\alpha}(x)^t f_{\theta,\alpha}(y)^{1-t}
 \,\dd\Pi(\theta,\alpha)\\
 &\leq
 \left\{\int_{\{|\alpha|<b\}} f_{\theta,\alpha}(x)
 \,\dd\Pi(\theta,\alpha)\right\}^t
 \left\{\int_{\{|\alpha|<b\}} f_{\theta,\alpha}(y)
 \,\dd\Pi(\theta,\alpha)\right\}^{1-t}\\
 &=
 \widetilde\kappa_{z_1}(x)^t
 \widetilde\kappa_{z_1}(y)^{1-t}.
\end{aligned}
\]
Hence $\widetilde\kappa_{z_1}$ is log-convex.  Equivalently,
\begin{equation}
 x\mapsto \log\kappa_{z_1}(x)+\frac{b^2x^2}{2(1-b^2)}\quad\text{is convex on }[0,\infty).              \label{eq:upper-kappa-adjusted-convex}
\end{equation}
Fix $M<\infty$ and restrict $0\leq x\leq M$.  From
\eqref{eq:upper-tilde-kappa-integrand}, for a constant $C_M<\infty$,
\begin{equation}
0\leq f_{\theta,\alpha}(x)\leq C_M\exp\left\{-\frac{(\theta+\alpha z_1)^2}{2(1-\alpha^2)}+\frac{M|\theta+\alpha z_1|}{1-\alpha^2}\right\}\leq C_M.
\label{eq:upper-f-compact-bound}
\end{equation}
The last inequality follows from
\[
 -\frac{(\theta+\alpha z_1)^2}{2(1-\alpha^2)}
 +\frac{M|\theta+\alpha z_1|}{1-\alpha^2}
 =
 -\frac{\{|\theta+\alpha z_1|-M\}^2}{2(1-\alpha^2)}
 +\frac{M^2}{2(1-\alpha^2)}
 \leq \frac{M^2}{2(1-b^2)}.
\]
Differentiating the integrand gives
\[
\frac{\dd}{\dd x}f_{\theta,\alpha}(x)=f_{\theta,\alpha}(x)\left[2\left\{\frac{b^2}{2(1-b^2)}-\frac{\alpha^2}{2(1-\alpha^2)}\right\}x+\frac{\theta+\alpha z_1}{1-\alpha^2}\tanh\left\{\frac{(\theta+\alpha z_1)x}{1-\alpha^2}\right\}\right].
\]
The bracketed factor is controlled by $1+|\theta+\alpha z_1|$.  Indeed,
on $0\leq x\leq M$,
\[
\left|2\left\{\frac{b^2}{2(1-b^2)}-\frac{\alpha^2}{2(1-\alpha^2)}\right\}x+\frac{\theta+\alpha z_1}{1-\alpha^2}\tanh\left\{\frac{(\theta+\alpha z_1)x}{1-\alpha^2}\right\}\right|
\leq C_M\{1+|\theta+\alpha z_1|\},
\]
because $|\tanh|\leq1$, $1-\alpha^2\geq1-b^2$, and $|\alpha|<b$.
Combining this with \eqref{eq:upper-f-compact-bound} gives
\[
\left|\frac{\dd}{\dd x}f_{\theta,\alpha}(x)\right|
\leq C_M\{1+|\theta+\alpha z_1|\}
\exp\left\{-\frac{(\theta+\alpha z_1)^2}{2(1-\alpha^2)}
          +\frac{M|\theta+\alpha z_1|}{1-\alpha^2}\right\}.
\]
Moreover,
\[
\begin{aligned}
&\{1+|\theta+\alpha z_1|\}
\exp\left\{-\frac{(\theta+\alpha z_1)^2}{2(1-\alpha^2)}
          +\frac{M|\theta+\alpha z_1|}{1-\alpha^2}\right\}\\
&\qquad =
\{1+|\theta+\alpha z_1|\}
\exp\left\{-\frac{\{|\theta+\alpha z_1|-M\}^2}{2(1-\alpha^2)}
          +\frac{M^2}{2(1-\alpha^2)}\right\}
\leq C_M,
\end{aligned}
\]
uniformly in $(\theta,\alpha)$ with $|\alpha|<b$, because
$1-\alpha^2\in[1-b^2,1]$.  Therefore
\[
 \sup_{0\leq x\leq M}
 \left|\frac{\dd}{\dd x}f_{\theta,\alpha}(x)\right|
 \leq C_M.
\]
Thus, on each compact interval, both $f_{\theta,\alpha}$ and
$\dd f_{\theta,\alpha}/\dd x$ are dominated by an integrable bound with
respect to the finite measure $\Pi$.  Hence $\widetilde\kappa_{z_1}$, and
therefore $\kappa_{z_1}$, is differentiable on $[0,\infty)$.

When $x=0$,
\[
\begin{aligned}
 \kappa_{z_1}(0)
 &=
 \int_{\{|\alpha|<b\}}
 \frac{\phi\left(\frac{\theta+\alpha z_1}{\sqrt{1-\alpha^2}}\right)}
      {\phi(0)\sqrt{1-\alpha^2}}\,\dd\Pi(\theta,\alpha)\\
 &\leq
 \int_{\{|\alpha|<b\}}\frac1{\sqrt{1-\alpha^2}}\,\dd\Pi(\theta,\alpha)
 \leq \frac{\eta}{\sqrt{1-b^2}}.
\end{aligned}
\]
Together with $\kappa_{z_1}(\tilde v)\geq Q_{z_1}(v_1)/2$, this gives
\[
 \frac{\kappa_{z_1}(\tilde v)}{\kappa_{z_1}(0)}
 \geq
 \frac{Q_{z_1}(v_1)\sqrt{1-b^2}}{2\eta}
 =e^{H_1},
\]
by the definition of $H_1$.  By
\eqref{eq:upper-kappa-adjusted-convex}, the derivative at $\tilde v$ is at
least the secant slope from $0$ to $\tilde v$:
\[
\begin{aligned}
 \frac{d}{dx}\left\{
 \log\kappa_{z_1}(x)+\frac{b^2x^2}{2(1-b^2)}
 \right\}\bigg|_{x=\tilde v}
 &\geq
 \frac{\log\kappa_{z_1}(\tilde v)-\log\kappa_{z_1}(0)
       +b^2\tilde v^2/\{2(1-b^2)\}}{\tilde v}\\
 &\geq \frac{H_1}{\tilde v}.
\end{aligned}
\]
The derivative of the convex function in
\eqref{eq:upper-kappa-adjusted-convex} is nondecreasing.

By \eqref{eq:upper-tilde-v-relative}, \eqref{eq:upper-ru-upper}, and
\eqref{eq:upper-H1-lower}, we have
$\tilde v\leq v_1+C u_q/v_1$, $bv_1\leq1+K_q/u_q$, and
$H_1\geq u_q^2/3$.  Since $v_1\geq u_q$, the first two bounds imply
\[
 b\tilde v
 \leq bv_1\left(1+\frac{C}{u_q}\right)
 \leq\left(1+\frac{K_q}{u_q}\right)\left(1+\frac{C}{u_q}\right)
 \leq 2
\]
after decreasing $q_0$, if necessary.  Also, if
$\tilde v\leq x\leq \tilde v+72\tilde v/u_q+1$, then $x\leq2\tilde v$.
Therefore, using $b\tilde v\leq2$ and $b\leq u_q^{-1}$,
\[
 \frac{b^2x\tilde v}{1-b^2}
 \leq \frac{2b^2\tilde v^2}{1-b^2}
 \leq C.
\]
Because $H_1\geq u_q^2/3$, decreasing $q_0$ once more makes this constant
at most $H_1/2$.  Thus,
\[
 \frac{b^2x\tilde v}{1-b^2}\leq\frac{H_1}{2}
 \qquad
 \left(\tilde v\leq x\leq \tilde v+72\frac{\tilde v}{u_q}+1\right).
\]
Hence, on this interval,
\begin{equation}
(\log\kappa_{z_1})'(x)\geq \frac{H_1}{\tilde v}-\frac{b^2x}{1-b^2}
\geq\frac{H_1}{2\tilde v}.
\label{eq:upper-kappa-log-derivative-lower}
\end{equation}
For $v$ defined in \eqref{eq:upper-v-explicit-choice},
\[
\begin{aligned}
 \log\frac{\kappa_{z_1}(v)}{\kappa_{z_1}(\tilde v)}
 &=
 \int_{\tilde v}^v(\log\kappa_{z_1})'(x)\,\dd x
 \geq \int_{\tilde v}^v\frac{H_1}{2\tilde v}\,\dd x\\
 &=\frac{H_1}{2\tilde v}(v-\tilde v)
 =\frac{H_1}{2\tilde v}\,72\frac{\tilde v}{u_q}
 =\frac{36H_1}{u_q}.
\end{aligned}
\]
Thus,
\begin{equation}
 \kappa_{z_1}(v)
 \geq \kappa_{z_1}(\tilde v)
 \exp\left\{\frac{36H_1}{u_q}\right\}
 \geq \frac{Q_{z_1}(v_1)}{2}
 \exp\left\{\frac{36H_1}{u_q}\right\}.
\label{eq:upper-kappa-v-growth}
\end{equation}
Since $H_1\geq u_q^2/3$, after decreasing $q_0$, if necessary,
\begin{equation}
 \frac{1-e^{-1}}2
 \exp\left\{\frac{36H_1}{u_q}\right\}
 \geq
 \frac{1-e^{-1}}2 e^{12u_q}
 \geq 12e^{10u_q}
 =R.
\label{eq:upper-kappa-growth-constant}
\end{equation}
Also, $v\geq\tilde v\geq v_1\geq u_q$, so $v^{-1}\leq1$ after
decreasing $q_0$, if necessary.  By \eqref{eq:upper-v-explicit-choice},
\[
 [v,v+v^{-1}]
 \subseteq
 \left[\tilde v,\tilde v+72\frac{\tilde v}{u_q}+1\right].
\]
The lower bound \eqref{eq:upper-kappa-log-derivative-lower} therefore gives
\[
 (\log\kappa_{z_1})'(x)\geq0,\qquad v\leq x\leq v+v^{-1},
\]
and hence $\kappa_{z_1}$ is nondecreasing on $[v,v+v^{-1}]$.  By the
definition of $Q_{z_1}$,
\[
\begin{aligned}
 Q_{z_1}(v)
 &\geq
 \frac{\int_v^{v+v^{-1}}2\phi(x)\kappa_{z_1}(x)\,\dd x}
      {2\Phib(v)}\\
 &\geq
 \kappa_{z_1}(v)
 \frac{\Phib(v)-\Phib(v+v^{-1})}{\Phib(v)}\\
\end{aligned}
\]
Moreover, by \eqref{eq:mills-hazard} of Lemma~\ref{lem:mills},
\begin{equation}
 \frac{\Phib(v+v^{-1})}{\Phib(v)}
 =
 \exp\left\{-\int_v^{v+v^{-1}}\lambda(t)\,\dd t\right\}
 \leq
 \exp\left\{-\int_v^{v+v^{-1}}t\,\dd t\right\}
 \leq e^{-1}.
\label{eq:upper-short-tail-ratio}
\end{equation}
Therefore
\[
 Q_{z_1}(v)
 \geq
 (1-e^{-1})\kappa_{z_1}(v)
 \geq RQ_{z_1}(v_1),
\]
where the first inequality uses \eqref{eq:upper-short-tail-ratio}, and the
second uses \eqref{eq:upper-kappa-v-growth} and
\eqref{eq:upper-kappa-growth-constant}.
Finally, by \eqref{eq:upper-v-explicit-choice} and
\eqref{eq:upper-tilde-v-relative},
\[
\begin{aligned}
 v-v_1
 &=(\tilde v-v_1)+72\frac{\tilde v}{u_q}\\
 &\leq C\frac{u_q}{v_1}
   +72\frac{v_1}{u_q}\left(1+\frac{C}{u_q}\right).
\end{aligned}
\]
Since $v_1\geq u_q$ from $v_1\in\mathcal C_{z_1}$ and
\eqref{eq:upper-threshold-class},
\[
 \frac{u_q}{v_1}\leq\frac{v_1}{u_q},
 \qquad
 72\frac{v_1}{u_q}\left(1+\frac{C}{u_q}\right)
 \leq C\frac{v_1}{u_q},
\]
after increasing $C$ and decreasing $q_0$, if necessary.  Therefore
\[
 v-v_1\leq C\frac{v_1}{u_q}.
\]
Thus $v$ defined in \eqref{eq:upper-v-explicit-choice} satisfies
\eqref{eq:upper-Q-large}.
\end{proof}

The following lemma is the deterministic input needed for the good-event
part of the finite BH proof.
\begin{lemma}\label{lem:upper-crossing}
There are universal constants $C<\infty$ and $q_0>0$ with the following property.
Let $\Pi$, $\varepsilon$, and $G_z$ be as in
\eqref{eq:upper-mixture}.  Let $m\geq1$ and $J\geq2$.  Define
\begin{equation}
 t_k=\frac{qk}{m},\qquad u_k=p^{-1}(t_k).            \label{eq:upper-rank-grid}
\end{equation}
For each $z$, let $w_k(z)$, $J<k\leq m$, be nonnegative measurable weights satisfying
\[
 \sum_{k>J}w_k(z)\leq1.
\]
and
\begin{equation}
 G_z(u_k)\geq\frac{k}{2m},\qquad
 G_z(u_j)\leq\frac{2j}{m}\quad(j > k)\,\, \text{if }w_k(z) > 0.              \label{eq:upper-rank-weight-conditions}
\end{equation}
Then, whenever $q\leq q_0$, uniformly over $m,J,\Pi$, the weights
$(w_k(z))$, and $0\leq r\leq1$,
\[
\int_{-u_q}^{u_q}\phi(z)
\sum_{k>J}w_k(z)\rho_r(u_k,z)\dd z
\leq \frac{u_q}{2\sqrt{2\pi}}+C\sqrt{u_q}.
\]
\end{lemma}

\begin{proof}[Proof of Lemma~\ref{lem:upper-crossing}]
For each $z$, put
\[
 \mathcal C_z=\{u_k:\ J<k\leq m,\ w_k(z)>0\}.
\]
If $u_k\in\mathcal C_z$, then the first condition in
\eqref{eq:upper-rank-weight-conditions} gives
\begin{equation}
 G_z(u_k)\geq \frac{k}{2m}=\frac{p(u_k)}{2q}.
 \label{eq:upper-rank-self-bound}
\end{equation}
For $u_q\leq v<u_k$, let $j=\lceil mp(v)/q\rceil$.  Then
$j>k$, $u_j\leq v$, and by monotonicity of $G_z$,
\[
 G_z(v)\leq G_z(u_j)\leq\frac{2j}{m}
 \leq2\left(1+\frac1J\right)\frac{p(v)}q
 \leq3\frac{p(v)}q,
\]
where $J\geq2$.  Thus $\mathcal C_z$ satisfies \eqref{eq:upper-threshold-class}.  Also,
since the weights are nonnegative and sum to at most one,
\[
 \sum_{k>J}w_k(z)\rho_r(u_k,z)
 \leq \sup_{u\in\mathcal C_z}\rho_r(u,z),
\]
with the right side interpreted as zero when $\mathcal C_z$ is empty.
It remains to prove
\begin{equation}
 \int_{-u_q}^{u_q}\phi(z)\sup_{u\in\mathcal C_z}\rho_r(u,z)\dd z
 \leq \frac{u_q}{2\sqrt{2\pi}}+C\sqrt{u_q}.     \label{eq:upper-crossing-target}
\end{equation}

If $r=1$, then Lemma~\ref{lem:upper-null-ratio} gives
\[
 \int_{-u_q}^{u_q}\phi(z)\sup_{u\in\mathcal C_z}\rho_1(u,z)\dd z
 \leq
 \int_{-u_q}^{u_q}\phi(z)\sup_{u\geq u_q}\rho_1(u,z)\dd z
 =0.
\]
This proves the desired bound in this case.

Next suppose $3/5\leq r<1$.  Since
$(1-r)/\sqrt{1-r^2}\leq1/2$, Lemma~\ref{lem:upper-null-ratio} gives
\begin{equation}
 \int_{-u_q}^{u_q}\phi(z)\sup_{u\geq u_q}\rho_r(u,z)\dd z
 \leq\frac12+\frac{u_q}{2\sqrt{2\pi}}
 \leq\frac{u_q}{2\sqrt{2\pi}}+C\sqrt{u_q}.       \label{eq:upper-crossing-large-r}
\end{equation}
This upper-bounds the smaller supremum over $\mathcal C_z$
and proves the desired bound in this range.

It remains to consider $0\leq r<3/5$.  For $r=0$, the integral is at
most $1-q$, which is bounded by the displayed right side after increasing
$C$.  Hence assume $0<r<3/5$.  Use $L_q,K_q$ defined in
\eqref{eq:upper-KLq-def}.
When $q$ is sufficiently small, \eqref{eq:upper-KLq-def} implies $L_q<u_q$.
We will consider the following decomposition of \eqref{eq:upper-crossing-target}:
\begin{equation}
\begin{aligned}
&\int_{-u_q}^{u_q}\phi(z)\sup_{u\in\mathcal C_z}\rho_r(u,z)\dd z\\
&\leq
\int_{|z|<L_q}\phi(z)
     \sup_{u\in\mathcal C_z}\rho_r(u,z)\dd z\\
&\quad+
\int_{L_q\leq |z|\leq u_q}\phi(z)
 \sup_{\substack{u\in\mathcal C_z:\\ \big||z|-ru\big|>K_q}}
 \rho_r(u,z)\dd z\\
&\quad+
\int_{L_q\leq |z|\leq u_q}\phi(z)
 \sup_{\substack{u\in\mathcal C_z:\\ \big||z|-ru\big|\leq K_q}}
 \rho_r(u,z)\dd z .
\end{aligned}
\label{eq:upper-crossing-decomposition}
\end{equation}
All empty suprema in \eqref{eq:upper-crossing-decomposition} are interpreted as zero.

\indent By construction, every $u\in\mathcal C_z$ satisfies $u\geq u_q$.  When
$|z|< L_q$, Lemma~\ref{lem:upper-localization}
gives
\[
 \phi(z)\rho_r(u,z)
 \leq C\left[
 e^{-(z-ru)^2/\{2(1-r^2)\}}
 +e^{-(z+ru)^2/\{2(1-r^2)\}}\right]
 \leq C
\]
uniformly over $u\in\mathcal C_z$.  Taking the supremum over
$\mathcal C_z$, with value zero when $\mathcal C_z$ is empty, and then
integrating over a set of length at most $2L_q$ gives
\begin{equation}
 \int_{|z|<L_q}\phi(z)
     \sup_{u\in\mathcal C_z}\rho_r(u,z)\dd z
 \leq CL_q.                                      \label{eq:upper-small-z-integral}
\end{equation}
On $L_q\leq|z|\leq u_q$, if $u\in\mathcal C_z$ and
$\big||z|-ru\big|>K_q$, then
\[
 \min\{|z-ru|,|z+ru|\}=\big||z|-ru\big|>K_q.
\]
Therefore \eqref{eq:upper-localization} in Lemma~\ref{lem:upper-localization}
gives
\[
 \phi(z)\rho_r(u,z)
 \leq
 C\exp\left\{-\frac{K_q^2}{2(1-r^2)}\right\}
 \leq C e^{-K_q^2/2}.
\]
Taking the supremum over such $u$, with value zero if there is no such
$u$, and integrating over a set of length at most $2u_q$, gives
\begin{equation}
 \int_{L_q\leq |z|\leq u_q}\phi(z)
 \sup_{\substack{u\in\mathcal C_z:\\ \big||z|-ru\big|>K_q}}
 \rho_r(u,z)\dd z
 \leq C u_qe^{-K_q^2/2}
 =O(u_q^{-3}).                                    \label{eq:upper-away-from-localization}
\end{equation}
where we use the definition $K_q=\sqrt{8\log u_q}$ in
\eqref{eq:upper-KLq-def}.

\indent Now consider the remaining term in \eqref{eq:upper-crossing-decomposition}.  For any contributing
pair, $|z|\in[L_q,u_q]$ and $\big||z|-ru\big|\leq K_q$.  Since
$u\in\mathcal C_z$,
\[
 G_z^{\mathrm{large}}(u)+G_z^{\mathrm{small}}(u)
 =G_z(u)\geq\frac{p(u)}{2q}.
\]
Thus
\[
 G_z^{\mathrm{large}}(u)\geq\frac{p(u)}{4q}
 \quad\text{or}\quad
 G_z^{\mathrm{small}}(u)\geq\frac{p(u)}{4q}.
\]
Thus,
\[
 z\in E_{\operatorname{sgn}(z)}^{\mathrm{large}}(r)
       \cup E_{\operatorname{sgn}(z)}^{\mathrm{small}}(r),
\]
and the last integral in \eqref{eq:upper-crossing-decomposition} satisfies
\[
\begin{aligned}
&\int_{L_q\leq |z|\leq u_q}\phi(z)
 \sup_{\substack{u\in\mathcal C_z:\\ \big||z|-ru\big|\leq K_q}}
 \rho_r(u,z)\dd z\\
&\leq
\int_{z\in E_+^{\mathrm{large}}(r)\cup E_-^{\mathrm{large}}(r)}
\phi(z)
 \sup_{\substack{u\in\mathcal C_z:\\ \big||z|-ru\big|\leq K_q}}
 \rho_r(u,z)\dd z\\
&\quad+
\int_{z\in E_+^{\mathrm{small}}(r)\cup E_-^{\mathrm{small}}(r)}
\phi(z)
 \sup_{\substack{u\in\mathcal C_z:\\ \big||z|-ru\big|\leq K_q}}
 \rho_r(u,z)\dd z .
\end{aligned}
\]
Lemma~\ref{lem:upper-large-coefficients} gives
\[
 |E_+^{\mathrm{large}}(r)|+|E_-^{\mathrm{large}}(r)|
 \leq(1-r)u_q+C\sqrt{u_q}.
\]
Lemma~\ref{lem:upper-small-coefficients}
gives
\[
 |E_+^{\mathrm{small}}(r)|+|E_-^{\mathrm{small}}(r)|
 \leq C(K_q+1)\leq C\sqrt{u_q}.
\]
Combining the two bounds, the total combined length of the corresponding
aligned factor values is at most
\begin{equation}
 (1-r)u_q+C\sqrt{u_q}.                                \label{eq:upper-combined-length}
\end{equation}
Combining \eqref{eq:upper-localization-quantitative} in
Lemma~\ref{lem:upper-localization} and \eqref{eq:upper-combined-length}, and
using $0<r<3/5$, hence $(1-r^2)^{-1/2}\leq5/4$, we obtain
\begin{align*}
&\int_{L_q\leq |z|\leq u_q}\phi(z)
 \sup_{\substack{u\in\mathcal C_z:\\ \big||z|-ru\big|\leq K_q}}
 \rho_r(u,z)\dd z\\
&\leq
 \left\{\frac{1+C(K_q+1)/u_q}{2\sqrt{2\pi}\sqrt{1-r^2}}
       +Ce^{-u_q}\right\}
       \left\{(1-r)u_q+C\sqrt{u_q}\right\}\\
&\leq
 \frac{1-r}{2\sqrt{1-r^2}\sqrt{2\pi}}u_q
 +C\frac{K_q+1+\sqrt{u_q}}{\sqrt{1-r^2}}
 +C u_qe^{-u_q}\\
&\leq
 \frac{1-r}{2\sqrt{1-r^2}\sqrt{2\pi}}u_q
 +C\{K_q+1+\sqrt{u_q}+u_qe^{-u_q}\}.
\end{align*}
Here the middle line expands the product and absorbs the term
$(K_q+1)/\sqrt{u_q}$ into $K_q+1$.  Since $L_q=K_q+\sqrt{u_q}$,
$K_q=O(\sqrt{\log u_q})$, and $u_q e^{-u_q}$ is negligible, combining this with
\eqref{eq:upper-crossing-decomposition}, \eqref{eq:upper-small-z-integral}, and
\eqref{eq:upper-away-from-localization} yields
\begin{equation}
 \int_{-u_q}^{u_q}\phi(z)\sup_{u\in\mathcal C_z}\rho_r(u,z)\dd z
 \leq\frac{1-r}{2\sqrt{1-r^2}\sqrt{2\pi}}u_q+C\sqrt{u_q}
 \leq\frac{u_q}{2\sqrt{2\pi}}+C\sqrt{u_q},       \label{eq:upper-crossing-small-r}
\end{equation}
where the last inequality uses
$(1-r)/\sqrt{1-r^2}=\sqrt{(1-r)/(1+r)}\leq1$.
This proves Lemma~\ref{lem:upper-crossing}.
\end{proof}

The next lemma converts observed rejection counts into bounds for their
conditional means.
\begin{lemma}\label{lem:upper-rank}
There are universal constants $C,c>0$ with the following property.
Let $J\geq1$, and let $N_j$ be Poisson-binomial with mean
$\mu_j$; the counts may be nested in $j$.  For $j\geq J$, define the
bad events
\begin{align}
\mathcal B_j^+&=\{N_j\geq j-1,\ \mu_j<j/2\},\label{eq:upper-rank-upper}\\
\mathcal B_j^-&=\{N_j\leq j-2,\ \mu_j>2j\}. \label{eq:upper-rank-lower}
\end{align}
Then
\[
 \mathbb P\left(\bigcup_{j\geq J}(\mathcal B_j^+\cup\mathcal B_j^-)\right)
 \leq Ce^{-cJ}.
\]
The constants are universal and the bound is uniform in the number and
success probabilities of the Bernoulli variables.
\end{lemma}

\begin{proof}
If $\mu_j<j/2$, the Chernoff upper-tail inequality bounds
$\mathbb P(\mathcal B_j^+)=\mathbb P(N_j\geq j-1)$ by
$C_0e^{-c_0 j}$; if $\mu_j\geq j/2$, then $\mathcal B_j^+$ is empty.
Similarly, if $\mu_j>2j$, the Chernoff lower-tail inequality bounds
$\mathbb P(\mathcal B_j^-)=\mathbb P(N_j\leq j-2)$ by $C_0e^{-c_0 j}$,
while if $\mu_j\leq2j$, then $\mathcal B_j^-$ is empty.  Enlarging
$C_0$, if necessary, absorbs the finitely many small values of $j$.
Sum these bounds over $j\geq J$.
\end{proof}

\begin{proof}[Proof of Theorem~\ref{thm:common-factor-upper}~(b)]
If $m_0=0$, then $\FDR(\BH_q)=0$.  Hence assume $m_0>0$.
Let $R_i$ be the BH rejection count after replacing the true-null $p$-value $P_i$ by zero, as in the proof of Lemma~\ref{lem:conditional-fdp-bound}.  For every hypothesis $\ell$, define
\[
 F_{\ell,z}(t)=\mathbb P(P_\ell\leq t\mid Z=z),\qquad
 t_k=\frac{qk}{m},\qquad u_k=p^{-1}(t_k).
\]
The leave-one-out identity \eqref{eq:loo-identity} and the same conditional-independence argument give, for almost every $z$,
\begin{align}
 \FDR(z)
 &=\frac qm\sum_{i\in\mathcal H_0}\sum_{k=1}^m
   \frac{F_{i,z}(t_k)}{t_k}
   \mathbb P(R_i=k\mid Z=z) \notag\\
 &=\frac qm\sum_{i\in\mathcal H_0}\sum_{k=1}^m
   \rho_{|a_i|}(u_k,z)
   \mathbb P(R_i=k\mid Z=z).           \label{eq:upper-fdr-rank}
\end{align}
Thus $\FDR=\int_{\mathbb R}\FDR(z)\phi(z)\dd z$.
The second equality follows from \eqref{eq:upper-null-tail-ratio}.
Let
\[
 A_q=\{|Z|>u_q\}.
\]
Decompose
\[
 \FDR
 =\mathbb E\{\FDR(Z)\,\mathbf1_{A_q}\}
  +\mathbb E\{\FDR(Z)\,\mathbf1_{A_q^c}\}.
\]
The first term is at most $q$:
\[
 \mathbb E\{\FDR(Z)\,\mathbf1_{A_q}\}
 \leq \mathbb P(A_q)
 =2\Phib(u_q)
 =p(u_q)
 =q,
\]
because $0\leq\FDR(Z)\leq1$ and $u_q=p^{-1}(q)$.
It remains to prove
\begin{equation}
 \mathbb E\{\FDR(Z)\,\mathbf1_{A_q^c}\}
 \leq
 \pi_0q\left\{\frac{u_q}{2\sqrt{2\pi}}+C\sqrt{u_q}\right\}.
 \label{eq:upper-central-target}
\end{equation}
Define
\[
 J=\lceil\sqrt{u_q}\rceil.
\]
Then $J\geq2$ when $q$ is sufficiently small.  We split the summands in
\eqref{eq:upper-central-target} into two parts.  It suffices to prove
\begin{equation}
 \begin{aligned}
 &\int_{A_q^c}\frac qm\sum_{i\in\mathcal H_0}\sum_{1\leq k\leq J}
 \rho_{|a_i|}(u_k,z)\mathbb P(R_i=k\mid Z=z)\phi(z)\dd z \\
 &\qquad\leq \pi_0qJ
 \leq\pi_0q(\sqrt{u_q}+1),
 \end{aligned}                                      \label{eq:upper-small-rank-target}
\end{equation}
and
\begin{equation}
 \begin{aligned}
 &\int_{A_q^c}\frac qm\sum_{i\in\mathcal H_0}\sum_{k>J}
 \rho_{|a_i|}(u_k,z)\mathbb P(R_i=k\mid Z=z)\phi(z)\dd z \\
 &\qquad\leq
 \pi_0q\left\{\frac{u_q}{2\sqrt{2\pi}}+C\sqrt{u_q}\right\}.
 \end{aligned}                                      \label{eq:upper-large-rank-target}
\end{equation}

We first prove \eqref{eq:upper-small-rank-target}.  Its left-hand side is
bounded directly from \eqref{eq:upper-fdr-rank}.  Using
\eqref{eq:upper-rho-normalization} with $u=u_k$,
$\mathbf1_{A_q^c}\leq1$, and
$\mathbb P(R_i=k\mid Z=z)\leq1$, it is at most
\[
 \frac qm\sum_{i\in\mathcal H_0}\sum_{1\leq k\leq J}1
 =\pi_0qJ
 \leq\pi_0q(\sqrt{u_q}+1).
\]

We next prove \eqref{eq:upper-large-rank-target}.  For each $i$ and
$z$, and for $1\leq j\leq m$, let
\begin{equation}
 N_{i,j}=\sum_{\ell\ne i}\mathbf1\{P_\ell\leq qj/m\},
 \qquad
 \mu_{i,j}(z)=\mathbb E(N_{i,j}\mid Z=z).        \label{eq:upper-counts}
\end{equation}
Define the good event corresponding to
\eqref{eq:upper-rank-upper}-\eqref{eq:upper-rank-lower} as follows.  Let
\[
 \mathcal B_{i,j}^+(z)=\{N_{i,j}\geq j-1,\ \mu_{i,j}(z)<j/2\},\qquad
 \mathcal B_{i,j}^-(z)=\{N_{i,j}\leq j-2,\ \mu_{i,j}(z)>2j\},
\]
and set
\begin{equation}
 \mathcal A_{i,z}
 :=
 \bigcap_{j=J}^m
 \left((\mathcal B_{i,j}^+(z))^c\cap(\mathcal B_{i,j}^-(z))^c\right).
 \label{eq:upper-good-event}
\end{equation}
We split the conditional probability in the left-hand side of
\eqref{eq:upper-large-rank-target} according to this event:
\[
 \mathbb P(R_i=k\mid Z=z)
 =\mathbb P\{R_i=k,\mathcal A_{i,z}\mid Z=z\}
  +\mathbb P\{R_i=k,\mathcal A_{i,z}^c\mid Z=z\}.
\]
It is enough to prove the bad-event bound
\begin{equation}
 \begin{aligned}
 &\int_{A_q^c}\frac qm\sum_{i\in\mathcal H_0}\sum_{k>J}
 \rho_{|a_i|}(u_k,z)\mathbb P\{R_i=k,\mathcal A_{i,z}^c\mid Z=z\}
 \phi(z)\dd z \\
 &\qquad\leq C\pi_0q,
 \end{aligned}
 \label{eq:upper-large-rank-bad-target}
\end{equation}
and the good-event bound
\begin{equation}
 \begin{aligned}
 &\int_{A_q^c}\frac qm\sum_{i\in\mathcal H_0}\sum_{k>J}
 \rho_{|a_i|}(u_k,z)\mathbb P\{R_i=k,\mathcal A_{i,z}\mid Z=z\}
 \phi(z)\dd z \\
 &\qquad\leq
 \pi_0q\left\{\frac{u_q}{2\sqrt{2\pi}}+C\sqrt{u_q}\right\}.
 \end{aligned}
 \label{eq:upper-large-rank-good-target}
\end{equation}

We first prove the bad-event bound \eqref{eq:upper-large-rank-bad-target}.
Set $\delta_J=Ce^{-cJ}$.  Fix $i$ and $z$, and condition on
$Z=z$.  The residuals of the hypotheses $\ell\ne i$ are then
independent, so the indicators in \eqref{eq:upper-counts} are independent
Bernoulli random variables, with success probabilities depending on
$\ell$, $j$, and $z$.  Hence each $N_{i,j}$ is Poisson-binomial
with mean $\mu_{i,j}(z)$, and the family may be nested in $j$, exactly as
allowed in Lemma~\ref{lem:upper-rank}.  Moreover,
$\mathcal A_{i,z}^c$ is the union over $J\leq j\leq m$ of the two events
\[
 \{N_{i,j}\geq j-1,\ \mu_{i,j}(z)<j/2\}
 \quad\text{and}\quad
 \{N_{i,j}\leq j-2,\ \mu_{i,j}(z)>2j\},
\]
which are exactly the bad events in Lemma~\ref{lem:upper-rank} with
$N_j=N_{i,j}$ and $\mu_j=\mu_{i,j}(z)$.  Thus the lemma gives, uniformly
in $z$,
\begin{equation}
 \mathbb P(\mathcal A_{i,z}^c\mid Z=z)\leq\delta_J.
 \label{eq:upper-good-event-prob}
\end{equation}
Using \eqref{eq:upper-good-event-prob}, and since $u_k\geq u_q$ for
$k>J$, for each $i$ and $z$,
\[
 \begin{aligned}
 &\sum_{k>J}\rho_{|a_i|}(u_k,z)
 \mathbb P\{R_i=k,\mathcal A_{i,z}^c\mid Z=z\} \\
 &\qquad\leq
 \sup_{u\geq u_q}\rho_{|a_i|}(u,z)
 \sum_{k>J}\mathbb P\{R_i=k,\mathcal A_{i,z}^c\mid Z=z\} \\
 &\qquad\leq
 \sup_{u\geq u_q}\rho_{|a_i|}(u,z)
 \mathbb P(\mathcal A_{i,z}^c\mid Z=z)
 \leq
 \delta_J\sup_{u\geq u_q}\rho_{|a_i|}(u,z).
 \end{aligned}
\]
Since $A_q^c=\{|z|\leq u_q\}$, inserting this pointwise bound into the
left-hand side of \eqref{eq:upper-large-rank-bad-target} gives
\begin{equation}
 \begin{aligned}
 &\int_{A_q^c}\frac qm\sum_{i\in\mathcal H_0}\sum_{k>J}
 \rho_{|a_i|}(u_k,z)\mathbb P\{R_i=k,\mathcal A_{i,z}^c\mid Z=z\}
 \phi(z)\dd z \\
 &\qquad\leq
 \frac qm\sum_{i\in\mathcal H_0}\delta_J
 \int_{-u_q}^{u_q}\phi(z)
 \sup_{u\geq u_q}\rho_{|a_i|}(u,z)\dd z \\
 &\qquad\leq
 \frac qm|\mathcal H_0|\delta_J C_q
 =\pi_0q\delta_J C_q
 \leq C\pi_0q(1+u_q)e^{-c\sqrt{u_q}}
 \leq C\pi_0q .
 \end{aligned}                                      \label{eq:upper-bad-rank}
\end{equation}
Here Lemma~\ref{lem:upper-null-ratio} gives
\[
 C_q:=\sup_{0\leq r\leq1}
 \int_{-u_q}^{u_q}\phi(z)\sup_{u\geq u_q}\rho_r(u,z)\dd z
 \leq\frac{1-q}{2}+\frac{u_q}{\sqrt{2\pi}}.
\]

Lastly, we prove the good-event bound
\eqref{eq:upper-large-rank-good-target}.
Here $\mathcal A_{i,z}$ is defined in \eqref{eq:upper-good-event}.
On $\{R_i=k\}$,
\begin{equation}
 N_{i,k}=k-1,\qquad N_{i,j}\leq j-2\quad(j>k).        \label{eq:upper-threshold-counts}
\end{equation}
Consequently, on $\{R_i=k\}\cap\mathcal A_{i,z}$, the exclusions in
\eqref{eq:upper-good-event} applied with $j=k$ and then with each $j>k$
give
\begin{equation}
 \mu_{i,k}(z)\geq k/2,\qquad
 \mu_{i,j}(z)\leq2j\quad(j>k).                 \label{eq:upper-count-means}
\end{equation}

Define the averaged conditional sub-CDF of the other hypotheses and its
scale version by
\begin{equation}
 F_{-i,z}(t)=\frac1m\sum_{\ell\ne i}F_{\ell,z}(t),
 \qquad G_{i,z}(u)=F_{-i,z}\{p(u)\}.                \label{eq:upper-mixture-other}
\end{equation}
For each $\ell\ne i$, the common-factor representation
\eqref{eq:common-factor-model} gives, conditionally on $Z=z$,
\[
 F_{\ell,z}\{p(u)\}
 =\mathbb P\{|T_\ell|\geq u\mid Z=z\}
 =\mathbb P\{|\theta_\ell+a_\ell z+\sqrt{1-a_\ell^2}\,\varepsilon_\ell|
             \geq u\}.
\]
Therefore
\[
 G_{i,z}(u)
 =\int \mathbb P\{|\theta+\alpha z+\sqrt{1-\alpha^2}\,\varepsilon|
                    \geq u\}\,\dd\Pi_i(\theta,\alpha),
 \qquad
 \Pi_i=\frac1m\sum_{\ell\ne i}\delta_{(\theta_\ell,a_\ell)}.
\]
This is exactly the form \eqref{eq:upper-mixture}.

We now verify the hypotheses of Lemma~\ref{lem:upper-crossing}, one
true null $i$ at a time.  For this fixed $i$, define
\[
 w_{i,k}(z)=\mathbb P\{R_i=k,\mathcal A_{i,z}\mid Z=z\},\qquad J<k\leq m,
\]
and take $G_z=G_{i,z}$.  First, the weights are nonnegative and, because the events
$\{R_i=k\}$, $1\leq k\leq m$, are disjoint,
\[
 \sum_{k>J}w_{i,k}(z)\leq1.
\]
Second, whenever $w_{i,k}(z)>0$, the event
$\{R_i=k,\mathcal A_{i,z}\}$ has positive conditional probability.  Since
the quantities $\mu_{i,j}(z)$ are fixed after conditioning on $Z=z$, the
implication above gives \eqref{eq:upper-count-means} for this $i,z,k$.
Using $\mu_{i,j}(z)=mG_{i,z}(u_j)$, these inequalities become
\[
 G_{i,z}(u_k)\geq\frac{k}{2m},\qquad
 G_{i,z}(u_j)\leq\frac{2j}{m}\quad(j>k).
\]
Thus \eqref{eq:upper-rank-weight-conditions} holds for every $k$ with
$w_{i,k}(z)>0$, with $G_z=G_{i,z}$ and $w_k(z)=w_{i,k}(z)$.
Applying Lemma~\ref{lem:upper-crossing} with $r=|a_i|$ gives
\[
 \int_{-u_q}^{u_q}\sum_{k>J}w_{i,k}(z)
 \rho_{|a_i|}(u_k,z)\phi(z)\dd z
 \leq
 \frac{u_q}{2\sqrt{2\pi}}+C\sqrt{u_q},
\]
where $A_q^c=\{|z|\leq u_q\}$.
Summing the preceding bound over $i\in\mathcal H_0$ and using
$|\mathcal H_0|/m=\pi_0$ gives
\eqref{eq:upper-large-rank-good-target}.

We have proved part (b) for $q\le q_0$.  If
$q_0<q\le2\bar\Phi(1)$, part (a)
gives
\[
\FDR(\BH_q)
\le q\left\{1+\pi_0\left(\frac{1-q}{2}
+\frac{u_q}{\sqrt{2\pi}}\right)\right\}.
\]
When $q\in[q_0,2\bar\Phi(1)]$, the quantities
$u_q$ and $\log(1/q)$ are bounded above and $\log(1/q)$ is bounded away
from zero.  Hence, after enlarging the universal constant $C$,
\[
\frac{1-q}{2}+\frac{u_q}{\sqrt{2\pi}}
\le
\frac1{2\sqrt\pi}\sqrt{\log(1/q)}
+C\{\log(1/q)\}^{1/4},
\]
which closes part (b) on the remaining $q$-range.
\end{proof}

\begin{remark}[PRDN for the two-sided null $p$-values]\label{rem:prdn-pvalues}
The theorem imposes no sign restriction on the loadings $a_i$, so the null
correlations $a_ia_j$ may be negative.  Nevertheless, the two-sided null
$p$-values satisfy PRDN.  Let
\[
 T_0=(T_i:i\in\mathcal H_0),\qquad
 a_0=(a_i:i\in\mathcal H_0),\qquad
 \varepsilon_0=(\varepsilon_i:i\in\mathcal H_0).
\]
Put
\[
 D_0=\operatorname{diag}(1-a_i^2:i\in\mathcal H_0).
\]
Since $\theta_i=0$ for $i\in\mathcal H_0$,
\[
 T_0=a_0Z+\sqrt{D_0}\,\varepsilon_0.
\]
Hence $T_0$ is centered Gaussian with covariance matrix
\[
 \Sigma_0=D_0+a_0a_0^\top.
\]
Assume first that $|a_i|<1$ for every null coordinate.  Let
\[
 B=\operatorname{diag}\{\operatorname{sign}(a_i):i\in\mathcal H_0\},
 \qquad
 r_0=(|a_i|:i\in\mathcal H_0)^\top,
\]
with $\operatorname{sign}(0)=1$.  Then
\[
 B\Sigma_0B=D_0+r_0r_0^\top.
\]
By the Sherman-Morrison formula,
\[
 (D_0+r_0r_0^\top)^{-1}
 =D_0^{-1}
  -\frac{D_0^{-1}r_0r_0^\top D_0^{-1}}
  {1+r_0^\top D_0^{-1}r_0}.
\]
Thus every off-diagonal entry of $(D_0+r_0r_0^\top)^{-1}$ is nonpositive.  The
Karlin-Rinott absolute-value multinormal criterion for
multivariate total positivity of order two (MTP$_2$)
\citep{KarlinRinott1981} therefore applies to $BT_0$, and
$(|(BT_0)_i|:i\in\mathcal H_0)$ is MTP$_2$.  Since
$|(BT_0)_i|=|T_i|$, coordinatewise sign changes do not alter the
absolute null statistics.  Applying the same decreasing map
$g(x)=p(x)=2\Phib(x)$ to every coordinate preserves MTP$_2$.
Indeed, for $u,v\in(0,1)^{m_0}$, coordinatewise monotonicity of $g^{-1}$
gives
\[
 g^{-1}(u\wedge v)=g^{-1}(u)\vee g^{-1}(v),
 \qquad
 g^{-1}(u\vee v)=g^{-1}(u)\wedge g^{-1}(v).
\]
If $f$ is the density of $(|(BT_0)_i|:i\in\mathcal H_0)$, the density
after the transformation is
\[
 \widetilde f(u)
 =f\{g^{-1}(u)\}\prod_{i\in\mathcal H_0}|(g^{-1})'(u_i)|.
\]
For each coordinate,
$\{u_i\wedge v_i,u_i\vee v_i\}=\{u_i,v_i\}$; hence the two products of
Jacobian factors are equal.  Therefore the MTP$_2$ inequality for $f$
implies
\[
\begin{aligned}
 \widetilde f(u\wedge v)\widetilde f(u\vee v)
 &=f\{g^{-1}(u)\vee g^{-1}(v)\}
   f\{g^{-1}(u)\wedge g^{-1}(v)\}
   \prod_{i\in\mathcal H_0}|(g^{-1})'(u_i)(g^{-1})'(v_i)|\\
 &\geq f\{g^{-1}(u)\}f\{g^{-1}(v)\}
   \prod_{i\in\mathcal H_0}|(g^{-1})'(u_i)(g^{-1})'(v_i)|\\
 &=\widetilde f(u)\widetilde f(v).
\end{aligned}
\]
Thus the vector of two-sided null $p$-values is MTP$_2$, which implies
PRDN.

If some $|a_i|=1$, define $a_i^{(\eta)}=(1-\eta)a_i$ for all
$i\in\mathcal H_0$, and let $\eta\downarrow0$.  The corresponding
absolute null vectors converge weakly to
$(|T_i|:i\in\mathcal H_0)$.  Under the measure-theoretic definition
of MTP$_2$, the class of MTP$_2$ probability measures is closed under
weak convergence \citep{ColangeloMullerScarsini2006}.  Hence the
limiting absolute-value law remains MTP$_2$.  Appendix A.1 of
\citet{Su2018} then implies that the induced vector of two-sided null
$p$-values satisfies PRDN.
\end{remark}

\begingroup

\subsection{Upper bounds for one-sided tests}

We derive an FDR upper bound for one-sided tests under the
common-factor model and show that its leading constant cannot be improved within
this class.

\begin{theorem}
\label{thm:one-upper}
Consider the common-factor model
\eqref{eq:common-factor-model}.  Test
$H_i:\theta_i\leq0$ with $P_i=\Phib(T_i)$.  Let
$\mathcal H_0=\{i:\theta_i\leq0\}$, $m_0=|\mathcal H_0|$, and
$\pi_0=m_0/m$.  For $0<q<1/2$, put
$u_q^+={\Phib}^{-1}(q)>0$.  The FDR is zero when $m_0=0$, and otherwise
\begin{equation}
 \FDR(\BH_q)
 \leq q\left\{2+\pi_0\left(\frac12+\frac{u_q^+}{\sqrt{2\pi}}\right)\right\}.
 \label{eq:one-upper-bound}
\end{equation}
Consequently, the worst-case FDR in this common-factor class is
$O\{q\sqrt{\log(1/q)}\}$, uniformly in $m$, the means
$(\theta_1,\ldots,\theta_m)$, and the loadings
$(a_1,\ldots,a_m)$, as $q\downarrow0$.
\end{theorem}

The following corollary identifies the exact leading constant for
the common-factor class.
\begin{corollary}
\label{cor:one-sharp}
For the common-factor model
\eqref{eq:common-factor-model}, where $\FDR_{\theta,a}$ denotes the FDR
under that model, as $q\downarrow0$,

\[
 \sup_{\substack{m\in\N,\ \theta\in\R^m,\\
                    a\in[-1,1]^m}}
 \FDR_{\theta,a}(\BH_q)
 =\frac{q\sqrt{\log(1/q)}}{\sqrt\pi}+O(q).
\]

Thus the constant $1/\sqrt\pi$ in
Proposition~\ref{prop:one-expansion} is optimal in the common-factor class.
\end{corollary}

\begin{proof}[Proof of Theorem~\ref{thm:one-upper}]
If there are no true nulls, the FDR is zero.  Partition the true nulls into
the two sign groups
\[
 \mathcal H_{0,1}=\{i\in\mathcal H_0:a_i\geq0\},\qquad
 \mathcal H_{0,-1}=\{i\in\mathcal H_0:a_i<0\}.
\]
{Let $\mathcal S\subseteq\{-1,1\}$ index the nonempty groups.
The leave-one-out identity \eqref{eq:loo-identity} and conditional reduction
\eqref{eq:conditional-loo-term-bound} are independent of the particular
form of the $p$-values.

Suppose $i\in\mathcal H_{0,\sigma}$, where
$\sigma\in\mathcal S$, and write
$r=|a_i|$ and
$Y_\sigma=\sigma Z\sim N(0,1)$.  Define the actual conditional CDF by
\[
 F^{\theta_i}_{i,y}(t)
 =\mathbb P_{\theta_i}(P_i\leq t\mid Y_\sigma=y)
\]
and let $F^0_{r,y}$ denote the corresponding conditional CDF at the
boundary mean $\theta_i=0$.  Conditional on $Y_\sigma=y$,
$T_i\sim N(\theta_i+ry,1-r^2)$.  Since $\theta_i\leq0$ and rejection is
increasing in $T_i$,
\[
 F^{\theta_i}_{i,y}(t)\leq F^0_{r,y}(t),\qquad 0<t<1.
\]
Thus it suffices in the conditional leave-one-out bound to control the
boundary CDF $F^0_{r,y}$.

For $r<1$, the one-tail Gaussian density-ratio calculation gives
\begin{equation}
 \phi(y)\sup_{0<t\leq q}\frac{F^0_{r,y}(t)}t
 \leq\frac1{\sqrt{1-r^2}\sqrt{2\pi}}
 \exp\left\{-\frac{(ru_q^+-y)_+^2}{2(1-r^2)}\right\}.
 \label{eq:one-conditional-ratio-bound}
\end{equation}
Indeed, write $t=p_+(u)$, where $u\geq u_q^+$.  The bivariate normal
density identity
\[
 \phi(y)\frac1{\sqrt{1-r^2}}\phi\!\left(\frac{x-ry}{\sqrt{1-r^2}}\right)
 =\phi(x)\frac1{\sqrt{1-r^2}}\phi\!\left(\frac{y-rx}{\sqrt{1-r^2}}\right)
\]
and $p_+(u)=\int_u^\infty\phi(x)\dd x$ give

\begin{align}
 \phi(y)\frac{F^0_{r,y}\{p_+(u)\}}{p_+(u)}
 &=\frac{\displaystyle\int_u^\infty
 \phi(x)\frac1{\sqrt{1-r^2}}\phi\!\left(\frac{y-rx}{\sqrt{1-r^2}}\right)\dd x}
 {\displaystyle\int_u^\infty\phi(x)\dd x}\notag\\
 &\leq\sup_{x\geq u}\frac1{\sqrt{1-r^2}}
 \phi\!\left(\frac{y-rx}{\sqrt{1-r^2}}\right)\notag\\
 &\leq\frac1{\sqrt{1-r^2}\sqrt{2\pi}}
 \exp\left\{-\frac{(ru-y)_+^2}{2(1-r^2)}\right\}.
 \label{eq:one-ratio-at-u}
\end{align}

The final inequality in \eqref{eq:one-ratio-at-u} follows
because, for every $x\geq u$,
\[
 (rx-y)^2\geq(ru-y)_+^2.
\]
Since $r\geq0$, the map $u\mapsto(ru-y)_+$ is nondecreasing.  Thus
the right-hand side of \eqref{eq:one-ratio-at-u} is
nonincreasing in $u$ and is maximized over
$u\geq u_q^+$ at $u=u_q^+$, proving
\eqref{eq:one-conditional-ratio-bound}.  Integrating over $y\leq u_q^+$
yields

\begin{align}
 I(r)
 &:=\int_{-\infty}^{u_q^+}\phi(y)
 \sup_{0<t\leq q}\frac{F^0_{r,y}(t)}t\dd y\notag\\
 &\leq\frac1{\sqrt{1-r^2}\sqrt{2\pi}}
 \left\{
 \int_{-\infty}^{ru_q^+}
 e^{-(ru_q^+-y)^2/\{2(1-r^2)\}}\dd y
 +\int_{ru_q^+}^{u_q^+}1\dd y
 \right\}\notag\\
 &=\frac1{\sqrt{1-r^2}\sqrt{2\pi}}
 \left\{
 \int_0^\infty e^{-v^2/\{2(1-r^2)\}}\dd v
 +(1-r)u_q^+
 \right\}\notag\\
 &=\frac12+\frac{(1-r)u_q^+}{\sqrt{1-r^2}\sqrt{2\pi}}\notag\\
 &\leq\frac12+\frac{u_q^+}{\sqrt{2\pi}}.
 \label{eq:one-integrated-ratio}
\end{align}

The inequality
$(1-r)u_q^+/\{\sqrt{1-r^2}\sqrt{2\pi}\}\leq u_q^+/\sqrt{2\pi}$ in
\eqref{eq:one-integrated-ratio} follows from
$(1-r)/\sqrt{1-r^2}=\sqrt{(1-r)/(1+r)}\leq1$.  If $r=1$, then on
$Y_\sigma<u_q^+$ the boundary-null p-value is larger than $q$, so
\eqref{eq:one-integrated-ratio} remains valid.

Let $V_\sigma$ be the false rejections in sign group $\sigma$ and
$\FDP_\sigma=V_\sigma/(R\vee1)$.  On
$A_\sigma=\{Y_\sigma\leq u_q^+\}$, sum the per-null conditional bound over
$\mathcal H_{0,\sigma}$ and apply \eqref{eq:one-integrated-ratio}.  This gives
\[
 \mathbb E(\FDP_\sigma\1_{A_\sigma})
 \leq q\frac{|\mathcal H_{0,\sigma}|}{m}
 \left(\frac12+\frac{u_q^+}{\sqrt{2\pi}}\right).
\]
On $A_\sigma^c$, use $0\leq\FDP_\sigma\leq1$ and
$\mathbb P(A_\sigma^c)=q$.  Summing over $\sigma\in\mathcal S$, using
$|\mathcal S|\leq2$ and
$\FDP=\sum_{\sigma\in\mathcal S}\FDP_\sigma$, proves
\eqref{eq:one-upper-bound}.
}
\end{proof}

\begin{proof}[Proof of Corollary~\ref{cor:one-sharp}]
Theorem~\ref{thm:one-lower} gives the lower bound within the
common-factor class.  Indeed, for the perturbed model \eqref{eq:eta-model}, set
\[
 a_i^{(\eta)}=\sqrt{1-\eta^2}\,b_i,
 \qquad
 \widetilde\varepsilon_i^{(\eta)}
 =\frac{\sqrt{(1-\eta^2)(1-b_i^2)}\,\varepsilon_i+\eta\xi_i}
 {\sqrt{1-\{a_i^{(\eta)}\}^2}}.
\]
The variables $\widetilde\varepsilon_i^{(\eta)}$ are independent standard
normal variables and are independent of $Z$, and
\[
 T_i^{(\eta)}
 =\theta_i+a_i^{(\eta)}Z
  +\sqrt{1-\{a_i^{(\eta)}\}^2}\,
   \widetilde\varepsilon_i^{(\eta)}.
\]
Thus the perturbation is of the form \eqref{eq:common-factor-model}.
Together with Proposition~\ref{prop:one-expansion}, this gives

\[
 \sup_{\substack{m\in\N,\ \theta\in\R^m,\\
                    a\in[-1,1]^m}}
 \FDR_{\theta,a}(\BH_q)
 \geq \ell_{\le}(q)
 =\frac{q\sqrt{\log(1/q)}}{\sqrt\pi}+\frac q2+o(q).
\]

For the matching upper bound, note first that
$u_q^+=p^{-1}(2q)$, so \eqref{eq:mills-quantile} in
Lemma~\ref{lem:mills} gives
$u_q^+\sim\sqrt{2\log(1/q)}$.  More precisely,
\eqref{eq:mills-asymp} in Lemma~\ref{lem:mills} and
$q=\Phib(u_q^+)$ imply
\[
 (u_q^+)^2
 =2\log(1/q)-2\log u_q^+-\log(2\pi)+o(1)
 =2\log(1/q)+O\{\log\log(1/q)\}.
\]
Consequently,
\[
 u_q^+=\sqrt{2\log(1/q)}
 +O\left\{\frac{\log\log(1/q)}{\sqrt{\log(1/q)}}\right\}
 =\sqrt{2\log(1/q)}+o(1).
\]
Theorem~\ref{thm:one-upper}, with $\pi_0\leq1$, now gives the reverse
inequality up to an $O(q)$ term, proving the claim.
\end{proof}
\endgroup

\section{Discussion}

For two-sided problems, Theorem~\ref{thm:main} and
Proposition~\ref{prop:ell-expansion} show that no
bound of the form $Cq$, with $C$ independent of $q$, can hold uniformly
over all Gaussian correlation matrices.  Because the example that achieves
the lower bound belongs to the one-common-factor class studied in
Section~\ref{sec:upper-bound}, Theorem~\ref{thm:common-factor-upper} shows
that the worst-case FDR within this class has the sharp leading constant
$1/(2\sqrt\pi)$. Similarly, For one-sided tests,
Theorem~\ref{thm:one-lower} and Corollary~\ref{cor:one-sharp} give the
optimal leading constant $1/\sqrt\pi$ within the same class.

In the literature, there are procedures with finite-sample FDR guarantees
for arbitrary dependent two-sided Gaussian tests. The conditional-calibration
framework of \citet{FithianLei2022} decomposes FDR into hypothesis-specific
contributions and, for each null $i$, conditions on a null sufficient statistic
under which the $i$th $p$-value remains conditionally valid, often eliminating
nuisance parameters. The resulting guarantees depend on the procedure and
the dependence assumptions: under PRDS, the procedure controls FDR and its
rejection set contains that of BH. Under arbitrary dependence, the procedure
controls FDR and its rejection set contains that of the Benjamini-Yekutieli
procedure \citep{BenjaminiYekutieli2001}. Conditional calibration can therefore
exploit known dependence without automatically paying the full $O(\log N)$
correction of \citet{BenjaminiYekutieli2001}. Another route is the fixed-$X$
knockoff filter of \citet{BarberCandes2015}. In a Gaussian linear model,
knockoffs augment the covariate matrix with synthetic variables that mimic
the correlation structure of the original variables, and achieve the FDR
control through the symmetry between the original test statistic and its knockoff.
\citet{LiFithian2021} recast fixed-$X$ knockoffs as adding user-generated
Gaussian noise to a Gaussian estimator to obtain a whitened estimator,
followed by conditional inference.  This whitening can substantially reduce
power when the covariance has large leading eigenvalues associated with dense
leading eigenvectors, because it then requires enough added noise to erase much
of the signal.  Knockoffs also face a threshold phenomenon when the rejection
set is small.  The calibrated knockoff method of \citet{LuoFithianLei2025}
uses conditional calibration to uniformly improve knockoff procedures, with
especially large gains in small-rejection regimes.

This discussion points to two complementary directions: identifying restrictions of
dependence structures  under which the standard BH remains reliable, and developing other procedures
that leverage the dependence structures to control FDR without further restrictions.

\section*{Acknowledgments and disclosure}

The author thanks Will Fithian for encouraging him to work on the two-sided
problems, Ruodu Wang for encouraging him to work on the one-sided problems,
and Edgar Dobriban for the surprising initial counterexample that reignited
interest in this project.  The author used OpenAI GPT-5.6 Sol for assistance
with proof exploration, exposition, and editing.  The author has been working
on this problem for over five years and the results in this paper could not
have been developed without the significant input of GPT 5.6 Sol.  Nevertheless,
the human input is equally important for the research paper: the author spent
about 180 hours on idea generation, verification, correction, simplification,
and writing.  From the author's perspective, it is a remarkably rewarding
experience of Human-AI collaboration.

\bibliographystyle{plainnat}
\bibliography{FDR_BH}

@misc{Dobriban2026,
  author        = {Dobriban, Edgar},
  title         = {The {Benjamini--Hochberg} procedure can fail to control the {FDR} for correlated two-sided {Gaussian} tests},
  year          = {2026},
  eprint        = {2607.12208},
  archivePrefix = {arXiv},
  url           = {https://arxiv.org/abs/2607.12208}
}

@article{ReinerBenaim2007,
  author  = {Reiner-Benaim, Anat},
  title   = {{FDR} control by the {BH} procedure for two-sided correlated tests with implications to gene expression data analysis},
  journal = {Biometrical Journal},
  year    = {2007},
  volume  = {49},
  number  = {1},
  pages   = {107--126},
  doi     = {10.1002/bimj.200510313}
}

@article{BenjaminiHochberg1995,
  author  = {Benjamini, Yoav and Hochberg, Yosef},
  title   = {Controlling the false discovery rate: A practical and powerful approach to multiple testing},
  journal = {Journal of the Royal Statistical Society: Series B},
  year    = {1995},
  volume  = {57},
  number  = {1},
  pages   = {289--300},
  doi     = {10.1111/j.2517-6161.1995.tb02031.x}
}

@article{BenjaminiYekutieli2001,
  author  = {Benjamini, Yoav and Yekutieli, Daniel},
  title   = {The control of the false discovery rate in multiple testing under dependency},
  journal = {The Annals of Statistics},
  year    = {2001},
  volume  = {29},
  number  = {4},
  pages   = {1165--1188},
  doi     = {10.1214/aos/1013699998}
}

@article{KarlinRinott1981,
  author  = {Karlin, Samuel and Rinott, Yosef},
  title   = {Total positivity properties of absolute value multinormal variables with applications to confidence interval estimates and related probabilistic inequalities},
  journal = {The Annals of Statistics},
  year    = {1981},
  volume  = {9},
  number  = {5},
  pages   = {1035--1049},
  doi     = {10.1214/aos/1176345583}
}

@article{ColangeloMullerScarsini2006,
  author  = {Colangelo, Antonio and M{\"u}ller, Alfred and Scarsini, Marco},
  title   = {Positive dependence and weak convergence},
  journal = {Journal of Applied Probability},
  year    = {2006},
  volume  = {43},
  number  = {1},
  pages   = {48--59},
  doi     = {10.1239/jap/1143936242}
}

@misc{Sarkar2023,
  author        = {Sarkar, Sanat K.},
  title         = {On controlling the false discovery rate in multiple testing of the means of correlated normals against two-sided alternatives},
  year          = {2023},
  eprint        = {2304.05261},
  archivePrefix = {arXiv},
  url           = {https://arxiv.org/abs/2304.05261}
}

@article{SarkarZhang2025,
  author  = {Sarkar, Sanat K. and Zhang, Shiyu},
  title   = {Shifted {BH} methods for controlling false discovery rate in multiple testing of the means of correlated normals against two-sided alternatives},
  journal = {Journal of Statistical Planning and Inference},
  year    = {2025},
  volume  = {236},
  pages   = {106238},
  doi     = {10.1016/j.jspi.2024.106238}
}

@misc{Su2018,
  author        = {Su, Weijie J.},
  title         = {The {FDR}-linking theorem},
  year          = {2018},
  eprint        = {1812.08965},
  archivePrefix = {arXiv},
  url           = {https://arxiv.org/abs/1812.08965}
}

@article{FithianLei2022,
  author  = {Fithian, William and Lei, Lihua},
  title   = {Conditional calibration for false discovery rate control under dependence},
  journal = {The Annals of Statistics},
  year    = {2022},
  volume  = {50},
  number  = {6},
  pages   = {3091--3118},
  doi     = {10.1214/21-AOS2137}
}

@article{BarberCandes2015,
  author  = {Barber, Rina Foygel and Cand{\`e}s, Emmanuel J.},
  title   = {Controlling the false discovery rate via knockoffs},
  journal = {The Annals of Statistics},
  year    = {2015},
  volume  = {43},
  number  = {5},
  pages   = {2055--2085},
  doi     = {10.1214/15-AOS1337}
}

@misc{LiFithian2021,
  author        = {Li, Xiao and Fithian, William},
  title         = {Whiteout: When Do Fixed-{X} Knockoffs Fail?},
  year          = {2021},
  eprint        = {2107.06388},
  archivePrefix = {arXiv},
  url           = {https://arxiv.org/abs/2107.06388}
}

@article{LuoFithianLei2025,
  author  = {Luo, Yixiang and Fithian, William and Lei, Lihua},
  title   = {Improving knockoffs with conditional calibration},
  journal = {The Annals of Statistics},
  year    = {2025},
  volume  = {53},
  number  = {5},
  pages   = {2283--2302},
  doi     = {10.1214/25-AOS2543}
}

@article{chi2025multiple,
  title =	 {Multiple testing under negative dependence},
  author =	 {Chi, Ziyu and Ramdas, Aaditya and Wang, Ruodu},
  journal =	 {Bernoulli},
  volume =	 {31},
  number =	 {2},
  pages =	 {1230--1255},
  year =	 {2025},
  doi =          {10.3150/24-BEJ1768},
  publisher =	 {Bernoulli Society for Mathematical Statistics and
                  Probability}
}

@article{benjamini2010discovering,
  author  = {Benjamini, Yoav},
  title   = {Discovering the false discovery rate},
  journal = {Journal of the Royal Statistical Society: Series B (Statistical Methodology)},
  year    = {2010},
  volume  = {72},
  number  = {4},
  pages   = {405--416},
  doi     = {10.1111/j.1467-9868.2010.00746.x},
  url     = {https://doi.org/10.1111/j.1467-9868.2010.00746.x}
}

@article{hochberg1995extensions,
  title={Extensions of multiple testing procedures based on Simes' test},
  author={Hochberg, Yosef and Rom, Dror},
  journal={Journal of Statistical Planning and Inference},
  volume={48},
  number={2},
  pages={141--152},
  year={1995},
  publisher={Elsevier}
}

@article{samuel1996simes,
  title={Is the Simes improved Bonferroni procedure conservative?},
  author={Samuel-Cahn, Ester},
  journal={Biometrika},
  volume={83},
  number={4},
  pages={928--933},
  year={1996},
  publisher={Oxford University Press}
}

@phdthesis{roux2018inference,
  author = {Roux, Marine},
  title  = {Inf{\'e}rence de graphes par une proc{\'e}dure de test multiple avec application en Neuroimagerie},
  school = {Universit{\'e} Grenoble Alpes (ComUE)},
  year   = {2018},
  month  = sep,
  note   = {Thesis no. 2018GREAT058},
  url    = {https://theses.fr/2018GREAT058}
}

\appendix

\section{\texorpdfstring{Auxiliary lemmas}{Auxiliary lemmas}}
\label{app:envelope}

This appendix collects the properties of the envelope and Gaussian tails used to define and analyze $\nu_{q,r}$.

The following lemma characterizes the likelihood-ratio envelope.
\begin{lemma}
\label{lem:envelope}
For $0\le y\le1$, $M(y)=1$.  For $y>1$, $a\mapsto L_a(y)$ has a unique maximizer $a(y)>0$, characterized by
\[
a(y)=y\tanh\{a(y)y\}.
\]
The maps $a:(1,\infty)\to(0,\infty)$ and $M:(1,\infty)\to(1,\infty)$ are continuous and strictly increasing, with
\begin{equation}
\lim_{y\downarrow1}a(y)=0,
\quad
\lim_{y\to\infty}M(y)=\infty.
\label{eq:envelope-limits}
\end{equation}
The inverse $y=y(a)$ is continuously differentiable on $(0,\infty)$, and
\begin{equation}
y'(a)
=
\frac{1-y(a)^2\sech^2\{ay(a)\}}
{\tanh\{ay(a)\}+ay(a)\sech^2\{ay(a)\}}
>0.
\label{eq:yprime}
\end{equation}
Moreover, with $a_q=a(y_q)$, the deficit in the secondary survival mass satisfies
\begin{equation}
1-qM\{y(a)\}>0\quad(0<a<a_q),
\quad
1-qM\{y(a_q)\}=0,
\quad
\frac{\mathrm d}{\mathrm da}\bigl[1-qM\{y(a)\}\bigr]<0.
\label{eq:deficit-properties}
\end{equation}
\end{lemma}

\begin{proof}
Let
\[
g_y(a)=\log L_a(y)
=-\frac{a^2}{2}+\log\cosh(ay).
\]
Then
\begin{equation}
g_y'(a)=-a+y\tanh(ay).
\label{eq:gprime}
\end{equation}
If $y\le1$, then $y\tanh(ay)\le ay^2\le a$, so $M(y)=L_0(y)=1$.

For $y>1$, put $t=ay$.  Equation $g_y'(a)=0$ is equivalent to
\begin{equation}
y^2=\frac{t}{\tanh t}.
\label{eq:t-param}
\end{equation}
{
To prove that the map $t\mapsto t/\tanh t$ in
\eqref{eq:t-param} is strictly increasing, let
$H(t)=\tanh t-t\sech^2t$.  Then $H(0)=0$ and
$H'(t)=2t\sech^2t\tanh t>0$ for $t>0$.  Hence
\[
\frac{\dd}{\dd t}
\left(\frac{t}{\tanh t}\right)
=
\frac{\tanh t-t\sech^2t}{\tanh^2t}>0.
\]
Moreover,
$t/\tanh t\to1$ as $t\downarrow0$ and
$t/\tanh t\to\infty$ as $t\to\infty$.  Because $y^2>1$,
\eqref{eq:t-param} in Lemma~\ref{lem:envelope} therefore has a unique
positive solution.  The sign of
$g_y'$ changes from positive to negative there, so the solution is the
unique maximizer.  Moreover,}
\begin{equation}
a(y)=\sqrt{t\tanh t},
\quad
y=\sqrt{\frac{t}{\tanh t}},
\label{eq:a-param}
\end{equation}
{
Both parameter functions in \eqref{eq:a-param} in
Lemma~\ref{lem:envelope} are smooth and strictly
increasing in $t>0$.  They tend respectively to zero and one as
$t\downarrow0$, and both tend to infinity as $t\to\infty$.  Thus
$a(y)$ is continuous and strictly increasing on $(1,\infty)$, with
$a(y)\to0$ as $y\downarrow1$; the inverse function theorem applied to
the second parameter function also shows that $a(y)$ is continuously
differentiable.}
Since $M(y)=L_{a(y)}(y)$, the chain rule and $\partial_aL_{a(y)}(y)=0$ give
\[
M'(y)
=
M(y)a(y)\tanh\{a(y)y\}>0.
\]
Thus $M$ is continuous and strictly increasing.  Also
\[
M(y)\ge L_y(y)\ge\frac12e^{y^2/2}\to\infty.
\]
{
Finally, apply the implicit function theorem to
\[
F(a,y)=a-y\tanh(ay)=0.
\]
At the maximizer, with $t=ay>0$, equation~\eqref{eq:t-param} in
Lemma~\ref{lem:envelope} gives
\[
 y^2\sech^2(ay)
 =\frac{t\sech^2t}{\tanh t}
 =\frac{2t}{\sinh(2t)}<1.
\]
Consequently,
\[
F_a=1-y^2\sech^2(ay)>0,
\qquad
F_y=-\tanh(ay)-ay\sech^2(ay)<0.
\]
Thus \eqref{eq:yprime} in Lemma~\ref{lem:envelope} follows from
$y'=-F_a/F_y$.}

It remains to prove \eqref{eq:deficit-properties} in
Lemma~\ref{lem:envelope}.  Since $a_q=a(y_q)$, we have
$y(a_q)=y_q$, and therefore $1-qM\{y(a_q)\}=0$ by \eqref{eq:yq}.  If
$0<a<a_q$, then $y(a)<y_q$; since $M$ is strictly increasing on
$(1,\infty)$, $qM\{y(a)\}<qM(y_q)=1$.  Finally,
\[
\frac{\mathrm d}{\mathrm da}\bigl[1-qM\{y(a)\}\bigr]
=-qM'\{y(a)\}y'(a)<0,
\]
because $M'>0$ on $(1,\infty)$ and $y'>0$ on $(0,\infty)$.
\end{proof}

The final lemma collects the Mills bounds and quantile asymptotics used to control Gaussian tails.
\begin{lemma}
\label{lem:mills}
{
The normal hazard $\lambda$ is increasing on $\R$ and satisfies
\begin{equation}
x<\lambda(x)<x+x^{-1}\quad(x>0).
\label{eq:mills-hazard}
\end{equation}
Moreover, as $x\to\infty$,
\begin{equation}
\lambda(x)=x+x^{-1}+O(x^{-3}).
\label{eq:mills-hazard-expansion}
\end{equation}
}
As $x\to\infty$,
\begin{equation}
\Phib(x)=\frac{\phi(x)}x\{1+o(1)\}.
\label{eq:mills-asymp}
\end{equation}
For every fixed $0<\sigma<1$, as $x\to\infty$,
\begin{equation}
\frac{\Phib(\sigma x)}{\Phib(x)}
\sim\frac1\sigma
\exp\left\{\frac{(1-\sigma^2)x^2}{2}\right\}.
\label{eq:mills-tail-ratio-asymp}
\end{equation}
For every $0<H<\infty$, uniformly over
$|u|\vee|v|\le H$, as $x\to\infty$,
\begin{equation}
\frac{p(x+u)}{p(x+v)}
=
\frac{x+v}{x+u}
\exp\left\{-x(u-v)-\frac{u^2-v^2}{2}\right\}
\{1+o(1)\}.
\label{eq:mills-local-ratio}
\end{equation}
If $c>0$ is fixed, then, as $q\downarrow0$,
\begin{equation}
p^{-1}(cq)\sim\sqrt{2\log(1/q)}.
\label{eq:mills-quantile}
\end{equation}
In particular, for every fixed $c_1,c_2>0$,
\[
\frac{p^{-1}(c_1q)}{p^{-1}(c_2q)}\longrightarrow1
\quad(q\downarrow0).
\]
Finally, for $x\ge0$,
\begin{equation}
2\Phib(x)\le e^{-x^2/2}.
\label{eq:mills-two-sided-exponential-tail}
\end{equation}
For $x>0$ and $\delta\ge0$,
\begin{equation}
\frac{p(x+\delta)}{p(x)}
\le
\left(1+x^{-2}\right)
\exp\left\{-x\delta-\frac{\delta^2}{2}\right\}.
\label{eq:mills-shift-upper}
\end{equation}
Consequently, if $x\to\infty$, then the following tail-ratio limits hold.
\begin{enumerate}[label=\textup{(\roman*)},leftmargin=2.2em]
\item For every $0<H<\infty$, uniformly over
$|u|\vee|v|\le H$,
\begin{equation}
\log\frac{p(x+u)}{p(x+v)}
=-x(u-v)-\frac{u^2-v^2}{2}+o(1).
\label{eq:general-ratio}
\end{equation}
\item For every $0<H<\infty$, uniformly over $|h|\le H$,
\begin{equation}
\frac{p(x)}{p(x+h/x)}\longrightarrow e^h.
\label{eq:h-over-x}
\end{equation}
\item If $0<\delta_0\le\delta\le\delta_1<\infty$, then
\begin{equation}
\frac{p(x+\delta)}{p(x)}\longrightarrow0
\label{eq:fixed-shift}
\end{equation}
uniformly in $\delta$.
\item If $\delta_x\ge0$ and $x\delta_x\to\infty$, then
\begin{equation}
\frac{p(x+\delta_x)}{p(x)}\longrightarrow0.
\label{eq:variable-shift}
\end{equation}
\end{enumerate}
\end{lemma}

\begin{proof}
For $x>0$, the upper Mills bound follows from
\begin{equation}
\Phib(x)<\frac1x\int_x^\infty t\phi(t)\,\dd t
=\frac{\phi(x)}x.
\label{eq:mills-upper}
\end{equation}
Integration by parts gives
\[
\Phib(x)=\frac{\phi(x)}x-
\int_x^\infty\frac{\phi(t)}{t^2}\,\dd t
>\frac{\phi(x)}x-\frac{\Phib(x)}{x^2},
\]
which proves the lower Mills bound
\begin{equation}
\frac{x}{1+x^2}\phi(x)<\Phib(x).
\label{eq:mills-lower}
\end{equation}
{
A further integration by parts gives
\[
\Phib(x)
=\frac{\phi(x)}x-\frac{\phi(x)}{x^3}
+3\int_x^\infty\frac{\phi(t)}{t^4}\dd t,
\]
where
\[
0<\int_x^\infty\frac{\phi(t)}{t^4}\dd t
<\frac{\phi(x)}{x^5}.
\]
Hence
\[
\Phib(x)
=\frac{\phi(x)}x\{1-x^{-2}+O(x^{-4})\},
\]
and inversion proves \eqref{eq:mills-hazard-expansion} in
Lemma~\ref{lem:mills}.}
Equations~\eqref{eq:mills-upper} and~\eqref{eq:mills-lower} imply
\eqref{eq:mills-hazard} in Lemma~\ref{lem:mills} for $x>0$.
{
Since $\lambda(x)-x>0$ also holds trivially for $x\le0$, the identity
\[
\lambda'(x)=\lambda(x)\{\lambda(x)-x\}
\]
shows that $\lambda$ is increasing on $\R$.}

Equations~\eqref{eq:mills-upper} and~\eqref{eq:mills-lower} give \eqref{eq:mills-asymp} in Lemma~\ref{lem:mills}.  Applying this at $x$ and $\sigma x$ proves \eqref{eq:mills-tail-ratio-asymp} in Lemma~\ref{lem:mills}.  Applying \eqref{eq:mills-asymp} in Lemma~\ref{lem:mills} at $x+u$ and $x+v$ gives \eqref{eq:mills-local-ratio} in Lemma~\ref{lem:mills}, for every $0<H<\infty$, uniformly over $|u|\vee|v|\le H$.  Moreover, as $x\to\infty$,
\[
-\log p(x)=\frac{x^2}{2}+O(\log x),
\]
which yields \eqref{eq:mills-quantile} in Lemma~\ref{lem:mills} by inversion.  To prove $2\Phib(x)\le e^{-x^2/2}$, set
\[
d(x)=e^{-x^2/2}-2\Phib(x).
\]
Then $d(0)=0$, $\lim_{x\to\infty}d(x)=0$, and
\[
d'(x)=e^{-x^2/2}\left(\sqrt{\frac2\pi}-x\right).
\]
Thus $d$ first increases and then decreases to zero, so $d(x)\ge0$ for all $x\ge0$.
Finally, \eqref{eq:mills-upper} and~\eqref{eq:mills-lower} imply
\[
\frac{p(x+\delta)}{p(x)}
\le
\frac{1+x^2}{x(x+\delta)}
\exp\left\{-x\delta-\frac{\delta^2}{2}\right\}
\le
\left(1+x^{-2}\right)
\exp\left\{-x\delta-\frac{\delta^2}{2}\right\},
\]
which proves \eqref{eq:mills-shift-upper} in Lemma~\ref{lem:mills}.

It remains to prove the four stated tail-ratio consequences.  Taking logarithms in \eqref{eq:mills-local-ratio} in Lemma~\ref{lem:mills} gives \eqref{eq:general-ratio} in Lemma~\ref{lem:mills}, because
\[
\log\frac{x+v}{x+u}=o(1)
\]
for every $0<H<\infty$, uniformly over
$|u|\vee|v|\le H$.  For \eqref{eq:h-over-x} in
Lemma~\ref{lem:mills}, apply \eqref{eq:mills-local-ratio} in
Lemma~\ref{lem:mills} with $u=0$ and $v=h/x$.
For every $0<H<\infty$, uniformly over $|h|\le H$,
\[
\frac{p(x)}{p(x+h/x)}
=
\frac{x+h/x}{x}
\exp\left\{h+\frac{h^2}{2x^2}\right\}
\{1+o(1)\}
\longrightarrow e^h.
\]
For \eqref{eq:fixed-shift} in Lemma~\ref{lem:mills}, apply \eqref{eq:general-ratio} in Lemma~\ref{lem:mills} with $u=\delta$ and $v=0$.  If $0<\delta_0\le\delta\le\delta_1<\infty$, then
\[
\log\frac{p(x+\delta)}{p(x)}
\le
-x\delta_0+O(1)
\longrightarrow-\infty
\]
uniformly in $\delta$.  Finally, \eqref{eq:variable-shift} in Lemma~\ref{lem:mills} follows from \eqref{eq:mills-shift-upper} in Lemma~\ref{lem:mills} with $\delta=\delta_x$, because $x\delta_x\to\infty$.
\end{proof}

\end{document}